\documentclass[12 pt]{amsart}
\usepackage[margin=2.5cm]{geometry}
\usepackage{amsfonts,graphicx,enumerate}
\usepackage{amssymb}
\usepackage{amsmath}
\usepackage{latexsym}
\usepackage{mathrsfs}
\usepackage{amsthm}
\usepackage{amscd}

\usepackage{comment}
\usepackage{framed}
\usepackage{color}
\definecolor{shadecolor}{gray}{0.875}

\newtheorem{thrm}{Theorem}[section]
\newtheorem{lem}[thrm]{Lemma}
\newtheorem{cor}[thrm]{Corollary}
\newtheorem{prop}[thrm]{Proposition}
\newtheorem{conj}[thrm]{Conjecture}

\theoremstyle{definition}
\newtheorem{defn}[thrm]{Definition}

\newtheorem{exmple}[thrm]{Example}
\newtheorem{rmk}[thrm]{Remark}
\newtheorem{ques}[thrm]{Question}
\newtheorem{notn}[thrm]{Notation}
\newtheorem{constr}[thrm]{Construction}

\newtheorem{conv}[thrm]{Convention}
\newtheorem*{caut}{Caution}

\newenvironment{claim}
            {\par \bigskip \noindent \textbf{Claim:}}
            {$\Box$ \par \noindent}

\newcommand{\factor}[2]{\left. \raise 0pt\hbox{\ensuremath{#1}} \right/
        \hskip -3pt\raise 0pt\hbox{\ensuremath{#2}}}

\DeclareMathOperator{\Sym}{Sym}

\DeclareMathOperator{\Eff}{\overline{Eff}}

\DeclareMathOperator{\Nef}{Nef}

\DeclareMathOperator{\Mov}{\overline{Mov}}

\DeclareMathOperator{\Supp}{Supp}

\DeclareMathOperator{\Sing}{Sing}

\DeclareMathOperator{\vol}{vol}

\DeclareMathOperator{\mc}{mc}

\DeclareMathOperator{\mob}{mob}

\DeclareMathOperator{\pl}{PL}

\DeclareMathOperator{\bpf}{BPF}

\newcommand{\M}{\overline{M}_{0,7}^{S_{7}}}

\input xy
\xyoption{all}

\begin{document}

\author{Mihai Fulger}
\address{Department of Mathematics, Princeton University\\
Princeton, NJ \, \, 08544}
\address{Institute of Mathematics of the Romanian Academy, P. O. Box 1-764, RO-014700,
Bucharest, Romania}
\email{afulger@princeton.edu}
\title{Zariski decompositions of numerical cycle classes}
\author{Brian Lehmann}
\thanks{The second author is supported by NSF Award 1004363.}
\address{Department of Mathematics, Rice University \\
Houston, TX \, \, 77005}
\email{blehmann@rice.edu}

\begin{abstract}
We construct a Zariski decomposition for cycle classes of arbitrary codimension.  This decomposition is an analogue of well-known constructions for divisors.   Examples illustrate how Zariski decompositions of cycle classes reflect the geometry of the underlying space.  We also analyze the birational behavior of Zariski decompositions, leading to a Fujita approximation-type result for curve classes. 
\end{abstract}

\maketitle
\setcounter{tocdepth}{1}
\tableofcontents

\section{Introduction}

In \cite{zariski62} Zariski introduced a fundamental tool for understanding linear series on surfaces now known as the Zariski decomposition. This was extended to pseudo-effective divisors by \cite{fujita79}. Given a pseudo-effective divisor $L$ on a smooth surface, the Zariski decomposition is an expression $L = P + N$ where $P$ is a nef $\mathbb{Q}$-divisor and $N$ is an effective $\mathbb{Q}$-divisor that is ``rigid'' in a strong sense. 
The key property of the decomposition -- and Zariski's original motivation for studying it -- is that $P$ captures all the information about sections of $L$.  More precisely, the natural inclusion map yields an identification $H^{0}(X,\mathcal{O}_{X}(\lfloor mP \rfloor)) = H^{0}(X,\mathcal{O}_{X}(mL))$ for every positive integer $m$.

Zariski's construction is generalized to divisors on smooth varieties of arbitrary dimension by the $\sigma$-decomposition of \cite{nakayama04} (and also by many related decompositions %or constructions
due to other authors; %as in \cite{bfj09}, \cite{boucksom04}, \cite[\S2.3.E]{lazarsfeld04}. S
see \cite{pro04} %and the references therein
for a survey of these constructions). 
Just as for surfaces, the $\sigma$-decomposition of a pseudo-effective divisor $L$ is a sum 
$$L = P_{\sigma}(L) + N_{\sigma}(L)$$ 
where $P_{\sigma}(L)$ is the ``positive'' part containing all the information about sections of $L$ and $N_{\sigma}(L)$ is the rigid ``negative'' part.  One key distinction from the surface case is that it is no longer possible to require the positive part to lie in the nef cone; rather, $P_{\sigma}(L)$ lies in the movable cone generated by divisors whose base locus has codimension at least two.

Our goal is to develop a decomposition theory for numerical classes of cycles of any codimension.  To define the decomposition, we first establish some preliminaries on numerical classes of cycles.  Let $X$ be a projective variety over an algebraically closed field.  We let $N_{k}(X)$ denote the $\mathbb{R}$-vector space of $k$-cycles with $\mathbb{R}$-coefficients modulo numerical equivalence. The \emph{pseudo-effective cone} $$\Eff_{k}(X) \subset N_{k}(X)$$ is the closure of the cone generated by classes of effective cycles; it is full-dimensional and does not contain any nonzero linear subspaces. It was previously studied in this arbitrary dimension context for example in \cite{delv11}, and \cite{fulger11}. Of particular interest are classes in the interior of $\Eff_{k}(X)$, called \emph{big classes}.  

\begin{defn}
Let $X$ be a projective variety.  The movable cone $\Mov_{k}(X) \subset N_{k}(X)$ is the closure of the cone generated by members of irreducible families of $k$-cycles which dominate $X$.  %effective cycles that deform in irreducible families which cover $X$.
\end{defn}

\noindent We show that the movable cone of cycles satisfies a number of natural geometric properties;
it is full-dimensional and does not contain nonzero linear
subspaces. For any ample Cartier divisor $H$ on $X$, the class of $H^{n-k}$ lies in the interior of $\Mov_{k}(X)$.

By analogy with the $\sigma$-decomposition, our Zariski decomposition of a pseudo-effective class $\alpha \in N_{k}(X)$ is a decomposition $\alpha = P(\alpha) + N(\alpha)$ such that the positive part $P(\alpha)$ is movable and retains all of the ``positivity'' of $\alpha$.  Since we may no longer apply the theory of linear systems, we describe the ``positivity'' of a class in a different way.  We use the mobility function, an analogue of the volume function for divisors (\cite[\S 2.2.C]{lazarsfeld04}) that is defined for arbitrary cycle classes.  

\begin{defn}[\cite{lehmann13}]
Let $X$ be a projective variety of dimension $n$ and let $\alpha \in N_{k}(X)$ be a class with integer coefficients.  The mobility of $\alpha$ is
\begin{equation*}
\mob(\alpha):= \limsup_{m \to \infty} \frac{\max \left\{ b \in \mathbb{Z}_{\geq 0} \, \left| \, 
\begin{array}{c} 
\textrm{For any }b \textrm{ general points on }X\textrm{ there is an}\\
\textrm{effective }\mathbb{Z}\textrm{-cycle of class } m\alpha \textrm{ containing them}
\end{array} \right. \right\}}{m^{\frac{n}{n-k}}/n!}
\end{equation*}
\end{defn}

\noindent The theory of mobility is developed by the second author in \cite{lehmann13}, building on \cite[Conjecture 6.5]{delv11}. He shows that the mobility extends to a homogeneous continuous function on all of $N_{k}(X)$.  In particular, the mobility is positive precisely for big classes $\alpha$.

\begin{defn} \label{introzardecomdef}
Let $X$ be a projective variety.  A Zariski decomposition for a big class $\alpha$ is a sum $\alpha = P(\alpha) + N(\alpha)$ where $P(\alpha)$ is movable, $N(\alpha)$ is pseudo-effective, and $\mob(P(\alpha)) = \mob(\alpha)$.
\end{defn}

\noindent Note that (via the mobility) the positive part captures all of the ``positivity'' of $\alpha$.  Conjecture \ref{negativepartispushedforward} predicts that the negative part satisfies a ``rigidity'' property analogous to the divisor case.

The definition of a Zariski decomposition extends to all pseudo-effective classes by taking limits.  Note that in contrast to the divisor case, we define the Zariski decomposition for numerical classes (and not  for cycles).  For divisors this distinction is unimportant; the negative part $N_{\sigma}(L)$ is uniquely determined by the numerical class of $L$ as a divisor.

\begin{exmple}
Suppose that $\alpha$ is a big divisor class on a smooth projective variety.  Proposition \ref{divisorzardecomprop} shows that Definition \ref{introzardecomdef} is the numerical generalization of the $\sigma$-decomposition.  Note however that our definition also makes sense on arbitrarily singular projective varieties.
\end{exmple}

\begin{exmple}
Suppose that $X$ is a smooth variety.  \cite[0.2 Theorem]{bdpp04} shows that the movable cone of curves on $X$ coincides with the nef cone of curves.    (Although \cite{bdpp04} is only stated over $\mathbb{C}$, the result can be extended to arbitrary characteristic; see Section \ref{divisorcharpsection}.)  Thus, the Zariski decomposition of a pseudo-effective curve class $\alpha$ has nef positive part $P(\alpha)$, as in the classical setting of divisors on surfaces.
\end{exmple}

\begin{caut}
Our notation conflicts with established notation for divisors.  A $\sigma$-decomposition of a pseudo-effective divisor $L$ is usually said to be a Zariski decomposition if $P_{\sigma}(L)$ is nef, but we do not impose any such condition for our decomposition.  \cite{nakayama04} shows that even after pulling back to a birational model a divisor $L$ may not admit a Zariski decomposition in the traditional sense.  Our choice of notation reflects the conviction that the most natural decomposition is Definition \ref{introzardecomdef}; decompositions where the positive part is nef occur in special geometric circumstances.
\end{caut}

Our main theorem is:

\begin{thrm} \label{maintheorem}
Let $X$ be a projective variety and let $\alpha \in N_{k}(X)$ be a pseudo-effective class.  Then $\alpha$ admits a Zariski decomposition $\alpha = P(\alpha) + N(\alpha)$.
\end{thrm}

\noindent In Section \ref{zardecomsection} we show that Zariski decompositions for arbitrary cycles share many of the good geometric properties of the $\sigma$-decomposition of Nakayama.  The main unresolved question is whether the Zariski decomposition of a big class is unique; this subtle question is closely related to the properties of the derivative of the mobility function.  (Using our definition, the Zariski decomposition of a class on the pseudo-effective boundary need not be unique even for divisors.)  Another fundamental open question is Conjecture \ref{negativepartispushedforward}, which would offer geometric structure to the negative parts of decompositions.

\begin{rmk}
Several other possible definitions of Zariski decompositions for cycles have been proposed.  In particular, \cite{nakayama04} explains how the $\sigma$-decomposition for divisors can be generalized to cycles using similar arguments.  For cycles of codimension at least $2$, our definition is usually different from the $\sigma$-decomposition: $N(\alpha)$ can be much larger than the negative part of the $\sigma$-decomposition.  We compare these notions, as well as several others, in Section \ref{comparisonsection}.
\end{rmk}

\begin{rmk}
Let $X$ be a complex projective variety, and let $N'_{k}(X) \subset H_{2k}(X,\mathbb{R})$ denote the subspace generated by classes of algebraic cycles.  \cite[Remark 1.7]{lehmann13} discusses an analogue of the mobility function on $N'_{k}(X)$.  This function shares the properties of $\mob$: it is continuous and is positive precisely on the interior of the closed cone generated by effective cycles.  The analogue of Theorem \ref{maintheorem} in this setting is also true, as can be verified by the same arguments.
\end{rmk}

The following examples illustrate how the Zariski decomposition of cycle classes yields new insights into the geometry of varieties.

\begin{exmple}
Let $C$ be a smooth curve and let $E$ be a locally free sheaf of rank $n$ and degree $d$ on $C$.  Set $X=\mathbb P(E)$. In \cite{fulger11}, the first author computed the effective cones $\Eff_k(X)$ in terms of the numerical data in the Harder--Narasimhan filtration of $E$.  In Section \ref{ssec:projbundles} we show that the movable cones and Zariski decompositions are also determined by the Harder-Narasimhan filtration of $E$.  Let
\begin{equation*}
0 = E_{s} \subset E_{s-1} \subset \ldots \subset E_{1} \subset E_{0} = E
\end{equation*}
denote the Harder-Narasimhan filtration of $E$ and let $\mu_{j}$ denote the slope of the quotient $\factor{E_{j-1}}{E_{j}}$.  Each subbundle $E_{j}$ determines an inclusion $\mathbb{P}(\factor{E}{E_{j}}) \to \mathbb{P}(E)$.  The subvarieties $\mathbb{P}(\factor{E}{E_{j}})$ represent the ``most rigid'' subvarieties of $X$ in their respective dimensions.

For an integer $1 \leq k \leq n-1$, let $j$ denote the smallest index such that $\mathrm{rk}(E_{j}) \leq n-k$.  Then the discrepancy between $\Eff_{k}(X)$ and $\Mov_{k}(X)$ is measured by the difference in slope $\mu_{s} - \mu_{j}$.  A class $\alpha \in \Eff_{k}(X)$ has a Zariski decomposition whose negative part is supported on $\mathbb{P}(\factor{E}{E_{j}})$ and whose positive part removes this ``rigid'' piece.
\end{exmple}

\begin{exmple}
Consider the Hilbert scheme $\mathbb{P}^{2[2]}$ parametrizing $0$-dimensional subschemes of $\mathbb{P}^{2}$ of length $2$.   Zariski decompositions for classes of arbitrary codimension have a similar description as in the divisor case: the ``rigid'' classes are exactly those parametrizing non-reduced $0$-cycles.

More precisely, suppose that $Z$ is an effective cycle on $\mathbb{P}^{2[2]}$ that is not movable.  We construct a Zariski decomposition for $[Z]$ whose negative part is the class of an effective cycle parametrizing non-reduced degree $2$ subschemes.  The positive part of $[Z]$ lies in the closure of the cone generated by cycles parametrizing reduced subschemes.
\end{exmple}

\begin{exmple}
Moduli spaces of pointed curves carry a rich combinatorial structure.  While the geometry of divisors has been analyzed extensively, there are fewer results on higher codimension cycles.  The easiest case to consider is surfaces on $\overline{M}_{0,7}$.  Since the divisor theory is already very complicated, we will instead work with the simpler space $\M$.

We give a partial description of $\Eff_{2}(\M)$ and $\Mov_{2}(\M)$, allowing us to describe Zariski decompositions in certain regions of the pseudo-effective cone.  In particular, we show that exactly three of the four combinatorial types of vital surfaces can occur as the negative part of a Zariski decomposition.  To illustrate our methods, we also analyze the Zariski decompositions that are controlled by the birational morphism $\M \to \overline{M}_{0,\mathcal{A}}^{S_{7}}$ to the symmetrization of Hassett's weighted moduli space which assigns weight $\frac{1}{3}$ to each point.
\end{exmple}

\vskip.25cm

An important application of the $\sigma$-decomposition is the study of exceptional loci of morphisms.  In Section \ref{birmapsection} we use Zariski decompositions to analyze the behavior of arbitrary classes under a birational map.  (In \cite{fl14} we extended these techniques to more general morphisms.)

Suppose that $\pi: Y \to X$ is a birational morphism of smooth varieties.  Any big class $\alpha \in N_{k}(X)$ admits a numerical pull-back $\pi^{*}\alpha \in N_{k}(Y)$.  However, in contrast to the situation for divisors, the pull-back $\pi^{*}\alpha$ need not be pseudo-effective.  Nevertheless there is some big class $\beta \in N_{k}(Y)$ that pushes forward to $\alpha$. Instead of looking for a ``distinguished'' big preimage $\beta$, we look for a construction in the spirit of Zariski decompositions.
The following definition encodes a slightly different perspective on the problem.  We require the class $\beta \in N_{k}(Y)$ to be movable at the cost of allowing an inequality $\pi_{*}\beta \preceq \alpha$.

\begin{defn}
Let $\pi: Y \to X$ be a birational morphism of projective varieties and let $\alpha \in N_{k}(X)$ be a big class.  A movable transform for $\alpha$ is a class $\beta \in \Mov_{k}(Y)$ such that $\pi_{*}\beta \preceq \alpha$ and $\mob(\beta) = \mob(\alpha)$.
\end{defn}

\noindent We prove the existence of movable transforms for arbitrary birational morphisms.  We do not know if they are unique -- this question is closely related to the uniqueness of Zariski decompositions.  
Nevertheless, movable transforms suggest themselves as a useful construction to analyze the birational theory of cycles.

\begin{exmple}
Suppose that $\pi: Y \to X$ is a birational morphism of smooth projective varieties and $L$ is a big divisor on $X$.  Then $[P_{\sigma}(\pi^{*}L)]$ is the (unique) movable transform for $[L]$. See Example \ref{ex:movtransformdiv}.
\end{exmple}

\begin{exmple}
Suppose that $\pi: Y \to X$ is a birational morphism of smooth projective varieties and $\alpha$ is a big curve class.  \cite{bdpp04} shows that movability coincides with nefness for curve classes; in particular movability is preserved by pullback.  Proposition \ref{movpullbackforcurves} shows that there is some positive part $P(\alpha)$ for $\alpha$ such that the pullback $\pi^{*}P(\alpha)$ is a movable transform for $\alpha$.
\end{exmple}

\begin{exmple}
Suppose that $X$ is a projective variety over $\mathbb{C}$ admitting a transitive action by a group variety.  Proposition \ref{prop:homvarpullback} shows that for any birational map $\pi: Y \to X$ and for any big class $\alpha \in \Eff_{k}(X)$ the class $\pi^{*}\alpha$ is a movable transform for $\alpha$.
\end{exmple}

\begin{thrm}
Let $\pi: Y \to X$ be a birational morphism of projective varieties.  Then any big class $\alpha \in N_{k}(X)$ admits a movable transform on $Y$.
\end{thrm}

\noindent We produce movable transforms by taking an asymptotic limit of classes of strict transforms.   When $X$ is a smooth variety, a movable transform can also be constructed by showing that the classes $\pi^{*}\alpha + \gamma$ have a common positive part as we vary $\gamma$ over all sufficiently large  $\pi$-exceptional effective classes $\gamma$.  This ``stable'' positive part is then a movable transform.  %Note that in contrast to the case of divisors, this stabilization result does not follow from the Hodge index theorem; one really must apply Zariski decomposition techniques.

Another important application of the $\sigma$-decomposition is the study of Fujita approximations for divisors.  Similarly, we use the Zariski decomposition to establish a ``Fujita approximation'' result for curves.  A crucial property of the movable cone of divisors is that it coincides with the birational ample cone (that is, the closure of the cone generated by pushforwards of the ample cone on birational models).  To formulate an analogue for the movable cone of cycles, we first need to identify a ``positive cone'' to play the role of the ample cone of divisors.  We will use a notion from \cite{fl13}.  (While the results of \cite{fl13} motivate the construction, we will in fact not need to use any of properties of this cone except for those proved in this paper.)

\begin{defn}
Let $X$ be a projective variety of dimension $n$.  We say that $\alpha \in N_{k}(X)$ is a strongly basepoint free class if there is:
\begin{itemize}
\item an equidimensional quasi-projective scheme $U$ of finite type over $K$,
\item a flat morphism $s: U \to X$,
\item and a proper morphism $p:U \to W$ of relative dimension $k$ to a quasi-projective variety $W$ such that each component of $U$ surjects onto $W$
\end{itemize}
such that
\begin{equation*}
\alpha= (s|_{F_p})_{*}[F_{p}]
\end{equation*}
where $F_{p}$ is a general fiber of $p$.

The basepoint free cone $\bpf_{k}(X)$ is the closure of the cone generated by all strongly basepoint free classes on $X$.
\end{defn}

The terminology is motivated by the following property of strongly basepoint free classes: for every subvariety $V \subset X$ there is an effective cycle of class $\alpha$ that intersects $V$ in the expected dimension.  Thus basepoint free classes provide a good geometric analogue in arbitrary codimension of nef divisor classes.

\begin{thrm} \label{movisbirbpf}
Let $X$ be a projective variety.  Then
\begin{equation*}
\Mov_{k}(X) = \overline{ \sum_{\pi} \pi_{*}\bpf_{k}(Y) }
\end{equation*}
where $\pi: Y \to X$ varies over all birational models of $X$.
\end{thrm}

\begin{exmple}
Suppose $X$ is a smooth projective variety.  For divisors, the movable cone is the birational basepoint free (=nef) cone by \cite[Proposition 2.3]{boucksom04}.  For curves, the movable cone is the birational basepoint free cone by \cite[0.2 Theorem]{bdpp04}.  Theorem \ref{movisbirbpf} generalizes these results to cycle classes of arbitrary codimension.
\end{exmple}

It is then natural to wonder how well we can approximate the geometric properties of a movable class by a basepoint free class on a birational model.  For divisors, this question is answered by the Fujita approximation theorem of \cite{fujita94}.  By combining the theory of movable transforms with the orthogonality theorem of \cite{bdpp04}, we are able to prove a ``Fujita approximation''-type theorem for curve classes on smooth projective varieties over $\mathbb{C}$.  The two statements in the following theorem are essentially equivalent to each other.

\begin{thrm} \label{introfujapprox}
Let $X$ be a smooth projective variety over $\mathbb{C}$ and let $\alpha \in \Eff_{1}(X)$ be a big curve class.
\begin{itemize}
\item For any $\epsilon > 0$, there is a birational morphism $\pi: Y \to X$ and an element $\beta \in \bpf_{1}(Y)$ such that $\pi_{*}\beta \preceq \alpha$ and $\mob(\beta)>\mob(\alpha) - \epsilon$.
\item Fix $\xi$ in the interior of $\bpf_{1}(X)$.  For any $\epsilon>0$, there is a birational model $\pi: Y \to X$, a movable transform $P$ of $\alpha$, and an element $\beta \in \bpf_{1}(Y)$ such that $P \preceq \beta + \epsilon \pi^{*} \xi$.
\end{itemize}
\end{thrm}

\begin{rmk}
In fact, the Fujita approximations $\beta$ in Theorem \ref{introfujapprox} are sums of complete intersections of ample divisors on $Y$, so they are ``positive'' in a very strong sense.
\end{rmk}

\subsection{Organization} Section \ref{backgroundsection} gives an overview of cycles and numerical equivalence.  It also contains a discussion on how to generalize several results known for divisor of curve classes over $\mathbb{C}$ to arbitrary algebraically closed fields.  Section \ref{sec:positive cones} defines the movable cone and explores its basic properties.  Section \ref{mobilitysection} describes the theory of the mobility function for $\mathbb{R}$-classes.  Zariski decompositions are defined and studied in Section \ref{zardecomsection}.  We then apply Zariski decompositions to analyze birational maps in Section \ref{birmapsection}.  Section \ref{examplesection} is devoted to examples.  Finally, Section \ref{comparisonsection} compares our definition of a Zariski decomposition with several other possibilities existent in the literature.

\subsection{Acknowledgments:} We are indebted to the anonymous referee for his many comments and suggestions. The authors thank Brendan Hassett for helpful discussions concerning Section \ref{mbarsubsection}.

\section{Background and preliminaries} \label{backgroundsection}
Throughout we will work over an algebraically closed ground field $K$.  A variety is an irreducible reduced separated scheme of finite type over the base field.  Projective bundles and Grassmannians parametrize quotients (not subspaces) unless explicitly stated otherwise.  We will frequently use the following important result of \cite{dejong96}: for any projective variety over $K$, there is an alteration $\pi: Y \to X$ from a smooth projective variety $Y$ over $K$. A closed convex cone inside a real vector space is called \emph{strictly convex}, if
it does not contain nonzero linear subspaces. The literature also calls them pointed or salient. 

\subsection{Cycles and dual cycles}
\label{ssec:cycles}

A \textit{cycle} on a projective variety $X$ is a finite formal linear combination 
$Z=\sum_i a_iV_i$ of closed subvarieties of $X$. We use the denominations \textit{integral}, \textit{rational}, or \textit{real} when the coefficients lie in $\mathbb Z$, $\mathbb Q$, or $\mathbb R$ respectively. When all $V_i$ have dimension $k$, we say that $Z$ is a \textit{$k$-cycle}. When for all $i$ we have $a_i\geq 0$, we say that the cycle is \textit{effective}.

To any closed subscheme $V\subset X$ we associate its \textit{fundamental} $\mathbb{Z}$-cycle
as follows: Let $V_i$ be the irreducible components of $V$ and let $\eta_i$ be their generic points. The fundamental cycle of $V$ is the linear combination of the supports of the $V_i$ with coefficients given by the lengths of the Artinian rings $\mathcal O_{V,\eta_i}$.

The group of $k$-cycles with $\mathbb{Z}$-coefficients is denoted $Z_k(X)$ and the Chow group of $k$-cycles with $\mathbb{Z}$-coefficients up to rational equivalence is denoted $A_{k}(X)$. 
\cite[Chapter 19]{fulton84} defines a $k$-$\mathbb{Z}$-cycle $Z$ to be \textit{numerically trivial} if
\begin{equation} \label{eq:numtriv}
\deg(P(E_I)\cap Z)=0\end{equation} 
for any weight $k$ homogeneous polynomial $P(E_I)$ in Chern classes of a finite set of vector bundles $E_I$ on $X$. Here $\deg:A_0(X)\to\mathbb Z$ is the group morphism that sends any point to 1, and $P(E_I)\cap Z\in A_0(X)$ is defined as in \cite[Chapter 3]{fulton84}. The resulting quotient $N_k(X)_{\mathbb Z}$ is a free abelian group of finite rank by \cite[Example 19.1.4]{fulton84}. It is a lattice inside $N_k(X)_{\mathbb Q}:=N_k(X)_{\mathbb Z}\otimes_{\mathbb Z}\mathbb Q$ and inside
$$N_k(X):=N_k(X)_{\mathbb Z}\otimes_{\mathbb Z}\mathbb R.$$
We call the latter the \textit{numerical group}. It is a finite dimensional real vector space. If $Z$ is a $k$-cycle with $\mathbb{R}$-coefficients, its class in $N_k(X)$ is denoted $[Z]$.

It is useful to consider the abstract duals $N^k(X)_{\mathbb Z}$, $N^k(X)_{\mathbb Q}$, and $N^k(X)$ of $N_k(X)_{\mathbb Z}$, $N_k(X)_{\mathbb Q}$, and $N_k(X)$ (with coefficients $\mathbb Z$, $\mathbb Q$, and $\mathbb R$ respectively).  We call $N^k(X)$ the \textit{numerical dual group}.  Any weighted homogeneous polynomial in Chern classes induces an element of $N^k(X)$.  Conversely, it follows formally from the definition that every $\beta\in N^k(X)$ is the equivalence class of a weighted homogeneous $\mathbb{R}$-polynomial in Chern classes of vector bundles.

\begin{rmk}\label{rmk:chowattributes}The quotient map $Z_k(X)\to N_k(X)_{\mathbb Z}$ factors through $A_k(X)$. 
Using \eqref{eq:numtriv} we deduce that many of the attributes of Chow groups descend to numerical groups with their natural grading:
\begin{itemize}

\item Proper pushforwards $\pi_*$ and dually proper pullbacks $\pi^*:=(\pi_*)^{\vee}$ for the groups $N^k(X)$. The flat pullback $\pi^*$ defined at the cycle level in \cite[\S1.7]{fulton84} descends to numerical equivalence if the target of $\pi$ is nonsingular.  (See \cite{fl13}.)

\item Chern classes for vector bundles operate on $N_k(X)$: For any $i$ and any vector bundle $E$ on $X$, the numerical class $[c_i(E)\cap Z]$ depends only on the numerical class of $Z$ for any cycle $Z$. We denote this class $c_i(E)\cap[Z]$. Since this operation is commutative and associative (see \cite[\S 3.2]{fulton84}), there is a 
natural way of defining $P(E_I)\cap [Z]$ for any finite collection $E_I$ of vector bundles on $X$ and any polynomial $P(E_I)$ of the Chern classes of each bundle in the collection.

\item The projection formula: If $\pi:Y\to X$ is a proper morphism, and $P(E_I)$ is a 
polynomial in the Chern classes of a finite set of vector bundles on $X$, then for any 
$\alpha\in N_k(X)$,
$$\pi_*(P(\pi^*E_I)\cap \alpha)=P(E_I)\cap \pi_*\alpha.$$
\end{itemize}

\noindent Let $n$ denote the dimension of $X$.  The association $[P]\to P\cap[X]$ induces a natural map 
\begin{equation*}%\label{eq:cyclification}
\varphi:N^{n-k}(X)\to N_k(X).\end{equation*} 
Its dual is the corresponding natural map $\varphi:N^k(X)\to N_{n-k}(X)$. Usually $\varphi$ is neither injective nor surjective; however, when $X$ is smooth $\varphi$ is an isomorphism by \cite[Example 15.2.16.(b)]{fulton84}.  We have similar statements for $\mathbb Q$-coefficients.
\end{rmk}

\begin{notn}Where there is little danger of confusion, we use $\cdot$ instead of $\cap$.
\end{notn}

\begin{conv} For the rest of the paper, the term cycle will always refer to a cycle with $\mathbb{R}$-coefficients unless otherwise qualified (and similarly for divisors and numerical classes).

When we discuss $k$-cycles on a projective variety we will implicitly assume that $0 \leq k < \dim X$.  Note that this range captures all the interesting behavior of cycles; some of our theorems will be nonsensical outside of this range.
\end{conv}

\subsubsection{The pseudo-effective cone}

\begin{defn}\label{def:eff}The closure of the convex cone in $N_{k}(X)$ generated by effective $k$-cycles on $X$ is denoted $\Eff_k(X)$. It is called the \textit{pseudo-effective} cone.  A class $\alpha\in N_k(X)$ is called \textit{pseudo-effective} (resp. \textit{big}) if it belongs to $\Eff_k(X)$ (resp. to its interior). For classes $\alpha,\beta \in N_{k}(X)$, we use the notation $\alpha \preceq \beta$ to denote that $\beta-\alpha$ is pseudo-effective.

We say that $\beta\in N^k(X)$ is \textit{pseudo-effective} if $\beta \cap [X] \in\Eff_{n-k}(X)$.  The pseudo-effective dual classes form a closed cone in $N^k(X)$ that we will denote $\Eff^k(X)$.
\end{defn}

The pseudo-effective cone is a full-dimensional strictly convex cone in $N_{k}(X)$.  
Pseudo-effectiveness is preserved by pushforward.  

\begin{caut}
Note that $\Eff^{k}(X) \subset N^{k}(X)$ may not be strictly convex when $k>1$ and $X$ is not smooth.
\end{caut}

\begin{prop}[\cite{fl13}]\label{prop:surjeff}
If $\pi:Y\to X$ is a surjective morphism of projective varieties, then
$\pi_*\Eff_{k}(Y)=\Eff_{k}(X)$.
\end{prop}

\subsubsection{The nef cone}

\begin{defn}The cone dual to $\Eff_k(X)$ in $N^k(X)$ is the \textit{nef} cone $\Nef^k(X)$.  Any element $\beta\in\Nef^k(X)$ is called \textit{nef}.
\end{defn}

\noindent The nef cone is a full-dimensional strictly convex cone in $N^{k}(X)$.  Nefness is preserved under pullback.  When $X$ is smooth, we will sometimes write $\Nef_{k}(X)$ for the image of $\Nef^{n-k}(X)$ under the isomorphism $\cap [X]$.

\begin{thrm}[\cite{fl13}] \label{cor:ci nef}
Let $X$ be a projective variety.  If $h_1,\ldots,h_k$ are ample divisor classes on $X$, then $h_1\cdot\ldots\cdot h_k$ is in the interior of $\Nef^k(X)$.
\end{thrm}

\noindent The idea of the proof is to consider Schur characteristic classes of
globally generated vector bundles on $X$. Adding and multiplying such
classes constructs the pliant cone $\pl^k(X)$ of \cite{fl13}, which is a full-dimensional subcone of $\Nef^k(X)$ with very good positivity features.
The behavior of Chern classes of vector bundles under direct sums
and tensor products then allows us to prove
that complete intersections lie in the interior of the pliant cone.

\begin{cor}[\cite{fl13}] \label{cor:norms} Let $X$ be a projective variety.  If $h$ is an ample divisor class on $X$, then for all $k$
there exists a norm $\|\cdot\|$ on $N_k(X)$ such that $\|\alpha\|=h^k\cdot\alpha$
for any $\alpha\in\Eff_k(X)$.
\end{cor}

\subsubsection{The basepoint free cone}

The basepoint free cone is an analogue of the nef cone of divisors.

\begin{defn}
Let $X$ be a projective variety of dimension $n$.  We say that $\alpha \in N_{k}(X)$ is a strongly basepoint free class if there is:
\begin{itemize}
\item an equidimensional quasi-projective scheme $U$ of finite type over $K$,
\item a flat morphism $s: U \to X$,
\item and a proper morphism $p:U \to W$ of relative dimension $k$ to a quasi-projective variety $W$ such that each component of $U$ surjects onto $W$
\end{itemize}
such that
\begin{equation*}
\alpha= (s|_{F_p})_{*}[F_{p}]
\end{equation*}
where $F_{p}$ is a general fiber of $p$.  Note that the resulting class is independent of the choice of fiber.

The basepoint free cone $\bpf_{k}(X)$ is defined to be the closure of the cone generated by strongly basepoint free classes.  If $X$ is smooth, then $\cap [X]$ is an isomorphism, we may also consider the cone $\bpf^{n-k}(X) \subset N^{n-k}(X)$.
\end{defn}

\begin{prop}[\cite{fl13}] Let $X$ be a smooth projective variety.  Then $\bpf_{k}(X)$ is a full-dimensional strictly convex cone.
\end{prop}

%\begin{lem}[\cite{fl13}] Let $\pi: Y \to X$ be a morphism of smooth projective varieties.  Then $\pi^{*}\bpf^{k}(X) \subset \bpf^{k}(Y)$.
%\end{lem}

\begin{rmk}
Just as with the ample cone of divisors, the basepoint free cone is most naturally considered as a  ``cohomological cone'' that is preserved under pullback.  This issue is discussed in more depth in \cite{fl13}.
\end{rmk}

\subsubsection{Families of cycles}
\label{sss:families}

\begin{defn}  \label{familydef}
Let $X$ be a projective variety.  A family of $k$-cycles (with $\mathbb{Z}$-coefficients) on $X$ consists of a variety $W$, a reduced closed subscheme $U \subset W \times X$, and an integer $a_{i}$ for each component $U_{i}$ of $U$ such that for each component $U_{i}$ of $U$ the first projection map $p: U_{i} \to W$ is flat dominant of relative dimension $k$.  We say that $p: U \to W$ is a family of effective $\mathbb{Z}$-cycles if each $a_{i} \geq 0$.  We will only consider families of cycles where each $a_{i} > 0$.
\end{defn}

We say that a family of effective $\mathbb{Z}$-cycles $p: U \to W$ is irreducible if $U$ has only one component.  Any irreducible component $U_{i}$ of a family $p: U \to W$ naturally yields an irreducible family $p_{i}: U_{i} \to W$ which assigns the coefficient $a_{i}$ to $U_{i}$ and removes all the other components of $U$.

For a closed point $w \in W$, the base change $w \times_{W} U_{i}$ is a $k$-dimensional subvariety of $X$ and thus defines a fundamental $k$-cycle $Z_{i}$; the cycle-theoretic fiber of $p: U \to W$ over $w$ is defined to be the cycle $\sum a_{i}Z_{i}$ on $X$.

Note that any two cycle-theoretic fibers of a family are algebraically equivalent.    We say that a family of $\mathbb{Z}$-cycles $p: U \to W$ represents a class $\alpha \in N_{k}(X)$ if $\alpha$ is the class of any cycle-theoretic fiber.

\begin{rmk} \label{closureoffamilies}
Given a family of $k$-$\mathbb{Z}$-cycles $p: U \to W$, we can extend $p$ to a projective closure of $W$ in the following way. Let $\overline{W}$ be any projective closure of $W$ and let $\overline{U} \subset \overline{W} \times X$ be the (reduced) closure of $U$.  Let $p': U' \to W'$ be a simultaneous flattening of the projection maps $\overline{p}_{i}: \overline{U}_{i} \to \overline{W}$ for the components $\overline{U}_{i}$ of $\overline{U}$.  Since there is a bijection between components of $U$ and components of $U'$, we can assign coefficients to the new family $p'$ in a natural way.
\end{rmk}

\begin{constr}[Strict transform families] \label{stricttransformconstr}
Let $X$ be a projective variety and let $p: U \to W$ be a family of effective $k$-$\mathbb{Z}$-cycles on $X$.  Suppose that $\phi: X \dashrightarrow Y$ is a birational map.  We define the strict transform family of effective $k$-$\mathbb{Z}$-cycles on $Y$ as follows.

First, modify $U$ by removing all irreducible components whose image in $X$ is contained in the locus where $\phi$ is not an isomorphism.  Then define the (reduced) closed subset $U'$ of $W \times Y$ by taking the strict transform of the remaining components of $U$.  Over an open subset $W^{0} \subset W$, the projection map $p': (U')^{0} \to W^{0}$ is flat on each component of $(U')^{0}$.  Since each component of $U'$ is the transform of a unique component of $U$, we can assign coefficients to the new family $p'$ in a natural way.
\end{constr}

\subsection{Birational theory of divisors} \label{divisorcharpsection}

In this section we review some important results in the theory of divisors and discuss how to establish them over an arbitrary algebraically closed field.  The key tool for extending to characteristic $p$ is the existence of Fujita approximations in arbitrary characteristic as established by \cite{takagi07} and \cite{lm09}.

\subsubsection{Zariski decompositions for divisors}
The $\sigma$-decomposition and its basic properties are described in \cite{nakayama04}.  \cite{mustata11} verifies that these constructions work in characteristic $p$ as well.

\begin{defn} \label{sigmagammadef}
Let $X$ be a smooth projective variety 
and let $L$ be a pseudo-effective divisor on $X$.  Fix an ample divisor $A$.  For any prime divisor $\Gamma$ on $X$ we define
\begin{equation*}
\sigma_{\Gamma}(L) = \lim_{\epsilon \to 0^{+}} \inf \{\mathrm{mult}_{\Gamma}(L') | L' \geq 0 \textrm{ and } [L'] = [L + \epsilon A] \}.
\end{equation*}
\end{defn}

\noindent Using this definition, $\sigma_{\Gamma}$ is clearly a numerical invariant.  It does not depend on the choice of $A$.

We verify briefly that Definition \ref{sigmagammadef} agrees with the definitions in \cite{nakayama04} and \cite{mustata11}.  Nakayama defines $\sigma_{\Gamma}(L)$ for a big divisor $L$ as the infimum of the multiplicity of $\Gamma$ in any effective divisor numerically equivalent for $L$, and then extends this to pseudo-effective divisors by taking a limit.  The equivalence with the definition above is stated explicitly as \cite[III.1.5 Lemma]{nakayama04}.   \cite{mustata11} defines $\sigma_{\Gamma}(L)$ for a big divisor $L$ by taking an asymptotic order of vanishing of $L$ along a general point in $\Gamma$, and then extends the definition to pseudo-effective divisors by taking a limit.  \cite[Theorem B]{mustata11} shows that, using this definition, $\sigma_{\Gamma}$ is a numerical invariant and defines a continuous function on the big cone, and the equivalence with Definition \ref{sigmagammadef} above follows immediately.

\cite[III.1.11 Corollary]{nakayama04} shows that for any pseudo-effective divisor $L$ there are only finitely many prime $\mathbb{Z}$-divisors $\Gamma$ on $X$ with $\sigma_{\Gamma}(L) > 0$, allowing us to make the following definition.

\begin{defn}[\cite{nakayama04} III.1.16 Definition] \label{sigmazardecom}
Let $X$ be a smooth projective variety and let $L$ be a pseudo-effective divisor on $X$.  Define
\begin{equation*}
N_{\sigma}(L) = \sum_{E \textrm{ prime}} \sigma_{E}(L) E \qquad \qquad P_{\sigma}(L) = L - N_{\sigma}(L)
\end{equation*}
The decomposition $L = P_{\sigma}(L) + N_{\sigma}(L)$ is called the $\sigma$-decomposition of $L$.
\end{defn}

\cite[III.1.10 Proposition]{nakayama04} shows that $N_{\sigma}(L)$ is the unique effective divisor in its numerical equivalence class.  The $\sigma$-decomposition has close ties to the diminished base locus.

\begin{defn} \label{nonnefdef}
Let $X$ be a smooth projective variety 
and let $L$ be a pseudo-effective divisor on $X$.  We define the $\mathbb{R}$-stable base locus of $L$ to be the subset of $X$ given by
\begin{equation*}
\mathbf{B}_{\mathbb{R}}(L) = \bigcap \{ \Supp(L') | L' \geq 0 \textrm{ and }L' \sim_{\mathbb{R}} L \}.
\end{equation*}
The diminished base locus of $L$ is
\begin{equation*}
\mathbf{B}_{-}(L) = \bigcup_{A \textrm{ ample divisor}}  \mathbf{B}_{\mathbb{R}}(L+A).
\end{equation*}
\end{defn}

\begin{prop}[\cite{nakayama04} V.1.3 Theorem and \cite{mustata11} Theorem C]
Let $X$ be a smooth projective variety and let $L$ be a pseudo-effective divisor.  The union of the codimension $1$ components of $\mathbf{B}_{-}(L)$ coincides with $\Supp(N_{\sigma}(L))$.
\end{prop}

\subsubsection{Positive product}

The positive product of \cite{bdpp04} is an important tool for understanding the positivity of divisor classes.  We will use the interpretation described by \cite{bfj09}.  \cite{cutkosky13} then extends this construction to arbitrary fields.

\begin{constr}[Characteristic $0$] \label{posprodchar0}
Let $X$ be a smooth projective variety over an algebraically closed field of characteristic $0$ and let $L_{1},\ldots,L_{k}$ be big $\mathbb{Q}$-divisors on $X$.  For any birational map $\phi: Y \to X$ from a normal variety $Y$ and for nef $\mathbb{Q}$-Cartier divisors $A_{i}$ on $Y$ such that $\phi^{*}L_{i} - A_{i}$ is numerically equivalent to an effective $\mathbb{Q}$-Cartier divisor, consider the class
\begin{equation*}
\gamma := \phi_{*}(A_{1} \cdot \ldots \cdot A_{k})
\end{equation*}
on $X$.  \cite{bfj09} shows that as we vary over all birational maps $\phi$ and choices of the $A_{i}$, the classes $\gamma$ form a directed set under the relation $\preceq$ and admit a unique maximum class under this relation.  The positive product of the $L_{i}$ is defined to be this maximal class and is denoted $\langle L_{1} \cdot \ldots \cdot L_{k} \rangle$.  By \cite[Proposition 2.9]{bfj09}, as a function on the $k$-fold product of the interior of $\Eff_{k}(X)$ the positive product is symmetric and continuous and is super-additive (with respect to $\preceq$) in each variable.  In particular it naturally extends to big $\mathbb{R}$-divisors.

When $L_{1},\ldots,L_{k}$ are only pseudo-effective, we define
\begin{equation*}
\langle L_{1} \cdot \ldots \cdot L_{k} \rangle = \lim_{t \to 0} \langle (L_{1} + tB) \cdot \ldots \cdot (L_{k} + tB) \rangle
\end{equation*}
where $B$ is any fixed big divisor; this definition is independent of the choice of $B$ as explained in \cite[Definition 2.10]{bfj09}.
\end{constr}

\begin{constr}[Characteristic $p$]
\cite{cutkosky13} extends the construction of \cite{bfj09} to fields of any characteristic.  However, Cutkosky works with a different notion of numerical equivalence.  Since his results do not directly apply in our situation, we will give a brief summary of how the positive product works for the spaces $N_{k}(X)$.

There are two potential difficulties in extending the results of \cite{bfj09} to characteristic $p$.  The first is the equality between the top positive product of a big divisor and its volume.  Over an arbitrary algebraically closed field this is a consequence of the existence of Fujita approximations as in \cite{takagi07} and \cite{lm09}.

The second is the lack of resolution of singularities.  In the first section of \cite{bfj09}, Lemma 1.2 and the comparison between the notion of pseudo-effectiveness for Weil and Cartier $\mathbf{b}$-divisors rely on resolution of singularities; however, these statements are not essential for the rest of the paper.  In the second and third sections of \cite{bfj09}, all the constructions only involve $\mathbb{Q}$-Cartier divisors so there is never an issue with $\mathbb{Q}$-factoriality.  (The key fact to check is \cite[Lemma 2.6]{bfj09} which proves that the $\gamma$ form a directed set.)   This yields the existence of the positive product and verifies that it satisfies the same basic properties as in the characteristic $0$ case.  See \cite{cutkosky13} for a detailed analysis of this argument in a slightly different situation.

In the fourth section we must be slightly more careful; all the material after \cite[Theorem 4.9]{bfj09} implicitly relies on the existence of embedded resolutions.  However, if we work on a fixed smooth variety $X$ (instead of just a normal variety), then any prime $\mathbb{Z}$-divisor $\mathcal{D}$ is already Cartier and so Theorem 4.9, Lemma 4.10, and Corollary 4.11 are valid.  We will only apply these results to prove Proposition \ref{divisorzardecomprop}, where we only consider smooth varieties.
\end{constr}

\begin{prop} \label{posprodprop}
Let $X$ be a smooth projective variety of dimension $n$ and let $L$ be a pseudo-effective divisor on $X$.
\begin{enumerate}
\item $\vol(L) = \langle L^{n} \rangle = \langle L^{n-1} \rangle \cdot P_{\sigma}(L)$ and $\langle L^{n-1} \rangle \cdot N_{\sigma}(L) = 0$.
\item $\langle L \rangle$ is the class of $P_{\sigma}(L)$.
\end{enumerate}
\end{prop}

\begin{proof}
The first equality in (1) is proved by \cite[Theorem 5.2]{cutkosky13}.  The others follow from \cite[Theorem 3.1]{bfj09} and \cite[Corollary 3.6]{bfj09} (which work equally well in any characteristic as discussed above).  

To show (2), first note that
\begin{equation*}
\Mov^{1}(X) = \overline{\cup_{\pi: Y \to X} \pi_{*}\Nef^{1}(Y)}
\end{equation*}
where $\pi$ varies over all birational morphisms from normal varieties.  Indeed, given any Cartier $\mathbb{Z}$-divisor $D$ whose base locus has codimension at least $2$, after resolving the base locus we can write $\pi^{*}D = M + E$ where $M$ is a basepoint free Cartier $\mathbb{Z}$-divisor and $E$ is effective and $\pi$-exceptional; the claim follows.  Note also that if $[D] \preceq [L]$ then $[M] \preceq [\pi^{*}L]$.  Using a perturbation argument, we see that $\langle L \rangle$ is a maximal element under the relation $\preceq$ over all movable classes $[D]$ such that $[D] \preceq [L]$.  We conclude by applying \cite[III.1.14 Proposition]{nakayama04} (which holds in any characteristic) which shows that $[P_{\sigma}(L)]$ is the unique class that satisfies this property.
\end{proof}

\subsubsection{The movable cone of curves.}

The following result is an important application of this circle of ideas.

\begin{thrm}[\cite{bdpp04}, Theorem 0.2]
Let $X$ be a smooth projective variety.  Then $\Mov_{1}(X) = \Nef_{1}(X)$.
\end{thrm}

\begin{proof}
There are several ingredients in the proof in \cite{bdpp04}:
\begin{itemize}
\item Fujita approximations.  These are constructed in arbitrary characteristic by \cite{takagi07} and \cite{lm09}.
\item Hodge inequalities.  As discussed in \cite[Remark 1.6.5]{lazarsfeld04} these hold true in arbitrary characteristic.
\item Continuity of the volume function.  The proof of \cite[Theorem 2.2.44]{lazarsfeld04} holds in arbitrary characteristic.
\end{itemize}
Given these facts, the proof in \cite{bdpp04} works over an arbitrary algebraically closed field with no changes.
\end{proof}

\section{The movable cone}
\label{sec:positive cones}

The movable cone of Cartier divisor classes is the closed cone generated by classes of $\mathbb{Z}$-divisors whose base locus has codimension at least two.  We introduce the analogous notion for $k$-cycles.

\begin{defn}We say that a $k$-$\mathbb{Z}$-cycle $Z$ on $X$ is \textit{strictly movable}
if it is a member of a family $p:U\to W$ of effective $k$-$\mathbb{Z}$-cycles on $X$ 
(in the sense of \S\ref{sss:families}) such that all the irreducible components of
$U$ dominate $X$. If $U$ is irreducible, we say that $Z$ is \textit{strongly movable}.

A class $\alpha\in N_k(X)_{\mathbb Z}$ is \textit{strictly movable} if it is
represented by a strictly movable $k$-$\mathbb{Z}$-cycle. The closed convex subcone of $N_k(X)$
generated by strictly movable classes is the \textit{movable} cone $\Mov_k(X)$.
Its elements are called \textit{movable}.
\end{defn}

\begin{rmk}\label{rmk:move away} A $k$-$\mathbb{Z}$-cycle $Z$ is strictly movable if and only if it is a member of an effective family of $\mathbb{Z}$-cycles $p:U\to W$ 
on $X$ such that for every proper closed subset $Y\subsetneq X$, the
support of a general cycle in $p$ has no components in $Y$.
(One implication is clear. For the other it is enough to assume that $U$ is irreducible and that $Y$ is a possibly reducible Cartier divisor. If $Y'$ is the pullback to $U$, then the general
fiber over $W$ is either empty, or has dimension $k-1$.)
\end{rmk}

\begin{caut}A strictly (resp. strongly) movable $\mathbb{Z}$-cycle $Z$ should always be presented together with a family (resp. irreducible family) $p:U\to W$ of $\mathbb{Z}$-cycles on $X$ and dominating $X$ that it sits in. At least in characteristic zero, one can extract such
a family from the Chow variety corresponding to the numerical class of $Z$ on $X$,
but not canonically. This is a distinguishing feature from Cartier divisors where one
constructs families canonically from complete linear series.
\end{caut}

\begin{exmple}For $k=1$ this definition agrees with the the cone of movable curves studied by \cite{bdpp04}.\end{exmple}

\begin{exmple}For any projective variety $X$ the map $\cdot [X]: N^{1}(X) \to N_{n-1}(X)$ is injective (see \cite{fl13}).  Thus one can ask whether the restriction of the movable cone $\Mov_{n-1}(X)$ to $N^{1}(X)$ coincides with the movable cone of Cartier divisors $\Mov^{1}(X)$.  Recall that the movable cone of Cartier divisors $\Mov^{1}(X)$ is defined to be the closure of the cone generated by the classes of $\mathbb{Z}$-divisors whose base locus has codimension at least two.  It is clear that $\Mov^{1}(X) \cdot [X] \subset \Mov_{n-1}(X) \cap N^{1}(X)$, and the problem is whether the converse holds.

Equality holds for smooth varieties $X$.  Indeed, the definition implies formally that the interior of $\Mov^{1}(X)$ consists exactly of those divisors such that the augmented base locus has codimension at least $2$.  Thus if $\alpha \not \in \Mov^{1}(X)$, there is some prime divisor $T$ and a positive constant $c$ such that $T$ occurs in every effective divisor of class $\alpha$ with a coefficient at least $c$.  This shows that $\alpha \not \in \Mov_{n-1}(X)$ as well.
\end{exmple}

\begin{rmk}\label{strict transform strictly movable}Let $\pi:Y\to X$ be a birational morphism, and let $p:U\to W$
be a family of $\mathbb{Z}$-cycles on $X$ such that each component of $U$ dominates $X$.
Let $U'\subset W\times Y$ be the strict transform cycle of $U\subset W\times X$ inducing a family $p':U'\to W$ after maybe shrinking $W$. Then the
general cycle in $p'$ is the strict transform of the corresponding cycle in $p$.
(We can assume that $U$ is irreducible, hence so is $U'$, and let $D\subset Y$ be an effective Cartier divisor containing the locus where $\pi$ is not an isomorphism in its support. Since $U'$ dominates $Y$, the intersection $Z:=U'\cap W\times D$ is a divisor on $U'$  (or is empty). Thus $Z$ has codimension at least $1$ in the general fiber of $p'$, yielding the conclusion.%The conclusion can only fail
%when the general fiber of $Z$ over $W$ has dimension $k$, which would lead to the contradiction $\dim Z=\dim U'$.)
\end{rmk}

\begin{lem} \label{lem: pullpush mov cone}
\begin{enumerate}
\item Let $\pi: Y \to X$ be a dominant morphism of projective varieties.  Then $\pi_{*}\Mov_{k}(Y) \subset \Mov_{k}(X)$.
\item Let $\pi: Y \to X$ be a flat morphism of projective varieties. Let $Z$ be a strictly movable $\mathbb{Z}$-cycle. Then $[\pi^{-1}Z]$ is also strictly movable. If additionally $X$ is nonsingular and $\pi$ has relative dimension $d$, then $\pi^{*}\Mov_{k}(X) \subset \Mov_{k+d}(Y)$.
\end{enumerate}
\end{lem}

\noindent Corollary \ref{cor:surj mov} proves  that in situation (1) there is actually an equality $\pi_{*}\Mov_{k}(Y) = \Mov_{k}(X)$.

\begin{proof}
(1) It is enough to prove that if $Z$ is a strongly movable $\mathbb{Z}$-cycle on $Y$ then $\pi_{*}[Z]$ is proportional to a strongly movable $\mathbb{Z}$-class on $X$. If $Z$ sits in the irreducible family $p:U\to W$ of $k$-$\mathbb{Z}$-cycles on $Y$, then the pushforward of the cycle $U$ under $\pi\times W$ is supported on an integral subscheme $U' \subset X \times W$.  Let $q: U' \to W$ denote the projection map; after shrinking $W$ to the flat locus of $q$, we can identify $q$ as an irreducible family of $k$-$\mathbb{Z}$-cycles on $X$. Since $U$ and $Y$ dominate $W$ and $X$ respectively, $U'$ dominates $X$ also. The push-forward of the general member of $p$ is proportional to a member of $q$.

(2) Note that $\pi$ is dominant since it is flat.  We can assume that $Z$ is strongly movable, and let $p:U\to W$ be an irreducible family of $\mathbb{Z}$-cycles
on $X$ deforming $Z$ and dominating $X$.  Every component of the base change $U \times_{X} Y$ dominates $Y$.  Let $U' = (U \times_{X} Y)_{red}$; after shrinking $W$, the induced morphism $U' \to W$ yields a family of $\mathbb{Z}$-cycles on $Y$ dominating $Y$, where to each component of $U'$ we assign the corresponding coefficient of the fundamental cycle of $U \times_{X} Y$.  By construction, a general member of this family is the flat pullback of a member of our original family $p$.

When $X$ is nonsingular, the pullback $\pi^*:N_k(X)\to N_{k+d}(Y)$ (where $d$ is the relative dimension of $\pi$) sends $[Z]$ to $\pi^*[Z]=[\pi^{-1}Z]$.
\end{proof}

\begin{exmple} If $X$ is a smooth projective variety of dimension $n$, then $\bpf_{k}(X) \subseteq \Mov_{k}(X)$. This follows from the fact that a general fiber of a morphism has movable class and the compatibility of movability with surjective pushforwards.
\end{exmple}

\begin{prop}
Let $X$ be a projective variety.  Then $\Mov_{k}(X)$ is a full-dimensional strictly convex cone in $N_{k}(X)$.
\end{prop}

\begin{proof}
Directly from the definition we see that $\Mov_k(X)\subset\Eff_k(X)$, hence $\Mov_k(X)$ is strictly convex.  

When $X$ is smooth of dimension $n$, we see that $\Mov_{k}(X)$ is full-dimensional since it contains the full-dimensional cone $\bpf_{k}$.  In the general case, let $\pi:Y\to X$ be a nonsingular alteration. Then $\Mov_k(X)$ contains the full-dimensional cone $\pi_*\Mov_{k}(Y)$.\end{proof}

Complete intersections are interior to the movable cone.

\begin{lem} Let $X$ be a projective variety of dimension $n$. If $h_1,\ldots,h_k$ are ample classes in $N^1(X)$, then $h_1\cdot\ldots\cdot h_k\cap[X]$ is in the interior of $\Mov_{n-k}(X)$.\end{lem}
\begin{proof}
We may assume that $0<k<n$. In particular $n\geq 2$.
When $X$ is nonsingular, then the result is true because by \cite{fl13} complete
intersections are in the interior of the basepoint free cone which is
is full-dimensional and a subcone of the movable cone.

In the general case, it suffices to prove that $h^{k}$ is in the interior of $\Mov_{n-k}(X)$ for an ample $\mathbb{Z}$-class $h$.  Indeed, if we choose an integer $m$ sufficiently large so that $mh_{i} - h$ is ample for every $i$, then $m^{k} h_{1} \cdot \ldots \cdot h_{k} = h^{k} + \beta$ for some $\beta \in \Mov_{n-k}(X)$.  Thus if $h^{k}$ is interior to $\Mov_{n-k}(X)$, so is $h_{1} \cdot \ldots \cdot h_{k}$.

There is an integral subvariety $Z$ that does not lie in $\Sing(X)$ whose numerical class lies in the interior of $\Mov_{n-k}(X)$.  To see this, let $\pi: Y \to X$ be a smooth alteration.  By the smooth case proved above, for an ample $\mathbb{Z}$-class $A$ on $Y$ we have that $A^{k}$ is in the interior of $\Mov_{n-k}(Y)$.  By Bertini's hyperplane theorem we can find an integral subvariety $Z'$ of $Y$ with class proportional to $A^{k}$ that is not contained in the preimage of $\Sing(X)$ or in the $\pi$-exceptional locus.  Since $\Mov_{n-k}(Y)$ is full-dimensional, so is $\pi_{*}\Mov_{n-k}(Y)$; thus $\pi_{*}[Z']$ lies in the interior of $\Mov_{n-k}(X)$.  Let $Z$ be the integral subvariety underlying $\pi(Z')$ and note that $[Z]$ is in the interior of $\Mov_{n-k}(X)$ since it is proportional to $\pi_{*}[Z']$.

Let $\pi:X'\to X$ be the blow-up of $Z$ with exceptional (Cartier) $\mathbb{Z}$-divisor $E$. Replacing $h$ by a multiple, we can assume that $\pi^*h-E$ is ample on $X'$. Then $(\pi^*h-E)^k\cap[X']\in\Mov_{n-k}(X')$, so $\pi_*((\pi^*h-E)^k\cap[X'])\in\Mov_{n-k}(X)$. 
\par Because $\pi$ contracts cycles of codimension less than $k$ contained in $E$, we have for $0<i<k$ that $\pi_*(E^i\cap[X'])=0$. We note that $$\pi_*((-E)^k\cap[X'])=-[Z].$$ This is the statement for Segre classes of normal cones that $s_0(Z,X)=[Z]$ (see \cite[\S3.1, \S4.3, and \S7.2]{fulton84}). Our computations and the projection formula now yield $$\Mov_{n-k}(X)\ni\pi_*((\pi^*h-E)^k\cap[X'])=h^k-[Z].$$ Since by assumption $[Z]$ is an interior class, the result follows.
\end{proof}

Except for dimension $1$ cycles, it seems difficult to describe the movable cone using intersections.  However, the following lemmas encapsulate the most important intersection-theoretic properties.

\begin{lem}\label{lem:movint} 
Let $\alpha\in\Mov_{k}(X)$. Then
\begin{enumerate}[i)]
\item For any $\beta\in\Eff^1(X)$ we have  $\beta \cdot \alpha \in \Eff_{k-1}(X)$. \label{lem:psefintmov}
\item If $H$ is a big and nef Cartier divisor, then $H\cdot \alpha =0$ if and only if $\alpha=0$. \label{cor:bigandnefcor}
\item If $h\in\Nef^1(X)$, then $h \cdot \alpha \in \Mov_{k-1}(X)$. \label{lem:movintnef} 
\end{enumerate}
\end{lem}

\begin{proof}
We first prove $i)$. Up to limits and linear combinations we may suppose that $\alpha$ is a strongly movable $\mathbb{Z}$-class defined by an irreducible family $p: U \to W$ that dominates $X$ and that $\beta$ is the class of an effective Cartier $\mathbb Z$-divisor $L$.  By Remark \ref{rmk:move away}, we can choose a cycle $Z$ in the family $p$ such that no component is contained in $\Supp(L)$. Then $L$ and $Z$ meet properly. By \cite[\S 7]{fulton84}, the intersection $L\cdot Z$ exists at the cycle level and is effective. It represents $\beta \cdot \alpha$.

For $ii)$, a big and nef divisor $H$ satisfies $[H] = [A] + [E]$ for an effective divisor $E$ and an ample divisor $A$.  By part $i)$, the intersection $\alpha \cdot E$ is pseudo-effective, so
\begin{equation*}
\alpha {\cdotp} H \succeq \alpha {\cdotp} A \succeq 0
\end{equation*}
by Theorem \ref{cor:ci nef}. By Corollary \ref{cor:norms} we see that $\alpha \cdot H=0$ if and only if $\alpha=0$.  

For $iii)$, by rescaling and taking limits we may assume that $\alpha$ is a strongly movable $\mathbb{Z}$-class defined by an irreducible family $p: U \to W$ that dominates $X$ and that $h$ is the class of a very ample $\mathbb{Z}$-divisor $H$. We construct a suitable family of $\mathbb{Z}$-cycles on $X$ by intersecting the general element of $p$ with the general element of the complete linear series $|H|$. We may assume $k>0$. 

Let $\mathbb P$ be 
the projective space of lines in $|H|$. Let $D$ be the universal hyperplane
in $\mathbb P\times X$. On $\mathbb P\times X\times W$, the Cartier $\mathbb{Z}$-divisor $D'=D\times W$ and the subvariety $U'=\mathbb P\times  U$ meet properly. This is because the general cycle in $p$ is not contained in the general element of $|H|$. There exists a well-defined effective cycle $T$ representing the intersection
$D'\cap U'$. It restricts generically over a point in $\mathbb P\times W$ to the cycle-theoretic intersection of the corresponding elements of $|H|$ and $p$ respectively.

By generic flatness, we find an open subset $S$ of $\mathbb P\times W$ such that over it $T$ is
a family $q:T\to S$ of $(k-1)$-$\mathbb{Z}$-cycles on $X$ with numerical class $H\cdot\alpha$.

We show that each component of the support of $T$ dominates $X$. It suffices to show that for any irreducible divisor $Y\subset X$, there is some cycle in $q$ with no component contained in $Y$.
Choose a general cycle $V$ in $p$; the dimension of the set-theoretic intersection of $V$ with $Y$ is $k-1$. 
By the freeness of $|H|$ a general element $A \in |H|$ does not contain any component of $V \cap Y$.  Then the cycle-theoretic intersection of $A$ with $V$ has no component contained in $Y$.  Since $A$ and $V$ are general, this cycle-theoretic intersection lies in the family $q$.
\end{proof}

We now return to the situation where we have a surjective morphism $\pi: Y \to X$.

\begin{prop} \label{prop:deg cover bound}
Let $\pi: Y \to X$ be a surjective morphism of projective varieties.  Let $A$ and $H$ be ample Cartier $\mathbb{Z}$-divisors on $Y$ and $X$ respectively.  There is some constant $C = C(\pi,A,H)$ such that for any effective strictly movable class $\alpha \in N_{k}(X)_{\mathbb{Z}}$, there is a movable effective
$\mathbb Q$-class $\alpha' \in N_{k}(Y)_{\mathbb{Q}}$ with $\pi_{*}\alpha' = \alpha$ satisfying
\begin{equation*}
A^k \cdotp \alpha' \leq C(H^k \cdotp \alpha).
\end{equation*}
\end{prop}

\begin{proof} Replacing $A$ and $H$ by fixed multiples,
we can assume that they are very ample $\mathbb{Z}$-divisors.  Furthermore, it suffices to consider the case when $\alpha$ is strongly movable.  Denote by $d$ the relative dimension of $\pi$.

\vskip.5cm
\textit{The flat nonsingular case.} Assume that $\pi$ is flat and that $X$ is nonsingular.  
Let $t$ be the positive integer such that 
$$\pi_*(A^d \cdotp [Y])=t[X].$$
Let $p: U \to W$ be a strongly movable family of class $\alpha$.  The flat pull-back construction in the proof of Lemma \ref{lem: pullpush mov cone} (2) yields a strictly movable family $q: U' \to W^{0}$ of class $\pi^{*}\alpha$.  The construction in the proof of Lemma \ref{lem:movint}.\ref{lem:movintnef} allows us to intersect the family $q$ by the linear series $|A|$; repeating this construction $d$ times yields a strictly movable family of effective $k$-$\mathbb{Z}$-cycles $r: T \to S$ of class $\beta := \pi^{*}\alpha \cdot A^{d}$.  By the projection formula, $\pi_{*}\beta = t \alpha$.

Note that $\pi^*:N_k(X)\to N_{k+d}(Y)$ is injective; this is because $\pi_*(\pi^*\alpha\cap A^d)=C\alpha$, where $C\neq 0$ is determined by
$\pi_*(A^d\cap[Y])=C\cdot [X]$. 

Fix the norms $\Vert \cdot \Vert_{1}$ on $N_{k+d}(Y)$ and $\Vert \cdot \Vert_{2}$ on $N_{k}(X)$ as in Corollary \ref{cor:norms} with respect to the ample divisors $A$ and $H$ respectively.  Note that
$$\|\pi^*\alpha\|_{1} = A^{k+d}\cap \pi^*\alpha = A^{k} \cdot \beta.$$
Restricting through the injective linear map $\pi^*$, we see $\|\cdot\|_{1}$ as a norm on $N_k(X)$ as well.
The existence of the constant $C=C(\pi,A,H)$ follows from the equivalence of this norm with $\Vert \cdot \Vert_{2}$ and the independence of $t$ from the construction.

\vskip.5cm
\textit{The generically finite dominant case.} There exists an effective Cartier $\mathbb{Z}$-divisor $E$ on $Y$ such that
$-E$ is $\pi$-ample. Replacing $H$ by a fixed multiple depending only on $\pi$ and $E$, we can assume that $\pi^*H-E$ is ample on $Y$. 
Using Corollary \ref{cor:norms}, the ample divisor classes $\pi^*H-E$ and $A$ determine equivalent norms on $N_k(Y)$, so without loss of generality we can assume that $A=\pi^*H-E$. 

Let $p: U \to W$ be a strongly movable family of class $\alpha$.  We construct a corresponding family on $Y$.   There exist blow-ups $f:Y'\to Y$ and $g:X'\to X$ and a (finite) flat morphism $\pi':Y'\to X'$ such that $\pi\circ f=g\circ \pi'$.
Let $U' \subset W \times X'$ be the strict transform of $U$.  After shrinking $W$ we can assume that $U'$ defines a strongly movable family of $\mathbb{Z}$-cycles $p': U' \to W$.  Let $Z'$ be a general cycle-theoretic fiber of $p'$.  Since $\pi'$ is flat, Lemma \ref{lem: pullpush mov cone} (2) shows that $(\pi')^{-1}Z'$ is strictly movable.  We set $\beta \in N_{k}(Y)$ to be the strictly movable class $f_{*}[(\pi')^{-1}(Z')]$. By \cite[Example 1.7.4]{fulton84}, we have
\begin{equation*}
\pi_{*}\beta = g_{*}\pi'_*[(\pi')^{-1}Z']=\deg\pi' \cdot g_{*}[Z'] = \deg\pi' \cdot \alpha.
\end{equation*}
Since $\beta$ is a strongly movable class and $\pi^{*}H-E$ is ample, there is an effective $\mathbb{Q}$-cycle $V$ of class $\beta \cdot (\pi^{*}H-E)^{k-1}$ such that no component of the support of $V$ is contained in $E$.  Thus
\begin{equation*}
\beta \cdot (\pi^{*}H - E)^{k} \leq \beta \cdot (\pi^{*}H - E)^{k-1} \cdot \pi^{*}H.
\end{equation*}
Repeating this argument inductively, we see that
\begin{equation*}
\beta \cdot (\pi^{*}H - E)^{k} \leq \beta \cdot \pi^{*}H^{k} = (\deg\pi)\cdot \alpha \cdot H^{k}.
\end{equation*}
The existence of the constant $C = C(\pi,A,H)$ follows.

\vskip.5cm
\textit{The general case.} There exist blow-ups $f:Y'\to Y$ and $g:X'\to X$ and a flat morphism $\pi':Y'\to X'$ such that $\pi\circ f=g\circ \pi'$.  After taking a generically finite base change to a nonsingular alteration of $X'$ we can assume that $X'$ is nonsingular.  Fix an ample $\mathbb{Z}$-divisor $H'$ on $X'$ and choose an ample $\mathbb{Z}$-divisor $A'$ on $Y'$ such that $A' - f^{*}A$ is ample.  By the projection formula, we may deduce the desired statement for $\pi: (Y,A) \to (X,H)$ from the statement for $\pi \circ f: (Y',A') \to (X,H)$.  However, this follows formally from the generically finite case $g: (X',H') \to (X,H)$ and the flat case $\pi': (Y',A') \to (X',H')$ over a nonsingular base.
\end{proof}

The following corollary is the analogue for the movable cone of Proposition \ref{prop:surjeff}.

\begin{cor} \label{cor:surj mov}
If $\pi:Y\to X$ is a surjective morphism of projective varieties, then
$\pi_*\Mov_k(Y)=\Mov_k(X)$.
\end{cor}

\begin{proof}Let $\alpha$ be a movable class on $X$. 
Write $\alpha$ as a limit of $\mathbb{Q}$-multiples $\alpha_{i}$ of strictly movable $\mathbb{Z}$-classes.
For each $i$, Proposition \ref{prop:deg cover bound} constructs
an effective movable class $\beta_i$ on $Y$ such that $\pi_*\beta_i=\alpha_i$
whose degree with respect to some polarization on $Y$ is bounded
independently of $i$.  Since the set of pseudo-effective classes of bounded degree is compact,
we can find an accumulation point $\beta$ for the 
sequence $\beta_i$.  Note that $\beta$ is movable.
Since $\pi_*$ is continuous, $\pi_*\beta=\alpha$.
\end{proof}

Finally, we show that the movable cone is the birational basepoint free cone.

\begin{prop} \label{movbirbpf}
Let $X$ be a projective variety.  Then
\begin{equation*}
\Mov_{k}(X) = \overline{ \sum_{\pi} \pi_{*}\bpf_{k}(Y) }
\end{equation*}
where $\pi: Y \to X$ varies over all birational models of $X$.
\end{prop}

\begin{proof}
It suffices to show that any strongly movable class is contained in the right hand side.  Suppose that $p: U \to W$ is a strongly movable family of $\mathbb{Z}$-cycles.  %By applying Construction \ref{closureoffamilies} we may assume that $U$ and $W$ are projective.
Flatten the map $s: U \to X$ to obtain $s': U' \to X'$.  Note that $U'$ is still irreducible since every component must dominate $X'$.  Although the map $p': U' \to W$ may not define a family of $\mathbb{Z}$-cycles, after shrinking $W$ to the flat locus of $p'$ we do obtain a family of $k$-$\mathbb{Z}$-cycles (by assigning $U'$ the same coefficient as $U$).  Furthermore the class of a cycle-theoretic fiber of $p'$ pushes forward to the class of a cycle-theoretic fiber of $p$.
\end{proof}

\section{Mobility for $\mathbb{R}$-classes} \label{mobilitysection}

Recall that if $X$ is a smooth variety of dimension $n$ and $L$ is a divisor on $X$, the volume of $L$ is
\begin{equation*}
\vol(L) := \limsup_{m \to \infty} \frac{\dim H^{0}(X,\mathcal{O}_{X}(\lfloor mL \rfloor))}{m^{n}/n!}.
\end{equation*}
The volume is an invariant of the numerical class of $L$ and is positive precisely when $[L]$ is big.  The mobility function is a generalization of the volume to arbitrary cycle classes.  We first define the mobility for classes $\alpha \in N_{k}(X)_{\mathbb{Z}}$.

\begin{defn} \label{mobilitydefn}
Let $X$ be a projective variety of dimension $n$ and let $p: U \to W$ be a family of effective $k$-$\mathbb{Z}$-cycles on $X$.  Denote the second projection map by $s: U \to X$.  The mobility count $\mc(p)$ of the family $p$ is the largest non-negative integer $b$ such that the morphism
\begin{equation*}
U \times_{W} U \times_{W} \ldots \times_{W} U \xrightarrow{s \times s \times \ldots \times s} X \times X \times \ldots \times X
\end{equation*}
is dominant, where we have $b$ terms in the product on each side.  (If the map is dominant for every positive integer $b$, we set $\mc(p) = \infty$; this can happen only when $p$ is a family of $n$-cycles.)

For a class $\alpha \in N_{k}(X)_{\mathbb{Z}}$, the mobility count $\mc(\alpha)$ is the maximum mobility count over all families of $\mathbb{Z}$-cycles $p: U \to W$ whose cycle-theoretic fibers have class $\alpha$.  We then define the mobility of $\alpha \in N_{k}(X)_{\mathbb{Z}}$ by the formula
\begin{equation*}
\mob_{X}(\alpha) = \limsup_{m \to \infty} \frac{\mc(m\alpha)}{m^{n/n-k}/n!}.
\end{equation*}
\end{defn}

\begin{exmple}
Suppose that $X$ is a smooth variety and that $L$ is a Cartier divisor on $X$.  Since general points impose codimension $1$ conditions on a linear series, the number of general points contained in elements of $|mL|$ is exactly $h^{0}(X,mL)-1$.  Thus the volume of $L$ can be interpreted as an asymptotic count of general points contained in members of $|mL|$.  (This interpretation is similar to Matsusaka's ``moving self-intersection number'' described in \cite[Definition 11.4.10]{lazarsfeld04}.)

Note that the mobility count of the family of divisors defined by $|mL|$ is essentially the same as the dimension of the space of sections.  The mobility count of the class $[mL]$ can be larger, since it accounts for non-linearly-equivalent divisors as well.  Nevertheless, since the set of all algebraically equivalent divisors is controlled by the (finite dimensional) Picard variety, these differences disappear in the asymptotic computation.  (See \cite[Example 6.16]{lehmann13} for more details.)
\end{exmple}

\cite[Lemma 6.15]{lehmann13} shows that the mobility function $\mob: N_{k}(X)_{\mathbb{Z}} \to \mathbb{R}$ is homogeneous of degree $\frac{n}{n-k}$ and thus can be extended naturally to $N_{k}(X)_{\mathbb{Q}}$.  The following theorem then extends the definition of mobility to arbitrary classes.

\begin{thrm}[\cite{lehmann13}, Theorem A] \label{continuousmobility}
Let $X$ be a projective variety.
\begin{itemize}
\item The mobility function on $N_{k}(X)_{\mathbb{Q}}$ can be extended to a continuous function $\mob: N_{k}(X) \to \mathbb{R}$.
\item The mobility is positive precisely on the big cone.
\end{itemize}
\end{thrm}

\begin{exmple}
Suppose that $X$ is a smooth projective variety and that $L$ is a Cartier divisor on $X$.  \cite[Example 6.16]{lehmann13} verifies that $\vol(L) = \mob([L])$.
\end{exmple}

An important conjecture about the mobility is:

\begin{conj} \label{concaveconj}
Let $X$ be a projective variety of dimension $n$.  Then $\mob$ is a log-concave function on $\Eff_{k}(X)$: for any classes $\alpha, \beta \in \Eff_{k}(X)$ we have
\begin{equation*}
\mob(\alpha+\beta)^{\frac{n-k}{n}} \geq \mob(\alpha)^{\frac{n-k}{n}} + \mob(\beta)^{\frac{n-k}{n}}.
\end{equation*}
\end{conj}

The following easy lemmas will be used often.

\begin{lem} \label{mobilityincreases}
Let $X$ be a projective variety and let $\alpha \in N_{k}(X)$.  If $\beta \in N_{k}(X)$ is pseudo-effective, then $\mob(\alpha + \beta) \geq \mob(\alpha)$.  If additionally $\alpha$ is pseudo-effective and $\beta$ is big then the inequality is strict.
\end{lem}

\begin{proof}
We first prove the inequality $\geq$.  Using the continuity of mobility we may assume that $\beta$ is in the big cone.  Furthermore, by continuity and homogeneity it suffices to consider the case when $\alpha$ and $\beta$ are $\mathbb{Z}$-classes.  By rescaling we may also assume that $\beta$ is represented by an effective $\mathbb{Z}$-cycle.  Then the inequality is proved by \cite[Lemma 6.17]{lehmann13}.

Suppose now that $\alpha$ is pseudo-effective and $\beta$ is big.  If $\mob(\alpha)=0$, then $\mob(\alpha+\beta)>\mob(\alpha)$ since $\alpha+\beta$ is big.  Otherwise, note that for sufficiently small $\epsilon$ we have $\alpha + \beta = (1+\epsilon)\alpha + (\beta - \epsilon \alpha)$ and $\beta - \epsilon \alpha$ is a big class.  Thus $\mob(\alpha + \beta) \geq \mob((1+\epsilon)\alpha) > \mob(\alpha)$ by homogeneity.
\end{proof}

\begin{lem} \label{addingcycle}
Let $X$ be a projective variety and let $\alpha \in N_{k}(X)_{\mathbb{Z}}$.  Let $Z$ be an effective $\mathbb{Z}$-cycle.  Then
\begin{equation*}
\mob(\alpha) = \limsup_{m \to \infty} \frac{\mc(m\alpha + [Z])}{m^{\frac{n}{n-k}}/n!}.
\end{equation*}
\end{lem}

\begin{proof}
We have $\mc(m\alpha) \leq \mc(m\alpha + [Z]) \leq \mc(m\alpha + k[Z])$ for any integer $k \geq 1$.  Thus for any positive integer $j$ the right hand side of the equation is bounded above by $\mob(\alpha + \frac{1}{j}[Z])$.  We conclude by the continuity of $\mob$.
\end{proof}

The proof of the analogous result for volumes of divisors does not use the continuity of volume, and is in fact part of its proof. However, it uses the restriction sequence and the induced long exact sequence in cohomology.
Currently these have no immediate equivalent for cycles of arbitrary codimension.

\subsection{Mobility of $\mathbb{R}$-classes}
We have defined the mobility of an $\mathbb{R}$-class by extending continuously from $N_{k}(X)_{\mathbb{Q}}$ but there is a direct approach that will be more useful.

\begin{defn}
Let $X$ be a projective variety.  An acceptable cone $\mathcal{K} \subset N_{k}(X)$ is a cone generated by the union of:
\begin{itemize}
\item a compact subset of $N_{k}(X)$ with non-empty interior and contained in the interior of $\Eff_{k}(X)$, and
\item a finite set of classes of effective $\mathbb{Z}$-cycles.
\end{itemize}
\end{defn}

\begin{defn}
Let $X$ be a projective variety and let $\mathcal{K}$ be an acceptable cone in $N_{k}(X)$.  For any class $\alpha \in N_{k}(X)$ define
\begin{equation*}
\mc_{\mathcal{K}}(\alpha) := \max_{\beta \in (\alpha - \mathcal{K}) \cap N_{k}(X)_{\mathbb{Z}}} \mc(\beta).
\end{equation*}
Define
\begin{equation*}
\mob_{\mathcal{K}}(\alpha) = \limsup_{m \to \infty} \frac{\mc_{\mathcal{K}}(m\alpha)}{m^{n/n-k}/n!}.
\end{equation*}
\end{defn}

The main goal of this subsection is to show that $\mob_{\mathcal{K}}$ coincides with $\mob$.  First we show that, just as for the mobility function, $\mob_{\mathcal{K}}$ is homogeneous of weight $\frac{n}{n-k}$.

\begin{lem}
Let $X$ be a projective variety and let $\mathcal{K}$ be an acceptable cone in $N_{k}(X)$.  For any class $\alpha \in N_{k}(X)$ and any positive integer $c$ we have
\begin{equation*}
\mob_{\mathcal{K}}(c\alpha) = c^{\frac{n}{n-k}}\mob_{\mathcal{K}}(\alpha).
\end{equation*}
\end{lem}

\begin{proof}
Suppose that $r$ is a positive integer such that $\mc_{\mathcal{K}}(r\alpha) > 0$ and let $\gamma$ denote the class in $(r\alpha - \mathcal{K}) \cap N_{k}(X)_{\mathbb{Z}}$ of maximal mobility count.  Then for any other positive integer $s$, we have 
$$\mc_{\mathcal{K}}((r+s)\alpha) \geq \mc_{\mathcal{K}}(s\alpha).$$ To see this, choose a class $\beta \in (s\alpha - \mathcal{K}) \cap N_{k}(X)_{\mathbb{Z}}$ with maximal mobility count; then $\mc(\beta + \gamma) \geq \mc(\beta)$. We conclude by the following Proposition \ref{lazlemma}.
\end{proof}

\begin{prop}[\cite{lazarsfeld04} Lemma 2.2.38] \label{lazlemma}
Let $f: \mathbb{N} \to \mathbb{R}_{\geq 0}$ be a function.  Suppose that for any $r,s \in \mathbb{N}$ with $f(r) > 0$ we have that $f(r+s) \geq f(s)$.
Then for any $k \in \mathbb{R}_{>0}$ the function $g: \mathbb{N} \to \mathbb{R}_{\geq 0} \cup \{ \infty \}$ defined by
\begin{equation*}
g(r) := \limsup_{m \to \infty} \frac{f(mr)}{m^{k}}
\end{equation*}
satisfies $g(cr) = c^{k}g(r)$ for any $c,r \in \mathbb{N}$.
\end{prop}

\begin{rmk}
Although \cite[Lemma 2.2.38]{lazarsfeld04} only explicitly addresses the volume function, the essential content of the proof is the more general statement above.
\end{rmk}

\begin{lem} \label{effectivenesslemma}
Let $X$ be a projective variety and let $\mathcal{K}$ be an acceptable cone in $N_{k}(X)$.  There is some $\mathbb{Z}$-class $\gamma$ in the interior of $\mathcal{K}$ such that for every $\beta \in \mathcal{K} \cap N_k(X)_{\mathbb{Z}}$, the class $\gamma + \beta$ contains an effective $\mathbb{Z}$-cycle. 
\end{lem}

\begin{proof}
By definition, there is some closed cone $\mathcal{K}'$ such that $\mathcal{K}' - \{0\}$ is contained in the interior of $\Eff_{k}(X)$ and a finite collection of classes of effective $\mathbb{Z}$-cycles $\{ \mu_{\bullet} \}$ such that $\mathcal{K}$ is generated by $\mathcal{K}'$ and the $\mu_{\bullet}$.  Let $\{ \zeta_{\bullet} \}$ be a finite set of classes of big effective $\mathbb{Z}$-cycles such that the real cone generated by the $\zeta_{\bullet}$ contains $\mathcal{K}' - \{0\}$ in its interior. 

Since $\mathcal{K}'$ has nonempty interior and is closed under scalings, it contains an interior class $\eta \in N_k(X)_{\mathbb Z}$. We can assume that $\eta$ is the class of an effective $\mathbb Z$-cycle. Let $\alpha'_{\bullet}\in \Eff_k(X)_{\mathbb Z}$ be a finite set of classes of
effective $\mathbb Z$-cycles that generate $N_k(X)_{\mathbb Z}$. For fixed large $m$, we then have that $\alpha_{\bullet}:=m\eta+\alpha'_{\bullet}$ belongs to $\mathcal{K}'^{int}$. It follows that $\eta$ and $\alpha_{\bullet}$ generate $N_k(X)_{\mathbb Z}$ as a group. 

Then \cite[\S 3 Proposition 3]{khovanskii92} applied to the set $\mathcal{T} := \{\mu_{\bullet}\} \cup \{\zeta_{\bullet}\}\cup\{\alpha_{\bullet}\}\cup\{\eta\}$ yields a class $\gamma$ such that $\gamma + \beta$ is equal to a sum of classes in $\mathcal{T}$ for every $\beta \in \mathcal{K} \cap N_{k}(X)_{\mathbb{Z}}$.
\end{proof}

\begin{prop} \label{kmobismob}
Let $X$ be a projective variety and let $\mathcal{K}$ be an acceptable cone in $N_{k}(X)$.  Then for any class $\alpha \in N_{k}(X)$ we have $\mob_{\mathcal{K}}(\alpha) = \mob(\alpha)$.
\end{prop}

\begin{proof}
First suppose $\alpha$ is a $\mathbb{Z}$-class.  Let $\gamma$ be a class as in Lemma \ref{effectivenesslemma} for $\mathcal{K}$.  Then
\begin{equation*}
\mc(m\alpha) \leq \mc_{\mathcal{K}}(m\alpha) \leq \mc(m\alpha + \gamma)
\end{equation*}
and we conclude by Lemma \ref{addingcycle}.  The statement for $\mathbb{Q}$-cycles follows by homogeneity.

To extend to $\mathbb{R}$-classes, note that if $\beta$ is in the interior of $\mathcal{K}$ then $\mob_{\mathcal{K}}(\alpha + \beta) \geq \mob_{\mathcal{K}}(\alpha)$.  Using the statement for $\mathbb{Q}$-classes, we see that for any big $\mathbb{R}$-class $\alpha$
\begin{align*}
\min_{\beta \in \mathcal{K}^{int}, \alpha + \beta \in N_{k}(X)_{\mathbb{Q}}} \mob(\alpha + \beta) \geq \mob_{\mathcal{K}}(\alpha) \geq \max_{\gamma \in \mathcal{K}^{int}, \alpha - \gamma \in N_{k}(X)_{\mathbb{Q}}} \mob(\alpha - \gamma)
\end{align*}
and we conclude by continuity of mobility.
\end{proof}

\section{Zariski decompositions} \label{zardecomsection}

Just as for divisors, the Zariski decomposition of a class $\alpha \in N_{k}(X)$ identifies a ``positive'' part $P(\alpha)$ that is movable and a ``negative'' part $N(\alpha)$ that does not contribute to the positivity of $\alpha$.  In this section we define the Zariski decomposition and analyze its geometric properties.

\subsection{Zariski decompositions}

A Zariski decomposition for a class $\alpha$ identifies a movable part $P(\alpha)$ that has the same mobility.

\begin{defn} \label{zariskidecompdefn}
Let $X$ be a projective variety.  A Zariski decomposition for a big class $\alpha \in N_{k}(X)$ is a sum $\alpha = P(\alpha) + N(\alpha)$ where $P(\alpha)$ is movable, $N(\alpha)$ is pseudo-effective, and $\mob(P(\alpha)) = \mob(\alpha)$.
\end{defn}

\begin{rmk}
The rational mobility of \cite{lehmann13} can be used in place of the mobility in Definition \ref{zariskidecompdefn}.  Since these two functions share many geometric properties, most of the theory developed below will carry over to this context.

It may also be possible to develop a theory of decompositions compatible with the variation function in \cite{lehmann13} (although one could no longer require the positive part to lie in the movable cone).  This perspective seems to be captured by the $\sigma$-decomposition for cycles of \cite[Remark after III.2.13 Corollary]{nakayama04} which will be discussed in Section \ref{nakayamadecompforcycles}.

Finally, as remarked in the introduction one can develop an analogous theory for other homology theories, e.g.~singular homology of complex varieties.
\end{rmk}

We first verify that for a divisor class this decomposition agrees with the $\sigma$-decomposition of \cite{nakayama04}.

\begin{prop} \label{divisorzardecomprop}
Let $X$ be a smooth projective variety 
and let $L$ be a big divisor.  Then $[L]$ has a unique Zariski decomposition (in the sense of Definition \ref{zariskidecompdefn}) given by
\begin{equation*}
P([L]) = [P_{\sigma}(L)] \qquad \textrm{and} \qquad N([L]) = [N_{\sigma}(L)].
\end{equation*}
\end{prop}

In the proof we will reference theorems that are only stated over $\mathbb{C}$; however, in Section \ref{divisorcharpsection} we have explained how all of these theorems hold over an arbitrary algebraically closed field.

\begin{proof}
It follows easily from the definitions that $[P_{\sigma}(L)]$ and $[N_{\sigma}(L)]$ form a Zariski decomposition for $[L]$.  Indeed, since $\mob([L])$ coincides with $\vol(L)$ in this situation, it suffices to show that $\vol(P_{\sigma}(L)) = \vol(L)$, and in fact it is even true that $h^{0}(X,mP_{\sigma}(L)) = h^{0}(X,mL)$ for any positive integer $m$.

We must show that this is the unique Zariski decomposition.  Suppose that $[L] = [P(L)] + [N(L)]$ is a Zariski decomposition for $[L]$, so $P(L)$ is a movable divisor with $\vol(P(L)) = \vol(L)$ and $N(L)$ is a pseudo-effective divisor.  Since $L$ is big, $0<\vol(L) = \vol(P(L))$ and so $P(L)$ is also big.

Recall that $\langle - \rangle$ denotes the positive product of \cite{bfj09}.  Note that
\begin{align*}
\vol(L) & = \langle L^{n-1} \cdot (P(L) + N(L)) \rangle \\
& \geq \langle L^{n-1} \cdot P(L) \rangle + \langle L^{n-1} \cdot N(L) \rangle \\
& \geq \langle P(L)^{n} \rangle + \langle L^{n-1} \cdot N(L) \rangle \\
& \geq \vol(L) +  \langle L^{n-1} \cdot N(L) \rangle.
\end{align*}
Let $H$ be an ample divisor such that $L-H$ is big.  By super-additivity, we see that
\begin{equation*}
H^{n-1} \cdot \langle N(L) \rangle = \langle H^{n-1} \cdot N(L) \rangle = 0.
\end{equation*}
Thus $P_{\sigma}(N(L))=0$ and $N(L)$ is numerically equivalent to the effective divisor $N_{\sigma}(N(L))$.  Since we are only interested in the numerical class, from now on we replace $N(L)$ by this effective divisor and set $P(L) = L - N(L)$.

Since $P(L)$ is movable and $N(L)$ is effective, we have $N_{\sigma}(L) \leq N(L)$ by \cite[III.1.14 Proposition]{nakayama04}.  Let $E$ denote the effective divisor $N(L) - N_{\sigma}(L)$ and let $\{E_{i}\}_{i=1}^{r}$ denote the irreducible components of $E$.  
We know that $\vol(P_{\sigma}(L) + E_{i}) = \vol(P_{\sigma}(L))$ so that by \cite[Corollary 3.6]{bfj09} we have $\langle P_{\sigma}(L)^{n-1} \rangle \cdot E_{i} = 0$.  But this implies that $P_{\sigma}(L)$ is not $E_{i}$-big by \cite[Theorem 4.9]{bfj09}, so that $E_{i}$ is contained in the augmented base locus of $P_{\sigma}(L)$.  Thus $\Supp(E)$ is contained in $\mathbf{B}_{+}(P_{\sigma}(L))$.

Fix an ample divisor $A$.  Since $\Supp(E) \subset \mathbf{B}_{+}(P_{\sigma}(L))$, for any $\epsilon>0$ there is a sufficiently small $\delta>0$ so that
\begin{equation*}
\delta E \leq N_{\sigma}(P_{\sigma}(L) - \epsilon A).
\end{equation*}
In particular $P_{\sigma}(P_{\sigma}(L) - \epsilon A + tE) = P_{\sigma}(P_{\sigma}(L) - \epsilon A)$ for any $t>0$.  By the continuity of the positive part on the big cone as in \cite[Proposition 2.9]{bfj09}, this implies that $P_{\sigma}(P_{\sigma}(L)+E) = P_{\sigma}(L)$.  In particular $P_{\sigma}(L)+E$ can not be movable unless $E=0$, showing that $P(L) = P_{\sigma}(L)$ as claimed.
\end{proof}

The main theorem of this section is the existence of Zariski decompositions.

\begin{thrm} \label{zardecomexist}
Let $X$ be a projective variety and let $\alpha \in N_{k}(X)$ be a big class.  Then there is a Zariski decomposition for $\alpha$.
\end{thrm}

\begin{proof}
We first consider the case when $\alpha \in N_{k}(X)_{\mathbb{Z}}$.  Fix a subset $\mathcal{M} \subset \mathbb{N}$ such that
\begin{equation*}
\mob(\alpha) = \lim_{m \to \infty, m \in \mathcal{M}} \frac{\mc(m\alpha)}{m^{\frac{n}{n-k}}/n!}.
\end{equation*}

For each $m \in \mathcal{M}$, let $p_{m}: U_{m} \to V_{m}$ denote a family of effective $\mathbb{Z}$-cycles with class $m\alpha$ of maximal mobility count.  We split $U_{m}$ into two pieces: those components that dominate $X$ and those that do not.  We set $\alpha_{m}$ to be the class of the first part and $\beta_{m}$ to be the class of the second part.  Since $\alpha_{m}$ is constructed by removing only those components of $U_{m}$ that do not dominate $X$, we have $\mc(\alpha_{m}) \geq \mc(m \alpha)$.  The reverse inequality is also true by Lemma \ref{mobilityincreases} since $\beta_{m}$ is the class of of an effective $\mathbb{Z}$-cycle.  Thus $\mc(\alpha_{m}) = \mc(m \alpha)$.

Note that $\alpha - \frac{1}{m}\beta_{m}$ is pseudo-effective, so that the classes $\frac{1}{m}\beta_{m}$ vary in a compact set as $m$ increases.  Let $N(\alpha)$ be an accumulation point of the set $\{\frac{1}{m}\beta_{m} \}_{m \in \mathcal{M}}$ and choose a subset $\mathcal{M}' \subset \mathcal{M}$ so that $\{\frac{1}{m}\beta_{m} \}_{m \in \mathcal{M}'}$ converges to $N(\alpha)$.  Clearly $N(\alpha)$ is pseudo-effective.  Furthermore $P(\alpha) := \alpha - N(\alpha)$ is movable, since it is the limit of the movable classes $\{ \frac{1}{m}(\alpha_{m}) \}_{m \in \mathcal{M}'}$.

It remains to show that $\mob(P(\alpha)) = \mob(\alpha)$.  
The inequality $\leq$ follows from Lemma \ref{mobilityincreases}.  Conversely, fix an acceptable cone $\mathcal{K}$ in $N_{k}(X)$.  Choose a  class $\gamma$ for $\mathcal{K}$ as in Lemma \ref{effectivenesslemma}.  Fix $\epsilon > 0$; then there are infinitely many values of $m \in \mathcal{M}'$ so that $\epsilon \gamma + P(\alpha) - \frac{1}{m}\alpha_{m}$ lies in the interior of $\mathcal{K}$.  For these $m$,
\begin{align*}
\mc_{\mathcal{K}}(mP(\alpha) + m \epsilon \gamma)  \geq \mc_{\mathcal{K}}(\alpha_{m}) \geq \mc(\alpha_{m}) = \mc(m\alpha)
\end{align*}
Dividing by $m^{n/n-k}/n!$ and taking limits, we find that for any $\epsilon>0$
\begin{align*}
\mob(P(\alpha) + \epsilon \gamma) & = \mob_{\mathcal{K}}(P(\alpha) + \epsilon \gamma) \textrm { by Proposition \ref{kmobismob}} \\
&  \geq \limsup_{m \to \infty, m \in \mathcal{M}'} \frac{\mc(m\alpha)}{m^{\frac{n}{n-k}}/n!} \\
& = \mob(\alpha) \textrm{ by construction.}
\end{align*}
Using the continuity of mobility, we see that $\mob(P(\alpha)) = \mob(\alpha)$ as claimed.  Thus every big $\mathbb{Z}$-class $\alpha$ admits a Zariski decomposition; by homogeneity, the same is true for every big $\mathbb{Q}$-class.

Finally, suppose that $\alpha$ is a big $\mathbb{R}$-class.  Let $\{ \alpha_{j} \}_{j=1}^{\infty}$ be a sequence of big $\mathbb{Q}$-classes converging to $\alpha$.  For each $\alpha_{j}$, choose a Zariski decomposition $\alpha_{j} = P_{j} + N_{j}$.  Since the $P_{j}$ vary inside a compact set, there is a subsequence converging to a class $P(\alpha)$.  Set $N(\alpha) = \alpha - P(\alpha)$.  By the continuity of mobility, $\alpha = P(\alpha) + N(\alpha)$ is a Zariski decomposition.
\end{proof}

It would be useful to know if Zariski decompositions are compatible with the mobility count and not just the mobility.

\begin{ques} \label{mobilitycountzardecomquestion}
Let $X$ be a projective variety and let $\mathcal{K} \subset N_{k}(X)$ be an acceptable cone.  Let $\alpha \in N_{k}(X)$ be a big class and let $\alpha = P(\alpha) + N(\alpha)$ be a Zariski decomposition.  Is there some $\gamma \in \Eff_{k}(X)$ such that $\mc_{\mathcal{K}}(m\alpha) \leq \mc_{\mathcal{K}}(mP(\alpha)+\gamma)$ for every $m>0$?
\end{ques}

The $\sigma$-decomposition of a divisor satisfies a number of important geometric properties.  We next study whether there are analogous statements for Zariski decompositions of cycles.

\subsection{Structure of Zariski decompositions}

For big divisor classes Zariski decompositions are unique.  We do not know whether this is true for higher codimension cycles.  In fact, even if $\alpha$ is movable Definition \ref{zariskidecompdefn} does not a priori rule out the existence of a Zariski decomposition $\alpha = P(\alpha) + N(\alpha)$ with $N(\alpha) \neq 0$.

\begin{ques}
Is there an example of a projective variety $X$ and a big class $\alpha \in N_{k}(X)$ such that $\alpha$ has more than one Zariski decomposition?
\end{ques}

Nevertheless, the set of possible positive parts of a big class $\alpha$ has some structure.

\begin{lem} \label{structureofnegativeparts}
Let $X$ be a projective variety and let $\alpha \in N_{k}(X)$ be a big class.  Let $\mathcal{P}_{\alpha}$ denote the set of positive parts of all possible Zariski decompositions of $\alpha$ and let $\mathcal{N}_{\alpha}$ denote the set of all possible negative parts.
\begin{itemize}
\item[(1)] $\mathcal{P}_{\alpha}$ and $\mathcal{N}_{\alpha}$ are compact.
\item[(2)] $\mathcal{N}_{\alpha}$ lies on the boundary of $\Eff_{k}(X)$. 
\item[(3)] Assume Conjecture \ref{concaveconj}.  Then $\mathcal{P}_{\alpha}$ and $\mathcal{N}_{\alpha}$ are convex.
\end{itemize}
\end{lem}

\begin{proof}
Note that $\mathcal{P}_{\alpha}$ is contained in the compact subset $\{ \beta \in \Mov_{k}(X) | \beta \preceq \alpha \}$.  (1) is then a consequence of the continuity of the mobility function and that fact that the movable and pseudo-effective cones are closed.  (2) follows immediately from the second part of Lemma \ref{mobilityincreases}.  To see (3), let $n$ denote the dimension of $X$.  If $P$ and $P'$ are both positive parts for $\alpha$, then Conjecture \ref{concaveconj} would imply that for any $0 \leq t \leq 1$
\begin{align*}
\mob(\alpha)^{\frac{n-k}{n}} & \geq \mob(tP + (1-t)P')^{\frac{n-k}{n}}\\
& \geq t \mob(P)^{\frac{n-k}{n}} +  (1-t)\mob(P')^{\frac{n-k}{n}} = \mob(\alpha)^{\frac{n-k}{n}}
\end{align*}
so that $tP + (1-t)P'$ is also the positive part of a Zariski decomposition of $\alpha$.
\end{proof}

\subsection{Directed property}\label{directed property}

A key property of $\sigma$-decompositions is the directed property: for a big divisor $L$, the set of all movable classes $\beta$ with $\beta \preceq [L]$ forms a directed set under the relation $\preceq$.  The class $[P_{\sigma}(L)]$ is the unique maximum of this directed set.

More generally, suppose that $\alpha \in N_{k}(X)$ is a big class such that the set $\mathcal{S} := \{ \beta \in \Mov_{k}(X) | \beta \preceq \alpha \}$ contains a unique maximal element $P$ under the relation $\preceq$.  Then $P$ is a positive part for $\alpha$: for any $\gamma \in \mathcal{S}$ we have
\begin{equation*}
\mob(\gamma) \leq \mob(P) \leq \mob(\alpha)
\end{equation*}
and Theorem \ref{zardecomexist} shows that $\mob(\gamma) = \mob(\alpha)$ for some $\gamma \in \mathcal{S}$.  However, in contrast to the situation for divisors, $\mathcal{S}$ is in general not a directed set.

\begin{rmk} \label{weakdirectedproperty}
There is a weaker property that may hold.  Say that a big class $\alpha$ has the weak directed property if the set of its positive parts forms a directed set under $\preceq$.  As explained above, this would be implied by the directed property for all movable classes dominated by $\alpha$.  However, the weak directed property is likely to be strictly weaker.

When the weak directed property holds, it allows us to distinguish an ``optimal'' Zariski decomposition.  This is particularly important for classes on the boundary of the pseudo-effective cone, as will be discussed in the next section.
\end{rmk}

\begin{exmple} \label{notdirectedexample}
We give an example of a smooth toric variety $X$ and a non-nef curve class $\alpha$ such that the set
\begin{equation*}
\mathcal{S} := \{ \beta \in \Mov_{1}(X) | \beta \preceq \alpha \}
\end{equation*}
does not form a directed set under the relation $\preceq$.

The example comes from \cite[Example 2]{pay06}.  Let $X$ be the toric variety defined by a fan in $N = \mathbb{Z}^{3}$ on the rays generated by the following
\begin{align*}
v_{1} & = (1,1,-1) \qquad v_{2} = (-1,0,-1)  & v_{3} & = (0,-1,-1) \qquad v_{4} = (1, 0,-1) \\
v_{5} & = (0,1,-1) \qquad v_{6} = (-1,-1,-1) & v_{7} & = (0,0,-1) \qquad v_{8} = (0,0,1)
\end{align*}
with maximal cones
\begin{align*}
\langle v_{1}, v_{4}, v_{8} \rangle, \, \langle v_{1}, v_{5}, v_{8} \rangle, \, \langle v_{2}, v_{5}, v_{8} \rangle, \, \langle v_{2}, v_{6}, v_{8} \rangle, \\
\langle v_{3}, v_{6}, v_{8} \rangle, \, \langle v_{3}, v_{4}, v_{8} \rangle, \, \langle v_{1}, v_{4}, v_{5} \rangle, \, \langle v_{2}, v_{5}, v_{6} \rangle, \\
\langle v_{3}, v_{4}, v_{6} \rangle, \, \langle v_{4}, v_{5}, v_{7} \rangle, \, \langle v_{5}, v_{6}, v_{7} \rangle, \, \langle v_{4}, v_{6}, v_{7} \rangle.
\end{align*}

We first identify the pseudo-effective cone of divisors.  Let $D_{i}$ denote the numerical divisor class corresponding to the ray $v_{i}$.  Then $N^{1}(X)$ is spanned by the classes $D_{1}$, $D_{2}$, $D_{3}$, $D_{7}$, and $D_{8}$; all calculations below will be done in this basis with this ordering.  The pseudo-effective cone is generated by $D_{1}$, $D_{2}$, $D_{3}$, $D_{7}$, $D_{8}$,
\begin{align*}
D_{4} & = \frac{1}{3} \left( -2,1,-2,-1,1 \right), \\
D_{5} & = \frac{1}{3} \left( -2,-2,1,-1,1 \right),  \textrm{ and }\\
D_{6} & = \frac{1}{3} \left( 1,-2,-2,-1,1 \right).
\end{align*}
(In fact $D_{8}$ is redundant.)  If $C_{ij}$ with $1\leq i<j\leq 8$ denotes the intersection of $D_i$ with $D_j$, then $C_{ij}=0$ whenever $v_i$ and $v_j$ do not span a face of the fan. The classes of the remaining $C_{ij}$ generate $\Eff_1(X)$. Furthermore we have the following intersection relations 
\begin{equation*}
\begin{array}{c|c|c|c|c|c|c|c|c}
& D_{1} & D_{2} & D_{3} & D_{4} & D_{5} & D_{6} & D_{7} & D_{8}  \\ \hline
C_{14} & -1 & 0 & 0 & 1 & 1 & 0 & 0 & 1 \\ \hline
C_{15} & -1 & 0 & 0 & 1 & 1 & 0 & 0 & 1 \\ \hline
C_{18} & -1 & 0 & 0 & 1 & 1 & 0 & 0 & 1 \\ \hline
C_{25} & 0 & -1 & 0 & 0 & 1 & 1 & 0 & 1 \\ \hline
C_{26} & 0 & -1 & 0 & 0 & 1 & 1 & 0 & 1 \\ \hline
C_{28} & 0 & -1 & 0 & 0 & 1 & 1 & 0 & 1 \\ \hline
C_{34} & 0 & 0 & -1 & 1 & 0 & 1 & 0 & 1 \\ \hline
C_{36} & 0 & 0 & -1 & 1 & 0 & 1 & 0 & 1 \\ \hline
C_{38} & 0 & 0 & -1 & 1 & 0 & 1 & 0 & 1 \\ \hline
C_{45} & 1 & 0 & 0 & -1 & -1 & 0 & 1 & 0 \\ \hline
C_{46} & 0 & 0 & 1 & -1 & 0 & -1 & 1 & 0 \\ \hline
C_{47} & 0 & 0 & 0 & 1 & 1 & 1 & -3 & 0 \\ \hline
C_{48} & 1 & 0 & 1 & -1 & 0 & 0 & 0 & 1 \\ \hline
C_{56} & 0 & 1 & 0 & 0 & -1 & -1 & 1 & 0 \\ \hline
C_{57} & 0 & 0 & 0 & 1 & 1 & 1 & -3 & 0 \\ \hline
C_{58} & 1 & 1 & 0 & 0 & -1 & 0 & 0 & 1 \\ \hline
C_{67} & 0 & 0 & 0 & 1 & 1 & 1 & -3 & 0 \\ \hline
C_{68} & 0 & 1 & 1 & 0 & 0 & -1 & 0 & 1
\end{array}
\end{equation*}
The pseudo-effective cone of curves is generated by the following curves (which we still write in the basis dual to $D_{1},D_{2},D_{3},D_{7},D_{8}$).
\begin{align*}
C_{14}=C_{15}=C_{18} & =(-1,0,0,0,1) \\
C_{25}=C_{26}=C_{28} & =(0,-1,0,0,1) \\
C_{34}=C_{36}=C_{38} & =(0,0,-1,0,1) \\
C_{47}=C_{57}=C_{67} & =(0,0,0,-3,0) \\
C_{45} & = (1,0,0,1,0) \\
C_{46} & = (0,0,1,1,0) \\
C_{56} & = (0,1,0,1,0)
\end{align*}
The other curves are
\begin{align*}
C_{48} & = C_{14} + 2 C_{45} + C_{46} + C_{57} \\
C_{58} & = C_{14} + 2 C_{45} + C_{56} + C_{57} \\
C_{68} & = C_{25} + 2 C_{56} + C_{46} + C_{57}
\end{align*}
Since $\Mov_{1}(X)$ is the dual of $\Eff^{1}(X)$, it is generated by
\begin{align*}
M_{1} & = (1,0,0,0,2) \qquad M_{2} = (0,1,0,0,2) \qquad M_{3} = (0,0,1,0,2) \\
M_{4} & = (0,0,0,1,1) \qquad M_{5}  = (0,0,0,0,1) \qquad M_{6} = (1,1,1,0,3).
\end{align*}

\noindent Now consider the curve class $$\alpha = (1,1,0,1,2).$$  
Note that $\alpha = \frac 23C_{14} + \frac23C_{25} + \frac23C_{34} + C_{47}+ \frac53C_{45} + \frac 23C_{46}+\frac 53C_{56}$ is big.  
It is not movable since the sum of the first four coordinates is greater than the last coordinate. 
We have $\alpha = M_{1} + C_{56}$ and $\alpha = M_{2} + C_{45}$. 
We claim that there is no movable curve class $\beta$ with $\alpha \succeq \beta$, $\beta \succeq M_{1}$, and $\beta \succeq M_{2}$. 
To see this, note that if we add any positive multiple of $C_{14}$, $C_{25}$, or $C_{34}$ to $M_{1}$ or $M_{2}$, we get a class that is no longer dominated by $\alpha$, since the last coordinate is too big. 
Assume $M_1+aC_{47}+bC_{45}+cC_{46}+dC_{56}$ is movable, where $a$, $b$, $c$, $d$, are nonnegative real numbers. 
By intersecting with $D_4$, $D_5$, and $D_7$, and using that $\Mov_1(X)$ is dual to $\Eff^1(X)$, we get
\begin{align*}
a-b-c&\geq0\\
a-b-d&\geq0\\
-3a+b+c+d&\geq0
\end{align*}
The only solution in $\mathbb R_{\geq 0}^4$ is $(0,0,0,0)$.\qed

\end{exmple}

\subsection{Continuity}

For surfaces, the continuity of Zariski decompositions on the big cone was established by \cite[Main Theorem]{bks04}, and the continuity of the positive part of the $\sigma$-decomposition on the big cone was established by \cite[III.1.7 Lemma]{nakayama04}.  By analogy, one expects Zariski decompositions to vary continuously on the big cone.  The following topological lemma allows us to make this expectation more precise.

\begin{lem} \label{continuityofzardecomlem}
Let $f: C \to D$ be a surjective continuous open map of compact metric spaces and let $g: C \to \mathbb{R}$ be a continuous function.  Define the set $\mathcal{S} := \{ c \in C | g(c) = \max_{c' \in f^{-1}f(c)} g(c')\}$.  For any $d \in D$, let $F_{d}$ denote the fiber of $f$ over $d$ and let $\mathcal{S}_{d}$ denote the (non-empty) intersection $F_{d} \cap \mathcal{S}$.  Then:
\begin{enumerate}
\item The function $h: D \to \mathbb{R}$ defined by $h(d) = \max_{c \in F_{d}} g(c)$ is continuous.
\item $\mathcal{S}$ is closed in $C$.
\item For any point $d \in D$ and any open neighborhood $U$ of $\mathcal{S}_{d}$, there is some neighborhood $V$ of $d$ such that $\mathcal{S}_{d'} \cap U$ is non-empty for every $d' \in V$.
\end{enumerate}
\end{lem}

\begin{proof}
(1) Note that since $f$ is a continuous map of compact metric spaces it is closed (as well as open).  Thus for any $r \in \mathbb{R}$ the set $f(g^{-1}([r,\infty))) = h^{-1}([r,\infty))$ is closed and the set $f(g^{-1}((r,\infty))) = h^{-1}((r,\infty))$ is open.  This shows that $h$ is both upper and lower semi-continuous.

(2) The set $\mathcal{S}$ is the locus of points where $g(c) = (h \circ f)(c)$, and is thus closed.

(3) Suppose conversely.  Fix a point $d \in D$ and an open neighborhood $U$ of $\mathcal{S}_{d}$ violating the theorem.  For any open ball $B_{1/n}(d)$, we can find an element $d_{n} \in B_{1/n}(d)$ such that $\mathcal{S}_{d_{n}} \cap U$ is empty.  Let $c_{n}$ be any element of $\mathcal{S}_{d_{n}}$; then by (1) some subsequence of the $c_{n}$ converges to an element of $\mathcal{S}_{d}$, contradicting the supposition.
\end{proof}

\begin{cor} \label{continuityofzardecom}
Let $X$ be a projective variety.  Let $\mathcal{T}$ be a compact set contained in the interior of $\Eff_{k}(X)$ such that $\mathcal{T}$ is the closure of its interior.  Let $\mathcal{P}_{\mathcal{T}} \subset \mathcal{T} \times N_{k}(X)$ denote the set of pairs $(\alpha,P(\alpha))$ where $\alpha$ is a class in $\mathcal{T}$ and $P(\alpha)$ is the positive part of a Zariski decomposition for $\alpha$.  Then:
\begin{enumerate}
\item $\mathcal{P}_{\mathcal{T}}$ is closed.
\item Let $\mathcal{P}_{\alpha}$ denote the fiber of $\mathcal{P}_{\mathcal{T}}$ over $\alpha \in \mathcal{T}$.  For any point $\alpha \in \mathcal{T}$ and for any open neighborhood $U$ of $\mathcal{P}_{\alpha}$, there is a neighborhood $V$ of $\alpha$ such that every $\beta \in V$ has a Zariski decomposition with $P(\beta) \subset U$.
\end{enumerate}
\end{cor}

\begin{proof}
Define the set $M_{\mathcal{T}} \subset \mathcal{T} \times N_{k}(X)$ consisting of the pairs $(\alpha,\gamma)$ where $\alpha \in \mathcal{T}$ and $\gamma \in \Mov_{k}(X) \cap (\alpha - \Eff_k(X))$.  Let $\pi_{1}$ and $\pi_{2}$ denote the two projection maps from $\mathcal{M}_{\mathcal{T}}$.  Note that $M_{\mathcal{T}}$ is compact and $\mathcal{P}_{\mathcal{T}}$ is the subset of $M_{\mathcal{T}}$ consisting of the elements $(\alpha,\gamma)$ that satisfy
\begin{equation*}
\mob(\gamma) = \max_{\beta \in \pi_{1}^{-1}(\alpha)} \mob(\beta).
\end{equation*}
$\mathcal{M}_{\mathcal{T}}$ is convex and is the closure of its interior.  Thus $\pi_{1}|_{\mathcal{M}_{\mathcal{T}}}$ is open: any open neighborhood $U$ of a point $p \in \mathcal{M}_{\mathcal{T}}$ contains a full-dimensional subset of $\mathcal{M}_{\mathcal{T}}$, so $\pi_{1}(U)$ contains an open neighborhood of $\pi_{1}(p)$.  We conclude by applying Lemma \ref{continuityofzardecomlem} to $\pi_{1}: \mathcal{M}_{\mathcal{T}} \to \mathcal{T}$ and $\mob \circ \pi_{2}: \mathcal{M}_{\mathcal{T}} \to \mathbb{R}$.
\end{proof}

Corollary \ref{continuityofzardecom} allows us to extend the definition of a Zariski decomposition to any pseudo-effective class in a natural way.

\begin{defn} \label{psefzardecomdef}
Let $X$ be a projective variety.  Consider the subset $M \subset \Eff_{k}(X) \times \Eff_{k}(X)$ that consists of all pairs $(\beta,P(\beta))$ where $\beta$ is a big class and $P(\beta)$ is the positive part of a Zariski decompositions of $\beta$.  A Zariski decomposition for a pseudo-effective class $\alpha$ is any expression $\alpha = P(\alpha) + N(\alpha)$ such that $(\alpha,P(\alpha))$ lies in the closure of $M$.
\end{defn}

By Corollary \ref{continuityofzardecom} this definition agrees with our earlier definition when $\alpha$ is big.  Note that positive parts of pseudo-effective classes are still movable and negative parts are still pseudo-effective.  Furthermore Lemma \ref{structureofnegativeparts} extends immediately to pseudo-effective classes. 

\begin{caut} This definition differs from the established conventions for divisor classes.  When $\alpha$ is a divisor class on the boundary of the pseudo-effective cone, with our definition the set of possible negative parts may be infinite (see Examples \ref{infinitesurfaceexample} and \ref{nakayamaexample}).  Usually one works only with $N_{\sigma}(L)$ which is defined as the unique effective divisor whose class is the minimum under the relation $\preceq$.

In arbitrary codimension, one might still hope to be able to choose a ``distinguished'' Zariski decomposition in some way.  Example \ref{notdirectedexample} shows that, in contrast to the situation for divisors, the set of movable classes dominated by $\alpha$ may not be a directed set.  However, we do not know of an example where the weak directed property of Remark \ref{weakdirectedproperty} fails.  Indeed, the weak directed property would probably follow from an affirmative answer to Question \ref{subadditivityques}.(2).

An alternative approach to choosing a ``distinguished'' Zariski decomposition is outlined in Section \ref{minlengthsection}.
\end{caut}

\begin{exmple} \label{boundaryzardecomexample}
Assume $\alpha$ is a boundary class in $\Eff_k(X)$. If $\alpha$ is not contained in the interior of any face of $\Eff_k(X)$ that intersects $\Mov_k(X)$, then $\alpha=0+N(\alpha)$ is the unique Zariski decomposition of $\alpha$.
\end{exmple}

For divisors Zariski decompositions are unique on the big cone by Proposition \ref{divisorzardecomprop}.  However, this uniqueness does not extend to the pseudo-effective boundary.  The following examples show that a divisor class $\alpha$ on the boundary of the pseudo-effective cone may have infinitely many Zariski decompositions.  Furthermore, it is possible for the negative part of a Zariski decomposition to be nef.

\begin{exmple} \label{infinitesurfaceexample}
Suppose that $S$ is a smooth surface.  Then for any fixed prime divisor $\Gamma$ the function $\sigma_{\Gamma}$ varies continuously on the entire pseudo-effective cone of $S$.  (This is verified for big classes in \cite[Main Theorem]{bks04} and for all pseudo-effective classes in \cite[III.1.19 Proposition]{nakayama04}).  However, in general $N_{\sigma}$ is only lower semicontinuous, and in fact there are two distinct behaviors for a pseudo-effective curve class $\alpha$:
\begin{itemize}
\item $N_{\sigma}$ is continuous at $\alpha$.  This is always the case for big curve classes by \cite{bks04}.   More generally, it holds if there are finitely many prime divisors $\Gamma_{i}$ and an open subset $U$ of $\alpha$ such that $N_{\sigma}$ of any pseudo-effective class contained in $U$ is supported on the $\Gamma_{i}$.

In this case $\alpha$ has a unique Zariski decomposition in the sense of Definition \ref{psefzardecomdef}.  To see this, note that if we take big classes $\beta_{i}$ whose limit is $\alpha$, the negative parts $N_{\sigma}(\beta_{i})$ must converge to $N_{\sigma}(\alpha)$.   Thus the unique Zariski decomposition of $\alpha$ consists of the numerical classes of its Fujita-Zariski decomposition.

\item $N_{\sigma}$ is not continuous at $\alpha$.  A well-known example is when $S$ is the blow-up of $\mathbb{P}^{2}$ at nine very general points and $\alpha$ is the class of $-K_{X}$.

In this case $\alpha$ has infinitely many Zariski decompositions in the sense of Definition \ref{psefzardecomdef}.  Indeed, write $\alpha = P+N$ for the Fujita-Zariski decomposition of $\alpha$.  Since $N_{\sigma}$ is not continuous, there must be a sequence of big curve classes $\beta_{i}$ such that $\lim_{i \to \infty} N_{\sigma}(\beta_{i}) \succ N$.  In other words, there is some $\gamma$ with $0 \prec \gamma \preceq P$ such that
\begin{equation*}
\alpha = (P - \gamma) + ( N + \gamma)
\end{equation*}
is a Zariski decomposition for $\alpha$.  Since the log-concavity of the volume function for divisors is known, Lemma \ref{structureofnegativeparts}.(3) shows that the set of all possible negative parts for $\alpha$ is convex; thus for any $0 \leq t \leq 1$ the expression
\begin{equation*}
\alpha = (P - t\gamma) + ( N +  t\gamma)
\end{equation*}
is again a Zariski decomposition.
\end{itemize} 
\end{exmple}

\begin{exmple} \label{nakayamaexample}
The same dichotomy holds for divisors in higher dimensions.  In fact, an example of \cite{nakayama04} shows that even $\sigma_{\Gamma}$ is not necessarily continuous: let $S$ be an abelian surface and let $L$ be a divisor on $S$ that is non-zero and nef but not big.  Set $X = \mathbb{P}_{S}(\mathcal{O} \oplus \mathcal{O}(-L))$ and let $\pi: S \to X$ denote the projection.  Let $E$ be the section defined by the projection of the defining bundle onto $\mathcal{O}(-L)$.  Then $\pi^{*}L + E$ is nef, but $E$ occurs with coefficient one in the negative part of nearby divisors of the form $\pi^{*}\alpha + E$ where $\alpha$ is on the pseudo-effective boundary of $\Eff_{1}(S)$.

In this case, by taking limits we see that the nef class $\alpha = [\pi^{*}L + E]$ admits a Zariski decomposition $\alpha = P + N$ with $N \neq 0$.  Since $\alpha$ also admits the trivial Zariski decomposition, it has infinitely many Zariski decompositions $\alpha = (P + tN) + (1-t)N$ for $0 \leq t \leq 1$ by using the log-concavity of the volume function for divisors and Lemma \ref{structureofnegativeparts}.(3).
\end{exmple}

\subsection{Subadditivity}

The negative parts of Zariski decompositions behave well with respect to certain kinds of sums.  We can only prove the most general statements if we assume Conjecture \ref{concaveconj}.

\begin{lem} \label{zardecomadditivity}
Let $X$ be a projective variety and let $\alpha \in \Eff_{k}(X)$.  Suppose that $\alpha = P(\alpha) + N(\alpha)$ is a Zariski decomposition for $\alpha$.
\begin{enumerate}
\item For $0 \leq t < 1$, the expression $P(\alpha) + (1-t)N(\alpha)$ is a Zariski decomposition for $\alpha - tN(\alpha)$.
\item Assume Conjecture \ref{concaveconj}.  Then for $t>0$, the expression $P(\alpha) + (1+t)N(\alpha)$ is a Zariski decomposition for $\alpha + tN(\alpha)$.
\end{enumerate}
\end{lem}

\begin{proof}
(1) is clear for big classes; it extends to pseudo-effective classes by taking closures.  Similarly it suffices to prove (2) when $\alpha$ is big.  In this case note that $\mob(P(\alpha)) \leq \mob(P(\alpha) + (1+t)N(\alpha))$ by Lemma \ref{mobilityincreases}.  Conversely, assuming the log-concavity of the mobility function as in Conjecture \ref{concaveconj}, we have for any $t>0$
\begin{align*}
\mob^{\frac{n-k}n}(P(\alpha)+(1+t)N(\alpha)) & \leq 
\mob^{\frac{n-k}n}((1+t)(P(\alpha)+N(\alpha))) - \mob^{\frac{n-k}n}(tP(\alpha)) \\
& \leq (1+t) \mob^{\frac{n-k}{n}}(\alpha) - t \mob^{\frac{n-k}{n}}(P(\alpha)) \\
& = \mob^{\frac{n-k}{n}}(P(\alpha)).
\end{align*}
\end{proof}

We can prove a slightly weaker statement than Lemma \ref{zardecomadditivity} (2) without relying on Conjecture \ref{concaveconj}.  In fact, the following statement works in a much more general setting.

\begin{lem} \label{zardecomstabilization}
Let $X$ be a projective variety and let $\alpha \in N_{k}(X)$ be a big class.  Suppose that $\beta \in \Eff_{k}(X)$ lies on an isolated extremal ray of $\Eff_{k}(X)$ that is not movable.  Set $\gamma_{t} := \alpha + t \beta$ for $t>0$.
\begin{enumerate}
\item The positive parts of the classes $\gamma_{t}$ for any $t>0$ are all contained in a bounded set.  In particular the values of $\mob(\gamma_{t})$ are bounded above.
\item For some sufficiently large $t_{0}$ there is a Zariski decomposition $\gamma_{t_{0}} = P + N$ such that for any $t > t_{0}$ the expression $\gamma_{t} = P + (N + t\beta)$ is a Zariski decomposition of $\gamma_{t}$.
\item There is some $t_{0} > 0$ such that $\mob(\gamma_{t_{0}}) = \sup_{t} \mob(\gamma_{t})$.
\end{enumerate}
\end{lem}

\begin{rmk}
Lemma \ref{zardecomstabilization} (1) holds for any $\beta \in \partial \Eff_{k}(X)$ that is contained in a hyperplane $H$ supporting $\Eff_{k}(X)$ but not intersecting $\Mov_{k}(X) - \{ 0 \}$.
\end{rmk}

\begin{proof}
Since $\beta$ lies on an isolated extremal ray, there is a function $L \in N^{k}(X)$ that vanishes on $\beta$ and takes non-negative values on $\Eff_{k}(X)$ and strictly positive values on $\Mov_{k}(X) - \{ 0 \}$.  For any $p>0$ the set  of classes $\xi \in \Mov_{k}(X)$ satisfying $L(\xi) \leq p$ is compact.  Note that for any $t>0$ we have $L(\gamma_{t}) = L(\alpha)$, showing that all possible movable parts of $\gamma_{t}$ for $t>0$ lie in the compact set of classes $\xi \in \Mov_{k}(X)$ with $L(\xi) \leq L(\alpha)$.  This proves (1).

To prove (2) and (3), we first show that the set
\begin{equation*}
\mathcal{C} = \{ \xi \in N_{k}(X) | \xi = \gamma_{t} - \tau \textrm{ for some }t>0 \textrm{ and some } \tau \in \Eff_{k}(X) \}
\end{equation*}
is closed.  Let $V$ be the vector space obtained as the quotient $g: N_{k}(X) \to V$ by the ray generated by $\beta$.  Since $\beta$ lies on an isolated extremal ray, $g(\Eff_{k}(X))$ is a closed subset of $V$.  Since
\begin{equation*}
\mathcal{C} = \alpha + (-1) \cdot g^{-1}g(\Eff_{k}(X)),
\end{equation*}
it is also closed.

Let $\mathcal{P}_{\gamma_{t}}$ denote the union of all the positive parts for $\gamma_{t}$.  Since each $\mathcal{P}_{\gamma_{t}}$ is non-empty, by (1) there must be an accumulation point $P$ of the sets $\mathcal{P}_{\gamma_{t}}$ as $t$ increases.  Then $\mob(P)=\sup_{t>0}\mob(\gamma_{t})$ by continuity of mobility.  Since $\mathcal{C}$ is closed, $P$ lies in $\mathcal{C}$.  Thus, there is some pseudo-effective $N$ such that $P+N = \gamma_{t_{0}}$ for some $t_{0} > 0$.  Then for $t \geq t_{0}$, we have
\begin{equation*}
\mob(\gamma_{t}) \leq \sup_{t>0} \mob(\gamma_{t}) = \mob(P) \leq \mob(\gamma_{t})
\end{equation*}
yielding an equality and establishing (2) and (3).
\end{proof}

The $\sigma$-decomposition satisfies a number of other (sub)additivity properties.  It is unclear whether the Zariski decomposition satisfies the same properties in higher codimension since the directed property fails.

\begin{ques} \label{subadditivityques}
Let $\alpha \in N_{k}(X)$ be a big class and let $\alpha = P+N$ be a Zariski decomposition.
\begin{enumerate}
\item Suppose that $\gamma \succeq 0$.  Is there a Zariski decomposition of $\alpha + \gamma$ with positive part $P' \succeq P$?
\item Suppose that $\gamma$ lies in the interior of the movable cone.  Is there a Zariski decomposition of $\alpha + \gamma$ with positive part $P' \succeq P + \gamma$?
\end{enumerate}
\end{ques}

\subsection{Effectiveness of negative parts}

In contrast to the situation for divisors, the negative part of a Zariski decomposition may not be the class of an effective cycle.

\begin{exmple}
Let $S$ be a smooth surface with a nef curve class $\beta$ that generates an isolated extremal ray of $N_{1}(S)$ and is not represented by an effective cycle.  An example of such a surface $S$ is given by Mumford (see \cite[Example 1.5.2]{lazarsfeld04}): let $C$ be a curve of genus $\geq 2$.  \cite{hartshorne70} verifies that $C$ carries a degree $0$ rank $2$ vector bundle $E$ satisfying
\begin{equation*}
H^{0}(C,\Sym^{m}E \otimes D) = 0
\end{equation*}
for every positive integer $m$ and every divisor $D$ of degree $\leq 0$.  We can take $S$ to be $\mathbb{P}_{C}(\mathcal{E})$ and set $\beta$ to be a nef class generating the unique boundary ray which is not the fiber of the projection map.

Fix an ample divisor $H$ on $S$ and consider the $\mathbb{P}^{1}$-bundle $\pi: X \to S$ defined by the sheaf $\mathcal{O}_{S} \oplus \mathcal{O}_{S}(-H)$.  The map $\pi$ has a section $T = \mathbb{P}(\mathcal{O}_{S})$; let $\alpha$ denote the pushforward of the class $\beta$ on $T$ under the inclusion map.

Note that $\alpha$ generates a extremal ray of $\Eff_{1}(X)$.  Indeed, if $D$ is a nef divisor on $S$ such that $\beta \cdot D = 0$ then $\alpha \cdot \pi^{*}D=0$.  Since $\beta$ is isolated, such divisors $D$ generate a codimension $1$ subspace of $N_{1}(S)$.  Furthermore, $\alpha$ has vanishing intersection against the nef divisor $T' = \mathbb{P}(\mathcal{O}_{S}(-H))$.  So the nef divisors that have vanishing intersection with $\alpha$ generate a codimension $1$ subspace of $N_{2}(X)$, showing that $\alpha$ must be an isolated extremal ray.  Since $\alpha$ is not nef, and hence not movable, we have $\alpha = N(\alpha)$ is the unique Zariski decomposition for $\alpha$.

However, $\alpha$ is not the class of any effective curve.  If it were, since $\alpha \cdot T < 0$ this curve would have a component $C$ contained in $T$.  Since $\alpha$ is extremal, $[C]$ would be proportional to $\alpha$, contradicting the chosen properties of $\beta$.
%We can also find a big class $\beta$ whose negative part is not represented by an effective cycle.  Note that since $\alpha$ is an isolated extremal ray, there is some neighborhood $U$ of $\alpha$ such that no element of $\partial \Eff_{1}(X) \cap U$ is represented by an effective curve.  Using the continuity of Zariski decompositions, any sufficiently small positive perturbation of $\alpha$ will suffice.
\end{exmple}

Nevertheless, the following conjecture predicts that negative parts can be understood by analyzing proper subvarieties of $X$.  Note that when $\alpha$ is a divisor class the conjecture exactly predicts the effectiveness of the negative part: if $Y$ is a scheme of dimension $\leq n-1$, the only pseudo-effective classes in $N_{n-1}(Y)$ are sums of the classes of components of $Y$ and thus are effective classes.

\begin{conj} \label{negativepartispushedforward}
Let $X$ be a projective variety and let $\alpha = P(\alpha) + N(\alpha)$ be a Zariski decomposition.  Then there is some proper closed subscheme $i: Y \subsetneq X$ and a pseudo-effective class $\beta \in N_{k}(Y)$ such that $N(\alpha) = i_{*}\beta$.
\end{conj}

We prove (a slightly weaker version of) Conjecture \ref{negativepartispushedforward} for curve classes.

\begin{prop}
Let $X$ be a projective variety and let $\alpha \in \Eff_{1}(X)$.  Then for some Zariski decomposition $\alpha = P(\alpha) + N(\alpha)$ there is a proper closed subscheme $i: Y \subsetneq X$ and a pseudo-effective class $\beta \in N_{1}(Y)$ such that $N(\alpha) = i_{*}\beta$.
\end{prop}

\begin{proof}
Any $\alpha$ has a Zariski decomposition such that $N(\alpha)$ is a sum of classes on extremal rays of $\Eff_{1}(X)$ that do not lie in the movable cone.  Thus it suffices to show the statement when $\alpha = N(\alpha)$ generates an extremal ray of $\Eff_{1}(X)$ but is not movable.

We first reduce to the case when $X$ is smooth.  Let $\pi: X' \to X$ be a nonsingular alteration; by Proposition \ref{prop:surjeff} and Corollary \ref{cor:surj mov} there is a $\beta \in \Eff_{1}(X')$ lying on an extremal ray not contained in the movable cone that satisfies $\pi_{*}\beta = \alpha$.  Since the desired statement pushes forward under alterations, we may replace $X$ by $X'$ to assume that $X$ is smooth.  In this situation we have $\Mov_{1}(X) = \Nef_{1}(X)$ by \cite[Theorem 0.2]{bdpp04} (see Section \ref{divisorcharpsection} for the extension to characteristic $p$).

Now, let $D$ be an effective divisor such that $\alpha \cdot D  < 0$.  By removing some components of $D$ we may assume that $D$ is irreducible reduced. Set $\epsilon = -\alpha \cdot D$ and let $i: D \to X$ denote the inclusion.  We show that $\alpha$ is the pushforward of a pseudo-effective curve class on $D$.

Let $\{ C_{j} \}_{j=1}^{\infty}$ be a sequence of effective $\mathbb{R}$-cycles such that $\lim_{j \to \infty} [C_{j}] = \alpha$.  For $j$ sufficiently large we have $C_{j} \cdot D < -\epsilon/2$.  For such $j$ there is an effective cycle $C_{j}' \leq C_{j}$ with support contained in $D$ satisfying $C_{j}' \cdot D < -\epsilon/2$.

Fix an ample divisor $H$ on $X$; since the intersections $|C_{j}' \cdot D|$ have a positive lower bound, so do the intersections $C_{j}' \cdot H$.  For each $i$, let $\gamma_{j}$ denote the class of $C_{j}'$ in $N_{1}(D)$.  Since $\gamma_{j} \cdot i^{*}H$ is bounded above and below, the classes $\gamma_{j}$ have a non-zero accumulation point $\beta \in \Eff_{1}(D)$.  Then $i_{*}\beta$ is a non-zero pseudo-effective class satisfying $i_{*}\beta \preceq \alpha$.  Since $\alpha$ is on an extremal ray, this is only possible if $i_{*}\beta$ is proportional to $\alpha$.  Thus, after rescaling $\beta$ we have $i_{*}\beta = \alpha$.
\end{proof}

\begin{rmk} \label{rem:pushforwardforcurves} We are not sure if the proposition holds for \emph{every} decomposition of $\alpha$.  However, the situation improves if we slightly alter the definition of the Zariski decomposition: suppose we include the condition that $P(\alpha)$ has maximal degree with respect to a fixed polarization on $X$.  Then \emph{any} negative part $N(\alpha)$ will be a sum of classes on extremal rays that are not contained in the movable cone, and so the conclusion of the proposition will always hold. See Section \ref{minlengthsection} for a brief discussion of a similar approach for cycle classes of arbitrary codimension.
\end{rmk}

\section{Birational maps and Zariski decompositions} \label{birmapsection}

One of the main applications of the $\sigma$-decomposition is to understand the exceptional loci of morphisms.  In this section we use Zariski decompositions to understand the behavior of higher codimension cycles under birational maps.

Suppose that $\pi: Y \to X$ is a birational morphism of smooth varieties.  For any pseudo-effective divisor $D$ on $X$ and any $\pi$-exceptional effective divisor $E$ we have $P_{\sigma}(\pi^{*}D + E) = P_{\sigma}(\pi^{*}D)$.  Corollary \ref{excepstabilityofzardecom} shows that the positive parts of arbitrary classes satisfy a similar stabilization property.  We call this ``stable positive part'' a movable transform.  When $X$ is not smooth, we must instead work with an ``asymptotic strict transform'' construction.  An important consequence is that for any $\alpha \in \Eff_{k}(X)$, there is a preimage $\beta \in \Eff_{k}(Y)$ with the same mobility.

We finish this section with a discussion of Fujita approximations for higher codimension cycles.  In \cite{fujita94} Fujita shows that for complex projective varieties the volume of a big divisor $L$ can be approximated by the volume of an ample divisor $A$ on a birational model $\phi: Y \to X$ satifying $\phi_{*}A \preceq L$.  It is natural to wonder whether an analogous statement holds true for the mobility.  Proposition \ref{movbirbpf} suggests we should phrase the problem using the basepoint free cone.

\begin{ques} \label{fujitaapproxques1}
Let $X$ be a smooth projective variety and let $\alpha \in N_{k}(X)$ be a big class.  Choose $\epsilon>0$; is there a birational model $\pi: Y \to X$ and a class $\beta \in \bpf_{k}(Y)$ such that $\mob(\beta) > \mob(\alpha) - \epsilon$ and $\pi_{*}\beta \preceq \alpha$?
\end{ques}

There is an equivalent formulation of Fujita's approximation theorem for divisors in terms of the $\sigma$-decomposition.  Let $L$ be a big divisor on $X$; for suitably chosen birational models $\pi: Y \to X$, the movable transform $P_{\sigma}(\phi^{*}L)$ gets ``closer and closer'' to the nef cone.  Again, the correct analogue for cycles should use the basepoint free cone:  can we ensure that a movable transform becomes ``closer to basepoint free'' as we pass to higher and higher models?

\begin{ques} \label{fujitaapproxques2}
Let $X$ be a smooth projective variety and let $\alpha \in N_{k}(X)$ be a big class.  Fix a class $\xi$ in the interior of $\bpf_{k}(X)$.  Choose $\epsilon>0$; is there a birational model $\pi: Y \to X$, a movable transform $P$ of $\alpha$, and an element $\beta \in \bpf_{k}(Y)$ such that $P \preceq \beta + \epsilon \pi^{*} \xi$?
\end{ques}

These two versions of Fujita approximation are more-or-less equivalent.  We show that both questions have an affirmative answer for curve classes on smooth projective varieties over $\mathbb{C}$.

\begin{rmk}
One can also pose Questions \ref{fujitaapproxques1} and \ref{fujitaapproxques2} using other positive cones of cycles.  Some possibilities are the complete intersection cone and the pliant cone as defined in \cite{fl13}.
\end{rmk}

\subsection{The movable cone and birational maps}

\begin{prop}\label{birmovpush}
Let $\pi: Y \to X$ be a birational map of projective varieties and let $\alpha \in \Mov_{k}(X)$.  If $\pi_{*}\alpha=0$ then $\alpha = 0$.
\end{prop}

\begin{proof}
Fix an ample divisor $H$ on $X$.  Note that if $\pi_{*}\alpha=0$ then $\alpha \cdot \pi^{*}H^{k} = 0$.  By Lemma \ref{lem:movint}.\ref{lem:movintnef} we have that $\alpha \cdot \pi^{*}H^{k-1}$ is movable; thus Lemma \ref{lem:movint}.\ref{cor:bigandnefcor} shows that $\alpha \cdot \pi^{*}H^{k-1} = 0$.  Repeating this argument inductively we find that $\alpha = 0$.
\end{proof}

\begin{cor}
Let $\pi: Y \to X$ be a birational map of projective varieties and let $\alpha \in \Eff_{k}(Y)$.  If $\pi_{*}\alpha = 0$ then $\alpha = N(\alpha)$ is the unique Zariski decomposition for $\alpha$.
\end{cor}

\begin{cor} \label{compactpushforwards}
Let $\pi: Y \to X$ be a birational morphism of projective varieties.  Fix a class $\alpha \in \Eff_{k}(X)$.  Then the set of classes $S := \{\beta \in \Mov_{k}(Y) | \alpha \succeq \pi_{*}\beta \}$ is compact.
\end{cor}

\begin{proof}
Clearly $S$ is closed, so it suffices to prove that $S$ is bounded.  Let $M$ denote $\ker(\pi_{*})$.  By Proposition \ref{birmovpush} we know that $M \cap \Mov_{k}(Y) = \{ 0 \}$.  Let $L \in N^{k}(Y)$ be a linear function that vanishes on $M$, is non-negative on $\Eff_{k}(Y)$, and is positive on $\Mov_{k}(Y) - \{ 0 \}$.  Let $\beta \in N_{k}(Y)$ be any class that pushes forward to $\alpha$; then $S$ is contained in the compact subset of $\Mov_{k}(Y)$ consisting of classes $\xi$ satisfying $L(\xi) \leq L(\beta)$.
\end{proof}

The following question would further elucidate the structure of the $\pi$-exceptional classes.

\begin{ques} \label{locpolypushforward}
Let $\pi: Y \to X$ be a birational map of projective varieties.  Let $M = \ker(\pi_{*}) \cap \Eff_{k}(Y)$.  Is there a cone $C \subset \Eff_{k}(Y)$ disjoint from $M - \{0\}$ such that $\Eff_{k}(Y) = M + C$?
\end{ques}

\subsection{Mobility and birational maps} \label{movablepullbacksection}

Suppose that $\pi: Y \to X$ is a birational map of projective varieties. \cite[Lemma 6.8]{lehmann13} shows that for any $\beta \in \Eff_{k}(Y)$ we have $\mob(\beta) \leq \mob(\pi_{*}\beta)$.  In particular, for any $\alpha \in \Eff_{k}(X)$
\begin{equation*}
\sup_{\pi_{*}\beta = \alpha}\mob(\beta) \leq \mob(\alpha).
\end{equation*}
In this section we show that the supremum on the left-hand side is achieved by some class, and moreover this class has the same mobility as $\alpha$.  The following definition codifies a slightly different perspective on the problem.

\begin{defn} \label{movpullbackdefn}
Let $\pi: Y \to X$ be a morphism of projective varieties and let $\alpha \in \Eff_{k}(X)$ be a big class.  We say that $\gamma \in \Mov_{k}(Y)$ is a movable transform for $\alpha$ if
\begin{enumerate}
\item $\mob(\gamma) = \mob(\alpha)$, and
\item $\pi_{*}\gamma \preceq \alpha$.
\end{enumerate}
\end{defn}

\noindent Note that $\pi_{*}\gamma$ is the positive part of a Zariski decomposition for $\alpha$.  Furthermore, the existence of a movable transform immediately implies the existence of a preimage $\beta \in \pi_{*}^{-1}\alpha$ such that $\mob(\beta) = \mob(\alpha)$.  We discuss the uniqueness of movable transforms in Remark \ref{movpullbackuniqueremark}.

\begin{exmple}\label{ex:movtransformdiv}
Suppose that $\pi: Y \to X$ is a birational morphism of smooth projective varieties. For a big Cartier divisor $L$ on $X$, the (unique) movable transform of $[L]$ is $[P_{\sigma}(\pi^{*}L)]$. (From $\mob([P_{\sigma}(\pi^{*}L)])=\vol(P_{\sigma}(\pi^{*}L))=\vol(\pi^*L)=\vol(L)=\mob([L])$ we see that $[P_{\sigma}(\pi^{*}L)]$ is indeed a movable transform. The uniqueness is a consequence of Proposition \ref{divisorzardecomprop} and Remark \ref{movpullbackuniqueremark}.)
\end{exmple}

\begin{lem} \label{stricttransformlem}
Let $\pi: Y \to X$ be a birational morphism of projective varieties.  For any class $\alpha \in N_{k}(X)_{\mathbb{Z}}$, there is a movable class $\beta \in N_{k}(Y)_{\mathbb{Z}}$ such that $\mc(\beta) \geq \mc(\alpha)$ and $\alpha-\pi_{*}\beta$ is the class of an effective $\mathbb{Z}$-cycle.
\end{lem}

\begin{proof}
Let $p: U \to V$ be a family representing $\alpha$ of maximal mobility count.  Let $p': U' \to V$ be the subfamily consisting of components that dominate $X$.  By \cite[Lemma 6.8]{lehmann13} the strict transform family $q$ of $p'$ (defined in Construction \ref{stricttransformconstr}) satisfies $\mc(q) = \mc(p') = \mc(p)$.  Since the class of the cycle-theoretic fibers of $q$ is movable, we can set $\beta$ to be this class.
\end{proof}

\begin{prop} \label{existenceofmovpull}
Let $\pi: Y \to X$ be a birational map of projective varieties.  Any big class $\alpha \in N_{k}(X)$ admits a movable transform $\beta \in N_{k}(Y)$.
\end{prop}

\begin{proof}
We first consider the case when $\alpha \in N_{k}(X)_{\mathbb{Z}}$.  Choose an acceptable cone $\mathcal{K}$ in $N_{k}(Y)$ and fix a class $\gamma$ in the interior of $\mathcal{K}$.  Choose a subset $\mathcal{M} \subset \mathbb{Z}$ such that
\begin{equation*}
\mob(\alpha) = \lim_{m \to \infty, m \in \mathcal{M}} \frac{\mc(m\alpha)}{m^{\frac{n}{n-k}}/n!}.
\end{equation*}
For each $m \in \mathcal{M}$, let $p_{m}: U_{m} \to W_{m}$ denote the movable components of a family of maximal mobility count for $m\alpha$.  Let $\alpha_{m}$ denote the class of $p_{m}$; note that $m\alpha \succeq \alpha_{m}$.  Let $\beta_{m}$ be the class of the strict transform family of $p_{m}$ on $Y$.  By Corollary \ref{compactpushforwards}, the classes $\frac{1}{m}\beta_{m}$ are contained in a compact subset of $\Mov_{k}(Y)$.  Let $\beta$ be an accumulation point; note that $\beta \in \Mov_{k}(Y)$ and $\pi_{*}\beta \preceq \alpha$.  Fix $\epsilon > 0$.   There are infinitely many values of $m \in \mathcal{M}$ such that
\begin{align*}
\mc_{\mathcal{K}}(m\beta + m \epsilon \gamma) & \geq \mc_{\mathcal{K}}(\beta_{m}) \\
& \geq \mc(\beta_{m}) \\
& \geq \mc(\alpha_{m}) \textrm{ by Lemma \ref{stricttransformlem}} \\
& \geq \mc(m\alpha)
\end{align*}
Dividing by $m^{n/n-k}/n!$ and taking limits over integers in $\mathcal{M}$, Proposition \ref{kmobismob} implies
\begin{align*}
\mob(\beta + \epsilon \gamma) \geq \mob(\alpha).
\end{align*}
Letting $\epsilon$ tend to $0$ and using the continuity of $\mob$, we see that $\mob(\beta) \geq \mob(\alpha)$.  The reverse inequality follows from \cite[Lemma 6.8]{lehmann13}.

The case when $\alpha \in N_{k}(X)_{\mathbb{Q}}$ follows by rescaling.  For arbitrary $\alpha \in N_{k}(X)$, let $\{ \alpha_{i} \}$ be a sequence in $N_{k}(X)_{\mathbb{Q}}$ tending to $\alpha$ and let $\{ \beta_{i} \}$ be a sequence of movable transforms in $N_{k}(Y)$.  Then any accumulation point $\beta$ for the $\beta_{i}$ is a movable transform for $\alpha$.
\end{proof}

An important consequence is the ability to find Zariski decompositions that are stable upon adding pseudo-effective $\pi$-exceptional classes.

\begin{cor} \label{excepstabilityofzardecom}
Let $\pi: Y \to X$ be a birational map of projective varieties.  Suppose $\alpha \in \Eff_{k}(X)$ is a big class and let $\beta \in \Eff_{k}(Y)$ be a movable transform for $\alpha$.  For any class $\gamma \in \ker(\pi_{*}) \cap \Eff_{k}(Y)$ we have a Zariski decomposition of $\beta + \gamma$ with $P(\beta + \gamma) = \beta$.
\end{cor}

\begin{proof}
By definition of a movable transform $\beta$ has the maximal mobility among all $\beta' \in \Eff_{k}(Y)$ such that $\pi_{*}\beta' \preceq \alpha$.
\end{proof}

\begin{rmk} \label{movpullbackuniqueremark}
Let $X$ be a smooth projective variety over $\mathbb{C}$ and let $\pi: Y \to X$ be a birational map.  In this situation one can show that $\ker(\pi_{*})$ is spanned by $\pi$-exceptional effective classes (see \cite[Proposition 3.8]{fl14}).  In particular, any two preimages $\beta_{1}$ and $\beta_{2}$ of $\alpha$ are dominated by a third preimage $\beta_{3}$ under the relation $\preceq$.

Let $\alpha$ be a big class on $X$.  By Corollary \ref{compactpushforwards} all the movable transforms for $\alpha$ are contained in a compact set.  Using Proposition \ref{prop:surjeff} and the previous paragraph, we see that for some sufficiently large $\pi$-exceptional effective class $\gamma$ every movable transform $\beta$ for $\alpha$ satisfies $\beta \preceq \pi^{*}\alpha + \gamma$.  Then the movable transforms for $\alpha$ are precisely the positive parts of Zariski decompositions for $\pi^{*}\alpha + \gamma$.  In particular, if Zariski decompositions of big classes are unique on $Y$ then $\alpha$ admits a unique movable transform.
\end{rmk}

We can now give a more explicit description of movable transforms for curve classes.  Proposition \ref{movpullbackforcurves} indicates that ``morally''  the usual pullback coincides with the movable transform for a movable curve class.  Note however that the actual statement is slightly weaker: even if $\alpha$ is movable, a priori we may need to pass to a different positive part $P(\alpha)$ to be able to apply the statement (although this seems unlikely to be necessary).

\begin{lem} \label{stricttransformcurveslem}
Let $\pi: Y \to X$ be a birational morphism of smooth varieties.  Let $p$ be a strictly movable family of curves on $X$ with class $\alpha$ and let $\beta \in N_{k}(Y)$ denote the class of the strict transform of a general member of $p$.  Then $\pi^{*}\alpha \succeq \beta$.
\end{lem}

\begin{proof}
Let $D$ be a nef divisor on $Y$.  The divisor $D + \pi^{*}\pi_{*}(-D)$ is $\pi$-nef and $\pi$-exceptional.  By the Negativity of Contraction Lemma (which holds in any characteristic), there is an effective $\pi$-exceptional divisor $E$ such that $-E = D + \pi^{*}\pi_{*}(-D)$.  By the projection formula
\begin{align*}
D \cdot (\pi^{*}\alpha - \beta) & = (\pi^{*}\pi_{*}D - E) \cdot (\pi^{*}\alpha - \beta) \\
& = \pi_{*}D \cdot \alpha - \pi_{*}D \cdot \pi_{*}\beta + E \cdot \beta \\
& = E \cdot \beta. 
\end{align*}
Since $\beta$ is the class of the strict transform of a curve not contained in a $\pi$-exceptional center, $E \cdot \beta \geq 0$.  Since $D$ was an arbitrary nef divisor, we have obtained the claim.
\end{proof}

\begin{prop}  \label{movpullbackforcurves}
Let $\pi: Y \to X$ be a birational morphism of smooth varieties.  For any big class $\alpha \in \Eff_{1}(X)$ there is some positive part $P(\alpha)$ such that $\mob(\pi^{*}P(\alpha)) = \mob(\alpha)$.  In particular $\pi^{*}P(\alpha)$ is a movable transform of $\alpha$.
\end{prop}

\begin{proof}
By continuity and homogeneity it suffices to consider the case when $\alpha \in N_{k}(X)_{\mathbb{Z}}$.

We use the same setup and notation as in the proof of Proposition \ref{existenceofmovpull}.  There we constructed a movable transform $\beta$ for $\alpha$ as an accumulation point of the strict transform classes $\frac{1}{m}\beta_{m}$.  Let $\mathcal{M}' \subset \mathcal{M}$ be a subset such that the corresponding classes $\frac{1}{m}\beta_{m}$ limit to $\beta$.

Let $P \in N_{k}(X)$ denote an accumulation point of the classes $\frac{1}{m}\alpha_{m}$ for $m \in \mathcal{M}'$.  After again shrinking $\mathcal{M}'$, we may assume that $P$ is a limit of the classes $\frac{1}{m}\alpha_{m}$.  Lemma \ref{stricttransformcurveslem} and Remark \ref{strict transform strictly movable} indicate that for each $m \in \mathcal{M}'$ we have $\beta_{m} \preceq \pi^{*}\alpha_{m}$.  Thus $\beta \preceq \pi^{*}P$.  Since
\begin{equation*}
\mob(\alpha) = \mob(\beta) \leq \mob(\pi^{*}P) \leq \mob(P) \leq \mob(\alpha)
\end{equation*}
we see that $P$ satisfies the desired property.
\end{proof}

\begin{cor} \label{stablemovpullbackforcurves}
Let $X$ be a smooth variety.  For any big class $\alpha \in \Eff_{1}(X)$ there is some positive part $P(\alpha)$ such that $\mob(\pi^{*}P(\alpha)) = \mob(\alpha)$ for any birational morphism $\pi: Y \to X$ with $Y$ smooth.
\end{cor}

\begin{proof}
Let $\mathcal{P}$ denote the set of possible positive parts for $\alpha$.  $\mathcal{P}$ is compact by Lemma \ref{structureofnegativeparts}.  For each birational morphism $\pi: Y \to X$ let $\mathcal{P}_{\pi}$ denote the set of positive parts of $\alpha$ satisfying the stated condition for $\pi$.  Each $\mathcal{P}_{\pi}$ is closed and is non-empty by Proposition \ref{movpullbackforcurves}.  Furthermore for any finite set $\{ \pi_{i} \}$ the intersection $\cap_{i} \mathcal{P}_{\pi_{i}}$ is non-empty since it contains $\mathcal{P}_{\psi}$ for any $\psi$ that simultaneously lies above all the $\pi_{i}$.  Since $\mathcal{P}$ is compact, the common intersection $\cap_{\pi} \mathcal{P}_{\pi}$ is non-empty.
\end{proof}

It would be interesting to find other examples of ``positive'' classes whose movable transform coincides with the pullback.  The following conjecture identifies some possibilities.

\begin{conj} \label{conj:bpfpullback}
Let $X$ be a smooth variety and let $\alpha \in \bpf_{k}(X)$.  Then for any birational morphism $\pi: Y \to X$ the class $\pi^{*}\alpha$ is a movable transform for $\alpha$.
\end{conj}

\noindent As evidence for this conjecture, we prove a statement of this form for varieties over $\mathbb{C}$ admitting a transitive action by a group variety.  For such varieties every pseudo-effective class is basepoint free; one can see this by constructing families using the group action.   Proposition \ref{prop:homvarpullback} shows that the pullback of these pseudo-effective classes  is always a movable transform.

As usual we let $\pi_{*}^{-1}$ denote the strict transform operation on cycles.

\begin{lem} \label{lem:homblowup}
Let $X$ be a projective variety over $\mathbb C$ acted upon transitively by a group
variety $G$ and let $\pi: Y \to X$ be a birational map of varieties.  For any effective $k$-cycle $Z$, the cycle $gZ$ satisfies $\pi^{*}[gZ] = [\pi_{*}^{-1}gZ]$ for general $g \in G$. 
\end{lem}

\begin{proof}
By the projection formula, we can assume that $\pi$ is a sequence
of blow-ups $\pi_i:X_{i+1}\to X_i$ with smooth centers $Y_i\subsetneq X_i$.
By Kleiman's lemma, for general $g\in G$, the subset $gZ\times_XY_i\subset X_i$ is either
empty or has the expected dimension. Consequently the strict transform
of $gZ$ meets $Y_i$ properly. We conclude by induction using
\cite[Corollary 6.7.2]{fulton84}.
 \end{proof}

%Since basepoint freeness is preserved by pullback by \cite{fl13}, we have:

\begin{cor} \label{cor:homblowuppreservesmov}
Let $X$ be a projective variety over $\mathbb C$ acted upon transitively by a group
variety $G$ and let $\pi: Y \to X$ be a birational map.  Then $\pi^{*}\Eff_{k}(X) \subset \Mov_{k}(Y)$. 
 \end{cor}
 
\begin{proof}
It suffices to prove that the pullback of the class of an effective $k$-$\mathbb{Z}$-cycle on $X$ is movable.  Consider the family of $k$-$\mathbb{Z}$-cycles $p: U \to G$ constructed by acting on $Z$ by the group $G$.  Note that every component of this family dominates $X$.  Thus, the strict transform family of $p$ also satisfies this property.  Lemma \ref{lem:homblowup} and Remark \ref{strict transform strictly movable} together show that the pullback $\pi^{*}[Z]$ is the class of this strict transform family, showing that $\pi^{*}[Z]$ is movable.
\end{proof}

\begin{prop} \label{prop:homvarpullback}
Let $X$ be a projective variety over $\mathbb C$ acted upon transitively by a group
variety $G$ and let $\alpha \in \Eff_{k}(X)$ be a big class.  For any birational map $\pi: Y \to X$ the class $\pi^{*}\alpha$ is a movable transform for $\alpha$.
\end{prop}

\begin{proof}
Suppose $\alpha' \in N_{k}(X)_{\mathbb{Z}}$ is an effective class.  Let $p: U \to W$ be a family of $\mathbb{Z}$-cycles representing $\alpha'$ of maximal mobility count.  Replace $p$ by the family $p'$ obtained by acting on $p$ by a general element of $G$; note that this does not affect the mobility count.  By Lemma \ref{lem:homblowup}, the strict transform family of $p'$ has the same class as $\pi^{*}\alpha'$.

By applying Corollary \ref{cor:homblowuppreservesmov}, we may proceed just as in the proof of Proposition \ref{movpullbackforcurves} to prove that for any birational map $\pi: Y \to X$ from a smooth variety $Y$ there is a positive part $P(\alpha)$ for $\alpha$ such that $\pi^{*}P(\alpha)$ is a movable pullback for $\alpha$.   Since pulling back from $X$ preserves pseudo-effectiveness by Corollary \ref{cor:homblowuppreservesmov}, $\pi^{*}\alpha \succeq \pi^{*}P(\alpha)$.  Thus $\mob(\pi^{*}\alpha) = \mob(\alpha)$ and by Corollary \ref{cor:homblowuppreservesmov} the class $\pi^{*}\alpha$ is movable.  Since this is true for any smooth model $Y$, it is also true for any birational model admitting a morphism to $X$.
\end{proof}

\subsection{Fujita approximations for curves}

We next prove a Fujita approximation-type result for curve classes on smooth varieties over $\mathbb{C}$.  The essential tools are the orthogonality theorem of \cite{bdpp04} and the $\sigma$-decomposition.

\begin{lem} \label{positiveproductandcurvepullback}
Let $X$ be a smooth projective variety over $\mathbb{C}$ of dimension $n$ and let $L$ be a big divisor on $X$.  Suppose that $\pi: Y \to X$ is a birational morphism from a smooth projective variety $Y$.  Then $\pi^{*}\langle L^{n-1} \rangle = \langle \pi^{*}L^{n-1} \rangle$.
\end{lem}

\begin{proof}
It is an easy consequence of the definition that $\pi_{*}\langle \pi^{*}L^{n-1} \rangle = \langle L^{n-1} \rangle$, so it suffices to check that $\langle \pi^{*}L^{n-1} \rangle \cdot E = 0$ for any irreducible $\pi$-exceptional divisor $E$.  Note that for any $\epsilon > 0$ the restriction of $\pi^{*}L + \epsilon E$ to $E$ is not pseudo-effective.  Thus, $\pi^{*}L$ is not $E$-big in the sense of \cite[Definition 4.1]{bfj09}.  Then \cite[Lemma 4.10]{bfj09} yields the claim.
\end{proof}

\begin{thrm} \label{pospartbecomesbpfforcurves}
Let $X$ be a smooth projective variety over $\mathbb{C}$ and let $\alpha$ be a big movable curve class.  Fix a curve class $\xi$ in the interior of $\Mov_{1}(X)$.  Then for any $\mu > 0$, there is a birational morphism $\phi: Y \to X$ and a curve class $\beta \in \bpf_{1}(Y)$ such that
\begin{equation*}
\phi^{*}(\alpha - \mu \xi) \preceq \beta \preceq \phi^{*}(\alpha+\mu \xi).
\end{equation*}
\end{thrm}

\begin{rmk}
\cite{bdpp04} shows that $\Mov_{1}(X)$ is the closure of the cone generated by the pushforward of basepoint free classes.  The point of Theorem \ref{pospartbecomesbpfforcurves} is to show that the basepoint free classes constructed in this way are also ``almost pullbacks.''
\end{rmk}

The proof uses the following statement.

\begin{prop}[\cite{lehmann12}, Proposition 3.7] \label{psigmaapproachesnef}
Let $X$ be a smooth projective variety over $\mathbb{C}$ and let $L$ be a big divisor with $L \geq 0$.  Then there is an effective divisor $G$ so that for any sufficiently large $m$ there is a smooth birational model $\phi_{m}: Y_{m} \to X$ and a big and nef divisor $B_{m}$ on $Y_{m}$ with
\begin{equation*}
P_{\sigma}(\phi_{m}^{*}L) - \frac{1}{m}\phi_{m}^{*}G \leq B_{m} \leq P_{\sigma}(\phi_{m}^{*}L).
\end{equation*}
\end{prop}

\begin{proof}[Proof of Theorem \ref{pospartbecomesbpfforcurves}:]
Let $n$ denote the dimension of $X$.  First consider the case when there is a big divisor $L$ such that $\langle L^{n-1} \rangle = \alpha$.  By replacing $L$ by a numerically equivalent divisor we may assume that $L$ is effective.  Apply Proposition \ref{psigmaapproachesnef} to $L$ to obtain morphisms $\phi_{m}$ and a divisor $G$.  Note that we may make $G$ more effective without changing the conclusion of the proposition; in particular, we may assume that $G$ is effective and ample.

Suppose that $m$ is sufficiently large so that $L - \frac{1}{m}G$ is pseudo-effective.  Since $\phi_{m}^{*}G$ is big and nef, we have an inequality $P_{\sigma}(\phi_{m}^{*}L) - \frac{1}{m}\phi_{m}^{*}G \geq P_{\sigma}(\phi_{m}^{*}(L - \frac{1}{m}G))$.  In particular
\begin{align*}
\phi_{m}^{*} \left\langle \left(L - \frac{1}{m}G \right)^{n-1} \right\rangle  & = 
\left\langle \phi_{m}^{*}\left(L - \frac{1}{m}G \right)^{n-1} \right\rangle \textrm{ by Lemma \ref{positiveproductandcurvepullback}} \\
& = \left\langle P_{\sigma}\left(\phi_{m}^{*}\left(L - \frac{1}{m}G \right)\right)^{n-1} \right\rangle \\
& \preceq B_{m}^{n-1}
\end{align*}
where in the second step we have applied the invariance of the positive product under replacing a pseudo-effective divisor by its positive part (see \cite[Proposition 4.13]{lehmann12}) and in the last step we have applied the super-additivity of the positive product (see \cite[Proposition 2.9]{bfj09}).
Similarly $B_{m}^{n-1} \preceq \langle P_{\sigma}(\phi_{m}^{*}L)^{n-1} \rangle = \phi_{m}^{*} \langle L^{n-1} \rangle$.

By choosing $m$ sufficiently large, we may ensure that $\mu \xi - (\alpha - \langle (L - \frac{1}{m}G)^{n-1} \rangle)$ is movable by the continuity of the positive product.  Since the pullback of a movable class is again movable, this implies that $\phi_{m}^{*}(\alpha-\mu \xi) \preceq B_{m}^{n-1} \preceq \phi_{m}^{*}\alpha$.  We conclude this case by noting that $B_{m}^{n-1}$ is a baspoint free class, since it is a limit of complete intersections of ample divisors.  

Note an important feature of this construction: we may replace $\phi_{m}: Y_{m} \to X$ by any higher birational model and may pullback $B_{m}$ without changing the validity of the argument.  Indeed, the only fact about $Y_{m}$ that we used is the inequality from Proposition \ref{psigmaapproachesnef}.  But this inequality remains true on higher birational models: since $\pi^{*}P_{\sigma}(L) \geq P_{\sigma}(\pi^{*}L)$ for any birational map $\pi$, the inequality on the left is preserved.  \cite[III.1.14 Proposition]{nakayama04} shows the the inequality on the right is also preserved.

Next suppose that $\alpha$ lies in the interior of $\Mov_{1}(X)$.  Note that by \cite[0.2 Theorem]{bdpp04} $\Mov_{1}(X)$ is the closure of the cone generated by classes of the form $\langle L^{n-1} \rangle$ for big divisors $L$.  Thus $\alpha$ can be written as a finite sum of classes $\{ \gamma_{i} \}_{i=1}^{r}$ of the form $\gamma_{i} = \langle L_{i}^{n-1} \rangle$ for big divisors $L_{i}$.  For each $i$, we can apply the above construction to $\gamma_{i}$ with the class $\xi$ and the constant $\frac{1}{r}\mu$.  As remarked above, the construction actually can be done on any higher birational model; in particular, we can find one smooth model $\pi: Y \to X$ so that the construction for each $\gamma_{i}$ can be carried out on $Y$.  The result is for each $i$ a class $\beta_{i} \in \bpf_{1}(Y)$ such that
\begin{equation*}
\pi^{*}(\gamma_{i}- (\mu/r) \xi ) \preceq \beta_{i}  \preceq \pi^{*}\gamma_{i}.
\end{equation*}
Set $\beta = \sum_{i}\beta_{i}$; since $\bpf_{1}(Y)$ is convex by definition, $\beta$ satisfies the desired properties.

Finally, we prove the statement for an arbitrary big and movable class $\alpha$.  Note that $\alpha + \frac{\mu}{2}\xi$ is in the interior of $\Mov_{1}(X)$.  Thus, we can apply the statement for this class with the constant $\mu/2$ to find a birational model $\pi: Y \to X$ and a class $\beta \in \bpf_{1}(Y)$ such that
\begin{equation*}
\pi^{*}\alpha \preceq \beta \preceq \pi^{*}(\alpha + \mu \xi).
\end{equation*}
\end{proof}

\begin{thrm} \label{fujapproxcurves}
Let $X$ be a smooth projective variety over $\mathbb{C}$ and let $\alpha$ be a big curve class.  For any $\epsilon > 0$, there is a birational map $\pi: Y \to X$ and an element $\beta \in \bpf_{1}(Y)$ such that $\mob(\beta) > \mob(\alpha) - \epsilon$ and $\pi_{*}\beta \preceq \alpha$.
\end{thrm}

\begin{proof}
Fix a class $\xi$ in the interior of $\Mov_{1}(X)$ and a $\delta>0$ sufficiently small so that $\alpha - \delta \xi$ is still big.  Choose a positive part $P$ for $\alpha - \delta \xi$ as in Corollary \ref{stablemovpullbackforcurves}.  Fix a positive constant $\mu < \delta/2$ and apply Theorem \ref{pospartbecomesbpfforcurves} to $P + \mu \xi$ with the class $\xi$ and the constant $\mu$.  The result is a birational model $\pi: Y \to X$ and a curve class $\beta \in \bpf_{1}(Y)$ such that $\pi^{*}P \preceq \beta$ and $\pi_{*}\beta \preceq P + 2\mu \xi$.  Since $2\mu < \delta$ we have
\begin{equation*}
\pi_{*}\beta \preceq P + \delta \xi \preceq \alpha - \delta \xi + \delta \xi = \alpha.
\end{equation*}

By construction $\mob(\alpha - \delta \xi) = \mob(\pi^{*}P) \leq \mob(\beta)$.  By shrinking $\delta$ and using the continuity of $\mob$, we can ensure that $\mob(\beta)$ is arbitrarily close to $\mob(\alpha)$.
\end{proof}

\begin{thrm}
Let $X$ be a smooth projective variety over $\mathbb{C}$ and let $\alpha$ be a big curve class.  Fix a class $\xi$ in the interior of $\bpf_{1}(X)$.  For any $\epsilon>0$, there is a birational model $\pi: Y \to X$, a movable transform $P$ of $\alpha$, and an element $\beta \in \bpf_{1}(Y)$ such that $P \preceq \beta + \epsilon \pi^{*} \xi$.
\end{thrm}

\begin{proof}
Choose a positive part $P(\alpha)$ for $\alpha$ as in Corollary \ref{stablemovpullbackforcurves} so that for any birational map the pullback of $P(\alpha)$ is a movable transform.  Apply Theorem \ref{pospartbecomesbpfforcurves} to $P(\alpha)$ for a class $\xi$ in the interior of $\Mov_{1}(X)$ and for a positive constant $\epsilon$ to obtain a birational model $\pi: Y \to X$ and a class $\beta \in \bpf_{1}(Y)$ with $\pi^{*}(P(\alpha))\preceq \beta + \epsilon \pi^{*}\xi$.
\end{proof}

\begin{rmk}
The previous two theorems show something stronger than claimed; namely, the $\beta$ constructed there is actually in the complete intersection cone (and hence also the pliant cone) as defined in \cite{fl13}.
\end{rmk}

\begin{rmk}
Theorem \ref{fujapproxcurves} reduces the calculation of the mobility of a big curve class of the form $\langle L^{n-1} \rangle$ for a big divisor $L$ to the special case of complete intersection classes.
\end{rmk}

\section{Examples} \label{examplesection}

In this section we give several examples of Zariski decompositions of higher codimension cycles.

\subsection{Projective bundles over curves} \label{ssec:projbundles}
Let $C$ be a smooth curve and let $E$ be a locally free sheaf of rank $n$ and degree $d$ on $C$. Let $X=\mathbb P(E)$ and let $\pi:X\to C$ be the bundle map. In \cite{fulger11}, the first author computed $\Eff_k(X)$ in terms of the numerical data in the Harder--Narasimhan filtration of $E$. In this subsection we compute the movable cones. Let 
$$E=E_0\supseteq E_1\supseteq\ldots\supseteq E_s=0$$ 
be the Harder--Narasimhan filtration of $E$ with semistable quotients $Q_i:=E_{i-1}/E_i$ of degree $d_i$, rank $r_i$, and slope $\mu_i:=d_i/r_i$. We have $\mu_1<\mu_2<\ldots<\mu_s.$

\par Let $\xi:=c_1(\mathcal O_{\mathbb P(E)}(1))$, and let $f$ be the class of a fiber of $\pi$. We have the intersection relations
$$\xi^n=d\quad ,\quad \xi^{n-1}f=1\quad ,\quad f^2=0.$$

\noindent The main result of \cite{fulger11} states that for $1\leq k\leq n-1$: 
$$\Eff_k(X)=\langle\xi^{n-k}+\epsilon_k\xi^{n-k-1}f,\ \xi^{n-k-1}f\rangle,$$
where $\epsilon_k$ is the $y$-coordinate of the point with $x$-coordinate $k$ on the Harder--Narasihman (or Shatz) polygonal curve $\mathcal P$ from $(0,-d)$ to $(n,0)$ obtained by gluing segments of $x$-length $r_i$ and slope $\mu_i$, in increasing order with $i$. We can denote $\epsilon_0=-d$ and $\epsilon_n=0$. Observe that $\epsilon_{n-1}=-\mu_s$.

\par As a corollary, we can also compute the nef cones for $1\leq k\leq n-1$:
\begin{equation}\label{eq:proj bundle nef}\Nef_k(X)=\langle\xi^{n-k}+\nu_k\xi^{n-k-1}f,\ \xi^{n-k-1}f\rangle,\end{equation}
where $\nu_k=-d-\epsilon_{n-k}$. These fit on the polygonal curve $(\frac n2,-\frac d2)-\mathcal P$.

\begin{prop}\label{prop:proj bundles movable cone}
For $1\leq k\leq n-1$, we have
$$\Mov_k(\mathbb P(E))= \langle\xi^{n-k}+\sigma_k\xi^{n-k-1}f,\ \xi^{n-k-1}f\rangle,$$
where $\sigma_k=\epsilon_{k-1}+\mu_s$. In other words, the $\sigma$'s can be read from $\mathcal P+(1,\mu_s)$.
\end{prop}

\noindent Note that $\sigma_{k} \geq \epsilon_{k}$; the difference is exactly the difference of slopes $\mu_{s} - \mu_{j}$ where $j$ is the smallest index such that $\mathrm{rk}(E_{j}) < n-k$.

\begin{proof}
The class $\xi^{n-k-1}f$ lies on the boundary of $\Eff_k(X)$, and is also an intersection of nef divisor classes (not necessarily $\xi$). It therefore lies on the boundary of $\Mov_k(X)$ for all $k$.  We let $\sigma_{k}$ denote the constant so that $\xi^{n-k} + \sigma_{k}\xi^{n-k-1}f$ gives the other boundary of $\Mov_{k}(X)$; the goal is to determine $\sigma_{k}$.

Nef and movable are equivalent properties for curves, hence $\sigma_1=\nu_1=-d-\epsilon_{n-1}$. Since $\epsilon_{n-1}=-\mu_s$, the stated formula is true when $k=1$. 

Consider the diagram:
$$\xymatrix{
{\rm Bl}_{\mathbb P(Q_1)}\mathbb P(E)\ar[d]^{B}\ar[r]^(.55){\eta}& Y:=\mathbb P(E_1)\ar[d]\\
X=\mathbb P(E)\ar[r]^(.55){\pi}& C
}$$

\noindent{\bf Claim:} If $k<n-r_1={\rm rank}(E_1)$, we show that $B_*\eta^*$ induces an isomorphism $\Mov^k(Y)\to\Mov^k(X)$ whose inverse is $$\alpha\mapsto\eta_*B^*(\alpha\cdot(\xi-\mu_1f)^{r_1}).$$

\noindent{\bf Proof of Claim:} Intuitively, $B_*\eta^*$ takes the ``cone with center $\mathbb P(Q_1)$'' over cycles on $Y$. Its inverse is the ``intersection" with $Y$. \cite[Lemma 2.7]{fulger11} verifies that the listed maps are isomorphisms at the level of numerical groups, and they restrict to isomorphisms $\Eff^k(X)\simeq\Eff^k(Y)$ when $k<n-r_1$. We check that they also preserve movable classes.

\par Since $\eta$ is flat, it is clear that $B_*\eta^*(\Mov^k(Y))\subseteq\Mov^k(X)$. Conversely, let $\alpha\in\Mov^k(X)$ be a strongly movable class. Let $U\to V$ be an irreducible family representing it and covering $X$. Let $U'$ be the unique irreducible component of $U\times_X{\rm Bl}_{\mathbb P(Q_1)}\mathbb P(E)$ that dominates $U$. 

The irreducible family $U'\to V$ represents generically a strongly movable class $\alpha'$ on the blow-up of $X$. $\alpha'$ is the class of the strict transform of a general cycle in the family $U$ by Remark \ref{strict transform strictly movable}, so $\alpha'-B^*\alpha$ is the pushforward of a class in the exceptional divisor on the blow-up. By \cite[Proposition 2.4.(iii)-(iv)]{fulger11}, any such class is canceled by multiplication with $B^*(\xi-\mu_1f)^{r_1}$, therefore
$$\eta_*B^*(\alpha\cdot(\xi-\mu_1f)^{r_1})=\eta_*(\alpha'\cdot B^*(\xi-\mu_1f)^{r_1}).$$
To see that this is a movable cycle, it is enough to observe that $\xi-\mu_1f$ is nef as demonstrated in \cite[Lemma 2.1]{fulger11}. The claim follows by continuity.

\vskip.5cm

We return to the proof of the proposition. We do induction on the number of terms $s$ in the Harder--Narasimhan filtration of $E$. When $s=1$, i.e. when $E$ is semistable, all effective cycles are generated by intersections of nef divisors by \cite[Lemma 2.2]{fulger11}.

For $s>1$, the claim and induction take care of all $\Mov_k(X)$ for $k>r_1$. To see this, observe that that polygon $\mathcal P$ associated to $E_1$ is obtained from that of $E$ by deleting the first segment, and translating to the left by $r_1$.  In particular, $\sigma_{r_1+1}=\epsilon_{r_1}+\mu_s$.

Finally, we determine $\Mov_{k}(X)$ for $1\leq k\leq r_1$.  The class $$(\xi^{n-r_1-1}+(\epsilon_{r_1}+\mu_s)\xi^{n-r_1-2}f)\cdot(\xi-\mu_1f)^{r_1-k+1}$$ is in $\Mov_k(X)$, because $\xi-\mu_1f$ is nef. Using the description of $\mathcal P$, it is straightforward that this is the same as $\xi^{n-k}+(\epsilon_{k-1}+\mu_s)\xi^{n-k-1}f.$  These imply that
$\sigma_k\leq\epsilon_{k-1}+\mu_s$. If the inequality is strict at any step, then the same intersection computation shows that it is also strict for $k=1$, which contradicts the curve computation.
\end{proof}

\begin{cor}
With notation as in the proposition, $\Mov_k(\mathbb P(E))=\Eff_k(\mathbb P(E))$ if and only if $n-r_s<k<n$. In particular, all effective divisors are movable if and only if $r_s>1$. Moreover, $\Nef_k(\mathbb P(E))=\Eff_k(\mathbb P(E))$ for $k>1$ if and only if $E$ is semistable ($s=1$), or $s=2$ and $r_2=1$.
\end{cor}

\begin{cor}
Keep notation as in the proposition.  Any pseudo-effective class $\alpha \in N_{k}(X)$ can be written as $\alpha =a (\xi^{n-k}+\epsilon_k\xi^{n-k-1}f) + b \xi^{n-k-1}$ for some $a,b \geq 0$.  If $a\epsilon_{k} + b < a\sigma_{k}$, then a Zariski decomposition for $\alpha$ is given by
\begin{equation*}
P(\alpha) = \frac{b}{\sigma_{k}-\epsilon_{k}} (\xi^{n-k}+\sigma_k\xi^{n-k-1}f) \qquad \qquad N(\alpha) = \frac{a(\sigma_{k}-\epsilon_{k}) - b}{\sigma_{k}-\epsilon_{k}} (\xi^{n-k}+\epsilon_k\xi^{n-k-1}f).
\end{equation*}
\end{cor}

\begin{proof}
This follows from the fact that the set $\{ \beta \in \Mov_{k}(X) | \beta \preceq \alpha \}$ has the unique maximal element $\frac{b}{\sigma_{k}-\epsilon_{k}} (\xi^{n-k}+\sigma_k\xi^{n-k-1}f)$ under the relation $\preceq$.
\end{proof}

\begin{exmple}Let $X=\mathbb P(E)$, where $E$ is the bundle $\mathcal O^{n-2}\oplus\mathcal O(1)^{\oplus 2}$ over $\mathbb P^1$. Using 
Proposition \ref{prop:proj bundles movable cone} and \eqref{eq:proj bundle nef}, for all $2\leq k\leq n-1$, the $k$-dimensional class $\xi^{n-k}-\xi^{n-k-1}f$ is movable, but not nef. 

In particular, these yield examples of varieties of any dimension $n>2$ and cycle classes of any dimension $1 < k \leq n-1$ that admit Zariski decompositions whose positive parts are not nef.\qed
\end{exmple}

\subsection{Cycles on $\mathbb{P}^{2[2]}$} \label{section:hilbscheme}

In order  to cohere to the cited references, in this section we will consider projective spaces of lines instead of hyperplanes. \cite{abch13} gives an extensive analysis of the divisor theory for Hilbert schemes of points on $\mathbb{P}^{2}$.  It would be interesting to see an analogous study performed for higher codimension cycles.  As a test case, we work out the geometry of cycles on $\mathbb{P}^{2[2]}:={\rm Hilb}^2\mathbb{P}^{2}$.  Since $\mathbb{P}^{2[2]}$ is a $\mathbb{P}^{2}$-bundle over $\mathbb{P}^{2}$, it is particularly amenable to computations; the basic intersection theory is worked out for example in \cite{graber01} and \cite{abch13}.  

For simplicity, let $X$ denote $\mathbb{P}^{2[2]}$.  Then $X$ is isomorphic to the projective bundle of lines $\mathbb{P}_{\mathbb{P}^{2}}(\Sym^2\Omega_{\mathbb{P}^{2}})$.  To see this, note that any degree $2$ zero-dimensional subscheme of $\mathbb{P}^{2}$ lies on a unique line; thus there is a morphism
\begin{equation*}
\pi: X \to \mathbb{P}^{2*}.
\end{equation*}
The fiber over a point $\ell$ is canonically identified with $\Sym^{2}(\ell)$; thus $X$ is the projective bundle defined by $\Sym^{2}(\mathcal{K})$ where $\mathcal{K}$ is the kernel of the tautological map $\mathcal{O}_{\mathbb{P}^{2*}}^{\oplus 3} \to \mathcal{O}_{\mathbb{P}^{2*}}(1)$.  One readily verifies that $c_{1}(\Sym^{2}(\mathcal{K})) = -3H$ and $c_{2}(\Sym^{2}(\mathcal{K})) = 6H^{2}$ where $H$ denotes the hyperplane class on $\mathbb{P}^{2*}$.  Furthermore, rational equivalence coincides with numerical equivalence.

The space of divisors $N_{3}(X)$ is $2$-dimensional.  Define the divisors:

\vskip.5cm

\begin{tabular}{ l | l | l  }
  Divisor & Geometric interpretation (of general points in cycle) & Numerical class \\ \hline
  $D_{1}$ & schemes whose associated line is incident to a fixed point & $\pi^{*}H$ \\ \hline
  $D_{2}$ &  subschemes incident to a fixed line & $\mathcal{O}_{\pi}(1)$ \\ \hline
  $E$ & non-reduced subschemes & $2(D_{2} - D_{1})$ \\
\end{tabular}

\vskip.5cm

Note that $D_{1}$ and $D_{2}$ generate the entire cohomology ring with the relations $D_{1}^{3}=0$ and $D_{2}^{3} = 3D_{1}D_{2}^{2} - 6D_{1}^{2}D_{2}$. 

\cite{abch13} verifies that the nef cone of divisors is generated by $D_{1}$ and $D_{2}$ and is equal to the movable cone.  The pseudo-effective cone is generated by $D_{1}$ and $E$; any non-movable divisor contains $E$ in its diminished base locus.  In particular, for any pseudo-effective divisor $L$, we have that $N_{\sigma}(L)$ is some multiple of $E$.  It turns out that a similar picture holds for cycles of any codimension.

\subsubsection{Curves}

$N_{1}(X)$ is also $2$-dimensional.  Define the curves:

\vskip.5cm

\begin{tabular}{ l | l | l  }
  Curve & Geometric interpretation (of general points in cycle) & Numerical class \\ \hline
  $C_{1}$ & subschemes supported at a fixed point & $D_{1}D_{2}^{2} - 3D_{1}^{2}D_{2}$ \\ \hline
  $C_{2}$ & subschemes of a fixed line and incident to a fixed point & $D_{1}^{2}D_{2}$ \\ \hline
  $B$ & subschemes incident to a fixed line and a fixed point & $C_{1} + C_{2}$ \\
\end{tabular}

\vskip.5cm

Then the movable cone of curves (and the nef cone) is generated by $B$ and $C_{2}$.  The pseudo-effective cone of curves is generated by $C_{1}$ and $C_{2}$.  If $C$ is a non-movable curve class that is effective, then it has negative intersection with the divisor $E$.  In particular, some component must be contained in $E$ and thus parametrizes non-reduced cycles.  If we write $[C] = a[C_{1}] + b[C_{2}]$ for non-negative $a$ and $b$, then $[C]$ is movable if $b \geq a$, and otherwise a Zariski decomposition for $[C]$ is
\begin{equation*}
P([C]) = b[C_{1}] + b[C_{2}] \qquad \qquad N([C]) = (a-b)[C_{1}].
\end{equation*}

\subsubsection{Surfaces}

$N_{2}(X) = \Sym^{2}N^{1}(X)$ is three-dimensional.  Define the classes:

\vskip.5cm

\begin{tabular}{ l | l | l  }
  Surface & Geometric interpretation (of general points in cycle) & Numerical class \\ \hline
  $S_{1}$ & reduced subschemes of a fixed line & $D_{1}^{2}$ \\ \hline
  $S_{2}$ & subschemes incident to a fixed point & $D_{1}D_{2} - 2D_{1}^{2}$ \\ \hline
  $S_{3}$ & $\frac{1}{2} \cdot $ (nonreduced subschemes supported on a fixed line) & $D_{2}^{2} - D_{1}D_{2}$ \\ \hline
  $M$ & image of $l_{1} \times l_{2} \dashrightarrow X$ for two distinct lines $l_{1},l_{2}$ & $S_{3} + 2S_{1}$ \\
\end{tabular}

\vskip.5cm

It is clear from the geometric description that $S_{1}$, $S_{2}$, and $S_{3}$ are all effective. The surfaces $S_1$ are the fibers of the bundle map $\pi:X\to\mathbb P^{2*}$, hence they are basepoint free. Let $p:U\to X$ be the flat universal family of $\mathbb P^{2[2]}$ with evaluation morphism $\epsilon:U\to\mathbb P^2$.
The surfaces $S_2$ can be identified with the fibers of $\epsilon$,
hence they are basepoint free.  Finally, we show that $M$ is movable and nef; in fact this is true for any irreducible surface $Z$ that intersects $X - \Supp(E)$.  Note that the natural action of $PGL_{3}$ is transitive on $X-\Supp(E)$ and on $\Supp(E)$.  In particular, we can construct a family of $\mathbb{Z}$-cycles by moving $Z$ using the group action to see that $Z$ has movable class.  Furthermore, note that $Z \cap (X - \Supp(E))$ and $Z \cap \Supp(E)$ both have codimension at least $2$ in their respective strata.  Using Kleiman's Tranversality Lemma \cite[2. Theorem.(i)]{kleiman74} for the two strata $X - \Supp(E)$ and $\Supp(E)$ we see that $Z$ is nef.

The intersection matrix can be calculated using the descriptions as products of divisors:

\vskip.5cm

\centerline{
\begin{tabular}{ l | l | l | l | l }
   & $S_{1}$ & $S_{2}$ & $S_{3}$ & $M$ \\ \hline
  $S_{1}$ & 0 & 0 & 1 & 1 \\ \hline
  $S_{2}$ & 0 & 1 & 0 & 0 \\ \hline
  $S_{3}$ & 1 & 0 & -2 & 0 \\ \hline
  $M$ & 1 & 0 & 0 & 2 \\
\end{tabular}
}
\vskip.5cm

Note that the cone generated by $S_{1}$ and $S_{2}$ has vanishing intersection with the nef class $S_{1}$ and so lies on the boundary of the pseudo-effective cone.  Arguing in a similar way for the other combinations, we see that the effective cone is generated by $S_{1}$, $S_{2}$, and $S_{3}$ and the nef cone is generated by $S_{1}$, $S_{2}$, and $M$.

We claim that the movable cone coincides with the nef cone.  Lemma \ref{lem:movint}.\ref{lem:movintnef} shows that if $\alpha$ is a movable surface class, then its intersection against any nef divisor is a movable curve class.  Since
\begin{equation*}
(aS_{1} + bS_{2} + cS_{3}) \cdot D_{2} = (b+2c)C_{1} + (a+b)C_{2} 
\end{equation*}
we see that $aS_{1} + bS_{2} + cS_{3}$ is movable only if $a \geq 2c$.  Since $M = 2S_{1} + S_{3}$ is movable, this condition is sufficient as well, proving the claim.

Suppose that the class $\alpha = aS_{1} + bS_{2} + cS_{3}$ is not movable, so $2c > a$.  The analysis above shows that
\begin{equation*}
aS_{1} + bS_{2} + cS_{3} = \left(bS_{2} + \frac{a}{2}M \right) + \left(c-\frac{a}{2} \right)S_{3}
\end{equation*}
expresses $\alpha$ as the sum of a movable and a pseudo-effective class.  Furthermore, any other expression $\alpha = \beta + \gamma$ with $\beta$ movable and $\gamma$ pseudo-effective must have $\beta \preceq bS_{2} + \frac{a}{2}M$.  Thus this expression is a Zariski decomposition.

Finally, note that $S_{3} \cdot E = 2(c_{1} - 2c_{2})$ is not pseudo-effective, so that any effective cycle of class $S_{3}$ must have some component contained in $E$.  Since $S_{3}$ is an extremal ray, in fact every effective cycle of class $S_{3}$ must be contained in $E$ (and so parametrizes non-reduced $0$-cycles).

\subsection{$\overline{M}_{0,7}^{S_{7}}$} \label{mbarsubsection}

Moduli spaces of pointed curves have a rich combinatorial structure yielding interesting results concerning the geometry of their cycles.  Since divisors and curves have been extensively studied, we focus on codimension $2$ cycles.  The first case that has not yet been worked out is $\overline{M}_{0,7}$.  However, the divisor theory is already very complicated; thus, we work instead with the simpler space $\M$.  We will only consider this space over $\mathbb{C}$.

Recall that the locus in $\overline{M}_{0,7}$ parametrizing curves with at least $(k+1)$-components has pure codimension $k$.  The irreducible components of this locus are called vital cycles.  One can identify a vital cycle by specifying the topological type of a curve parametrized by a general point: the arrangement of $\mathbb{P}^{1}$s composing the curve and the assignment of marked points to components.  Since we are interested in the $S_{7}$-invariant behavior, we will often surpress the denumeration of the marked points.  We say that the combinatorial type of a vital cycle is an identification of the arrangement of components and assignment of (unnumbered) marked points.  Note that for vital divisors or surfaces on $\M$ the only possible arrangement of $\mathbb{P}^{1}$s is a chain, so we only need to identify the distribution of points.  Thus there are two combinatorial types of vital divisors -- (2,5) and (3,4) -- and four combinatorial types of vital surfaces -- (2,2,3), (3,1,3), (2,3,2), and (2,1,4).

$\M$ is a normal $\mathbb{Q}$-factorial fourfold.  We first describe the numerical groups $N_{k}(\M)$.  The quotient $\pi: \overline{M}_{0,7} \to \M$ induces surjective pushforward maps $\pi_{*}: N_{k}(\overline{M}_{0,7}) \to N_{k}(\M)$.  Thus the pushforward maps
$\pi_{*}: N_{k}(\overline{M}_{0,7})^{S_{7}} \to N_{k}(\M)$ are also surjective.
\cite{km96} shows that $N_{3}(\M) = N^{1}(\M)$ is two-dimensional so that $\pi^{*}: N^{1}(\M) \to N^{1}(\overline{M}_{0,7})^{S_{7}}$ is an isomorphism.  The dual space $N_{1}(\M) = N^{3}(\M)$ is also $2$-dimensional.

Since $N_{2}(\overline{M}_{0,7})$ is spanned by vital surfaces, $N_{2}(\overline{M}_{0,7})^{S_{7}}$ is spanned by the orbits of the four combinatorial types of vital surface.  In fact $N_{2}(\overline{M}_{0,7})^{S_{7}}$ is three-dimensional.  To see this, note that by \cite{keel92} $N^{*}(\overline{M}_{0,7})$ is generated as a ring by vital divisor classes.  So two elements $\alpha,\beta \in N_{2}(\overline{M}_{0,7})^{S_{7}}$ are numerically equivalent if and only if for every pair of divisors $L_{1},L_{2}$ on $\overline{M}_{0,7}$ we have
\begin{equation*}
\alpha \cdot L_{1} \cdot L_{2} = \beta \cdot L_{1} \cdot L_{2}.
\end{equation*}
However, if we replace $L_{1}$ and $L_{2}$ by their $S_{7}$-orbits, both sides are multiplied by the same constant.  Thus to show the numerical equivalence of $\alpha$ and $\beta$ it suffices to intersect against $S_{7}$-invariant divisors, showing that $N_{2}(\overline{M}_{0,7})^{S_{7}}$ has dimension no larger than $\dim \Sym^{2}N^{1}(\overline{M}_{0,7})^{S_{7}} = 3$.  We will shortly see that the map $\Sym^{2}N^{1}(\M) \to N^{2}(\M)$ is an injection, so the dimension is exactly $3$.

Suppose that $\alpha \in N^{2}(\M)$ is an effective class.  There is some effective class $\beta \in N^{2}(\overline{M}_{0,7})$ pushing forward to it.  Then $\frac{1}{7!}\sum_{g \in S_{7}} g\beta$ also pushes forward to $\alpha$.  Since this class is proportional to $\pi^{*}\alpha$, we see that $\pi^{*}: N^{2}(\M) \to N^{2}(\overline{M}_{0,7})$ preserves pseudo-effectiveness.  The same is true for cycles of any dimension.

Let $D_{1}$ and $D_{2}$ denote the divisor classes on $\M$ which pull back under $\pi^{*}$ to the sum of all vital divisors of combinatorial type $(2,5)$ and of type $(3,4)$ respectively.  In summary, we have established:
\begin{itemize}
\item Each cap map $\cap: N^{k}(\M) \to N_{4-k}(\M)$ is an isomorphism.
\item The cohomology ring $N^{*}(\M)$ is generated by $D_{1}$ and $D_{2}$ and
\begin{equation*}
N^{2}(\M) = \Sym^{2}N^{1}(\M).
\end{equation*}
\item The pullbacks $\pi^{*}: N^{k}(\M) \to N^{k}(\overline{M}_{0,7})$ map isomorphically onto $N^{k}(\overline{M}_{0,7})^{S_{7}}$ and preserve pseudo-effectiveness.
\end{itemize}
Let $T_{1},T_{2},T_{3}$ be the basis for $N_{2}(\M)$ dual to the products $D_{1}^{2},D_{1}D_{2},D_{2}^{2}$ and let $C_{1}, C_{2}$ be the basis for $N_{1}(\M)$ dual to $D_{1},D_{2}$.  The entire cohomology ring is then determined by the relations
\begin{align*}
D_{1}^{2} & = 63(-27T_{1}+20T_{3}) \\
D_{1} \cdot D_{2} & = 105(12T_{2}+7T_{3}) \\
D_{2}^{2} & = 105(12T_{1}-7T_{2}+3T_{3})
\end{align*}
which can be established using the results of \cite{keel92}.

\subsubsection{Divisors and curves}

\cite{km96} shows that
\begin{align*}
\Eff_{3}(\M) & = \langle D_{1},D_{2} \rangle \qquad \qquad \qquad \Nef_{3}(\M) = \langle D_{1}+3D_{2},D_{1}+D_{2}\rangle \\
\Eff_{1}(\M) & = \langle 3C_{1}-C_{2},-C_{1}+C_{2} \rangle \qquad  \Nef_{1}(\M) = \langle C_{1},C_{2}\rangle
\end{align*}

\cite{km96} shows that $-K_{\M} = \frac{2}{3}D_{1}$.  Note that $D_{1}$ is not movable by Lemma \ref{lem:movint}.\ref{lem:movintnef} since $D_{1}(D_{1}+D_{2})^{2} = -1911C_{1} + 7735C_{2}$ is not nef.  %In particular, since $N^{1}(\M)$ is two-dimensional, every ray of the movable cone of $\M$ contains a class of the form $K_{X} + A$ for an ample divisor $A$.  Thus $\M$ is a Mori Dream Space.

The nef divisor $D_{1} + 3D_{2}$ is semiample.  It lies on one boundary of the movable cone and has $D_{2}$ as its augmented base locus.  It defines the morphism $\phi$ on the right hand side of the following diagram:

\begin{equation*}\begin{CD}
\overline{M}_{0,7}    @>>>   \M \\
 @VVV  	@VV\phi V \\
\overline{M}_{0,\mathcal{A}} @>>> \overline{M}_{0,\mathcal{A}}^{S_{7}}
\end{CD}\end{equation*}

Here $\overline{M}_{0,\mathcal{A}} = (\mathbb{P}^{1})^{7}//SL_{2}$ is the moduli space of stable weighted pointed curves defined by assigning the weight $1/3$ to each point as in \cite{hassett03}.  The quotient from $\overline{M}_{0,7}$ to $\overline{M}_{0,\mathcal{A}}$ is $S_{7}$-equivariant and descends to define the map $\phi$.

The nef divisor $D_{1} + D_{2}$ defines a flipping contraction.  The image is identified by \cite[Proposition 4.3]{giansiracusa13} as the $S_{7}$-symmetrization of a GIT quotient of the parameter space of seven-pointed conics.

\subsubsection{Surfaces}

We set the classes $S_{1},S_{2},S_{3},S_{4}$ in $N_{2}(\M)$ to be respectively the push-forward to $\M$ of any vital surface on $\overline{M}_{0,7}$ of type $(2,2,3)$, $(3,1,3)$, $(2,3,2)$, and $(2,1,4)$.  Using the projection formula, we can calculate the classes of the $S_{i}$ using the intersection theory described in \cite{keel92}:
\begin{align*}
S_{1} & = 3T_{2}-2T_{3} \\
S_{2} &= 18T_{1} - 6T_{2} + 2T_{3} \\
S_{3} &= -9T_{1} + 3T_{2} + 5T_{3} \\
S_{4} &= 6T_{2} - 3T_{3}.
\end{align*}
Note that $S_{4} = 2S_{1} + \frac{1}{12}(S_{2} + 2S_{3})$ so that $S_{4}$ is contained in the cone generated by $S_{1},S_{2},S_{3}$.

\begin{lem}
$\langle S_{1},S_{2},S_{3} \rangle \subseteq \Eff_{2}(\M) \subseteq \langle S_{1},S_{2},S_{3},S_{3}-S_{2} + \frac{1}{4}S_{1} \rangle$.
\end{lem}

\begin{proof}
The class $(D_{1}+3D_{2})^{2}$ is nef since it is the square of a nef divisor.  It has vanishing intersection against $S_{1}$ and $S_{2}$, so it must cut out a face of $\Eff_{2}(\M)$.

The class $\alpha := D_{1} (D_{1} + 3D_{2})$ has vanishing intersection against $S_{2}$ and $S_{3}$ and positive intersection with $S_{1}$.  We check that it is a nef class, so that it cuts out a face of $\Eff_{2}(\M)$.  Since $(D_{1} + 3D_{2})$ is a nef divisor, it suffices to show that $\alpha$ has non-negative intersection against every effective surface contained in $D_{1}$.  Let $L \subset \overline{M}_{0,7}$ be a vital divisor of type $(2,5)$, so that $\pi: L \to D_{1}$ is a surjective morphism.  Note that $L \cong \overline{M}_{0,6}$.  Since classes on $\M$ are determined by their combinatorial type, it suffices to check that $\pi^{*}(D_{1} (D_{1}+3D_{2}))|_{L}$ has non-negative intersection against every effective $S_{6}$-invariant divisor on $L$.  However, \cite{km96} shows that the $S_{6}$-invariant cone of effective divisors on $\overline{M}_{0,6}$ is generated by orbits of vital cycles.  As $\alpha$ has non-negative intersection with the $S_{i}$, it is indeed a nef class.

The class $\beta := D_{2}(D_{1}+3D_{2}) + D_{1}(D_{1}+D_{2})$ has vanishing intersection against $S_{1}$ and $S_{3} - S_{2}$. We next show that $\beta$ is nef.  Since $D_{1}+3D_{2}$ and $D_{1}+D_{2}$ are nef, it suffices to show that $\beta$ has non-negative intersection against effective surfaces contained in $D_{1}$ and in $D_{2}$.  However, there are vital divisors  $L_{1} \cong \overline{M}_{0,6}$ and $L_{2} \cong \mathbb{P}^{1} \times \overline{M}_{0,5}$ that map surjectively onto $D_{1}$ and $D_{2}$ via $\pi$.  By the same argument as in the previous paragraph, we conclude by noting that $\beta$ is non-negative against the $\pi$-pushforward of every vital surface.  This shows that $\Eff_{2}(\M) \subseteq \langle S_{1},S_{2},S_{3}-S_{2} \rangle$.

Finally, we show that the class $\gamma = aS_{3} + bS_{1} + c(S_{3}-S_{2})$ for non-negative $a,b,c$ is not pseudo-effective if $c > 4b$.  Let $\psi: \overline{M}_{0,7} \to \overline{M}_{0,6}^{S_{6}}$ be the composition of the morphism that forgets the first marked point and the quotient by $S_{6}$.  We let $B_{(4,2)}$ and $B_{(3,3)}$ denote the divisors on $\overline{M}_{0,6}^{S_{6}}$ which are the pushforward of a vital divisor on $\overline{M}_{0,6}$ of type $(4,2)$ or $(3,3)$ respectively.  \cite{km96} shows that $B_{(4,2)}$ and $B_{(3,3)}$ generate the pseudo-effective cone.

Since $\pi^{*}$ preserves pseudo-effectiveness, we see that $\gamma$ can be pseudo-effective only if $\psi_{*}\pi^{*}\gamma$ is pseudo-effective.  But $\psi_{*}\pi^{*}\gamma = 6!(4(a+c)B_{(4,2)} + (4b-c)B_{(3,3)})$ is only pseudo-effective if $c \leq 4b$.
\end{proof}

We next give a partial description of $\Mov_{2}(\M)$.  Lemma \ref{lem:movint}.\ref{lem:psefintmov} and Lemma \ref{lem:movint}.\ref{lem:movintnef} allow us to bound the movable cone using intersection theory.  The results are demonstrated in the following picture which shows a ``barycentric cross-section'' of $\Eff_{2}(\M)$.  The movable cone must be contained in the shaded region; the dotted lines reflect the fact that we have not identified $\Eff_{2}(\M)$ precisely.

\begin{figure}[h] 
\includegraphics[scale=0.6]{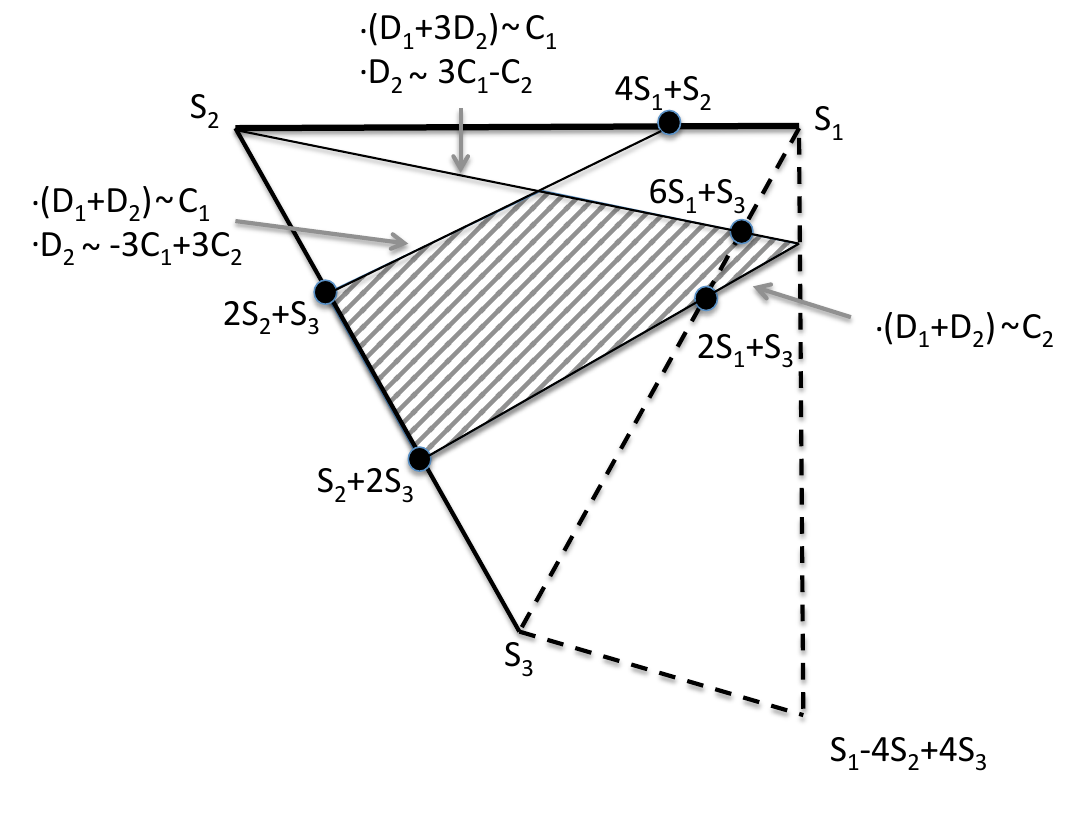}
\caption{A cross-section of $\Eff_{2}(\M)$}
\label{fig:figure1}
\end{figure}

Note that $S_{1}$, $S_{2}$, and $S_{3}$ are all non-movable extremal rays.  Thus each is its own negative part.

We can now analyze Zariski decompositions in certain regions of the cone.  To demonstrate the method, we focus only on the cone $\mathcal{R}$ generated by $S_{1}$, $S_{2}$, and the class $\gamma = 12S_{1} + 7S_{2} + 2S_{3}$.  (Note that $\gamma$ generates an extremal ray of the shaded region in the picture.)  The geometry of this portion of the cone is controlled by the morphism $\phi: \M \to \overline{M}_{0,\mathcal{A}}^{S_{7}}$ described above.  This morphism contracts both $S_{1}$ and $S_{2}$; thus $\mathcal{R}$ consists of classes of the form $\phi^{*}A^{2} + E$ where $A$ is an ample divisor and $E$ is effective $\phi$-exceptional.

\begin{claim}
For $\alpha \in \mathcal{R}$ let $\alpha = a_{1}S_{1} + a_{2}S_{2} + a_{3}\gamma$ be the unique such expression.  Let $P(\gamma)$ be the positive part of a Zariski decomposition for $\gamma$.  Then $a_{3}P(\gamma)$ is a positive part for $\alpha$.
\end{claim}

\begin{proof}
It suffices to show that $\mob(a_{3}\gamma) = \mob(\alpha)$. Let $P$ be a positive part for $\alpha$.  Since $P \in \Mov_{2}(\M)$, Figure \ref{fig:figure1} shows that there are positive constants $b_{1}, b_{2}$ such that $P + b_{1}S_{1} + b_{2}S_{2} = a_{3}\gamma$.  Thus $\mob(P) \leq \mob(a_{3}\gamma) \leq \mob(\alpha) = \mob(P)$ and the conclusion follows.
\end{proof}

To conclude the description of these Zariski decompositions, we would need to identify a Zariski decomposition for $\gamma$.  Note that $\gamma = \frac{2}{231}(D_{1}+3D_{2})^{2} + \frac{15}{11}S_{2}$.  Thus identifying a positive part for $\gamma$ is equivalent to identifying a movable transform for the square of an ample divisor over the map $\phi$.

\begin{conj}
The decomposition $P = \frac{2}{231}(D_{1}+3D_{2})^{2}$ and $N= \frac{15}{11}S_{2}$ is a Zariski decomposition for $\gamma$.
\end{conj}

\noindent This would follow from Conjecture \ref{conj:bpfpullback}.  It would also follow from an affirmative answer to \cite[Question 7.1]{lehmann13} which predicts the value of the mobility for classes of the form $H^{n-k}$ where $H$ is a nef divisor.

\section{Comparison against other notions} \label{comparisonsection}

We now compare Definition \ref{zariskidecompdefn} against several other options.  

\subsubsection{Nakayama-Zariski decomposition} \label{nakayamadecompforcycles}

In \cite[Remark after III.2.13 Corollary]{nakayama04}, Nakayama explains how his arguments for divisors can be extended to define a $\sigma$-decomposition $\alpha = P_{\sigma}(\alpha) + N_{\sigma}(\alpha)$ for cycle classes of any codimension.  Some of the main features of this decomposition are:
\begin{enumerate}
\item For any effective $\mathbb{R}$-cycle $Z$ of class $\alpha$ we have $N_{\sigma}(\alpha) \leq Z$.
\item The coefficients of $N_{\sigma}(\alpha)$ vary continuously as $\alpha$ varies in the big cone.
\end{enumerate}
Note that (1) and (2) together imply that any movable class $\alpha$ must have $N_{\sigma}(\alpha)=0$.   In particular, we see:

\begin{prop}
Let $X$ be a smooth projective variety and let $\alpha \in N_{k}(X)$ be a big class.  Then for any Zariski decomposition $\alpha = P(\alpha) + N(\alpha)$ we have that $N(\alpha) \succeq N_{\sigma}(\alpha)$.
\end{prop}

\begin{proof}
Replicating the proof of \cite[III.1.14 Proposition]{nakayama04} for cycles, one sees that the $\sigma$-decomposition has the following property: if $Z$ is an effective $\mathbb{R}$-cycle such that $\alpha - [Z]$ is big, then
\begin{equation*}
N_{\sigma}(\alpha) \preceq Z + N_{\sigma}(\alpha - [Z]).
\end{equation*}
Let $\gamma$ be a big class and let $\epsilon >0$ be sufficiently small so that $P(\alpha) - \epsilon \gamma$ is big.  Since there is an effective $\mathbb{R}$-cycle representing $N(\alpha) + \epsilon \gamma$, the above inequality shows that for any sufficiently small $\epsilon > 0$
\begin{equation*}
N_{\sigma}(\alpha) \preceq N(\alpha) + \epsilon \gamma + N_{\sigma}(\alpha - N(\alpha) - \epsilon \gamma).
\end{equation*}
As $\epsilon$ tends to $0$, we see that $N_{\sigma}(\alpha - N(\alpha) - \epsilon \gamma)$ tends to $N_{\sigma}(P(\alpha)) = 0$ by the continuity of the coefficients of $N_{\sigma}$.   Taking a limit as $\epsilon$ tends to $0$ concludes the proof. 
\end{proof}

However, the $\sigma$-decomposition is not the same as the Zariski decomposition.  Usually the $\sigma$-decomposition removes much less; $P_{\sigma}(\alpha)$ is often not movable.

\begin{exmple}
Let $X$ be the blow-up of a smooth threefold at a point.  Let $E$ denote the exceptional divisor on $X$ and let $L$ be a line in $E$.  Suppose that $C$ lies in the interior of $\Mov_{1}(X)$.  For sufficiently small $t$ the big class $\alpha = [L] + t[C]$ is not movable since it has negative intersection with $E$.  However since $L$ deforms to cover a surface we have $N_{\sigma}(\alpha) = 0$.
\end{exmple}

\subsubsection{BCK decomposition}
Suppose we have a $\mathbb{Q}$-vector space $V$ with an inner product $\langle -,-\rangle$ and a basis $\{v_{i}\}_{i=1}^{r}$ of $V$ satisfying
\begin{equation*}
\langle v_{i},v_{j} \rangle \geq 0 \textrm{ for }i \neq j.
\end{equation*}
\cite{bck12} then describes a linear algebra theory of Zariski decompositions for vectors in $V$.  For effective curves on surfaces, one recovers the Zariski decomposition of a curve $C$ by setting the $v_{i}$ to be the components of $C$.

In principle this decomposition can sometimes be applied to half-dimensional classes $N_{k}(X)$ on a smooth variety of dimension $2k$.  However, since the components of a $k$-cycle $Z$ may no longer satisfy the condition $Z_{i} \cdot Z_{j} \geq 0$ for $i \neq j$ there does not seem to be a canonical way of applying the theory of \cite{bck12}.

Even when every pair of different components of the cycle $Z$ intersects positively, there may be no relationship between the two decompositions.  This is perhaps unsurprising, as the decomposition of \cite{bck12} is based on the ``nefness'' of the components of $Z$ and not their ``mobility''.  The following example gives an extreme case.

\begin{exmple}
Let $X$ be the projective bundle over $\mathbb P^1$ corresponding
to the vector bundle 
$$E=\mathcal O(-2)\oplus\mathcal O\oplus\mathcal O(1)^{\oplus 2}.$$ 
With the notation of \S\ref{ssec:projbundles}, the class of the irreducible surface $Z=\mathbb P(\mathcal O(-2)\oplus\mathcal O(1))$ on $X$ is $\xi^2-\xi f$, which is movable by Proposition \ref{prop:proj bundles movable cone}. Furthermore, $[Z]^2=-2$.
 
Since $Z$ is movable, $P([Z]) = [Z]$ is a Zariski decomposition. On the other hand, $Z$ is its own negative part in the BCK decomposition.
\end{exmple}

\subsection{Alternative approaches}

Finally, we briefly discuss two other possible approaches to Zariski definitions.

\subsubsection{Minimal length vectors} \label{minlengthsection}

Fix an ample divisor $A$ on a smooth projective variety $X$.  By Corollary \ref{cor:norms} one can define a ``geometric norm'' $\Vert - \Vert$ on the space $N_{k}(X)$ such that for any class $\alpha \in \Eff_{k}(X)$ we have
\begin{equation*}
\Vert \alpha \Vert = A^{k} \cdot \alpha.
\end{equation*}
Using the directed property of the negative parts of $\sigma$-decompositions, we see that for a pseudo-effective divisor $L$ the class $[N_{\sigma}(L)]$ coincides with the element in $\{ \alpha \in \Eff^{1}(X) | [L] - \alpha \in \Mov^{1}(X) \}$ that has minimal length under a geometric norm.  It would be interesting to see if in other codimensions one can interpret the Zariski decomposition using minimal length vectors.  (Since the directed property no longer holds in higher codimensions, it would probably be necessary to change the norm depending on the class $\alpha$.  One would then like to have an explicit construction of the norm depending on the birational properties of $\alpha$.)

It might be impossible to construct every Zariski decomposition using this approach.  However, this technique can still be useful for identifying a ``best choice'' of a Zariski decomposition.  See Remark \ref{rem:pushforwardforcurves} for an application of this idea. %For example, using our definition a pseudo-effective divisor may admit infinitely many negative parts by Example \ref{nakayamaexample}.  However, for any polarization the

\subsubsection{Birational definition}

For a big divisor $L$, $P_{\sigma}(L)$ can be interpreted using Fujita approximations.  More precisely, if for positive integers $t$ we let $B_{t}$ denote the pushforward of the ample part of a $\frac{1}{t}$-Fujita approximation, then $P_{\sigma}(L)$ is the limit of the $B_{t}$ as $t$ increases.  If one can construct Fujita approximations for arbitrary classes as discussed in Section \ref{birmapsection}, then it seems likely that one could use these to provide an alternative definition of Zariski decompositions.

\nocite{*}
\bibliographystyle{amsalpha}
\bibliography{zardecomcycles}

\end{document}